\documentclass[10pt]{article}

\usepackage[plainpages=false,pdfpagelabels,colorlinks=true,linkcolor=mypurple,urlcolor=mypurple,citecolor=mypurple]{hyperref}

\newtoggle{vaos}
\togglefalse{vaos}
\newcommand{\aosversion}[2]{\iftoggle{vaos}{#1}{#2}}

\usepackage{tptstyle}

\usepackage{fullpage}
\usepackage[font={footnotesize,it}]{caption}

\title{Unveiling Invariant and Transferable Latent Factors Across Heterogeneous Environments via ATLAS}
\author{Yihong Gu ~~~~ Katherine Liao ~~~~ Tianxi Cai\\Harvard University}
\date{}

\begin{document}

\begin{titlepage}
\clearpage\maketitle

\begin{abstract}
This paper considers a multi-environment factor model in which high-dimensional covariates are collected from heterogeneous environments, with auxiliary labels available in a subset of these environments. The joint distribution of the covariates may vary across environments, whereas the latent structure is decomposed into invariant factors with shared loadings and heterogeneous factors with environment-specific loadings. Such a model is motivated by transfer learning and latent factor regression, where one seeks stable low-dimensional representations for both interpretation and robust out-of-sample prediction of the response $Y$. Leveraging the invariance principle, we show that the invariant and heterogeneous factors are disentangled under a minimal structural condition. Based on this, we propose ATLAS, an Auxiliary-label and invariance-guided Transfer via Latent Alignment across heterogeneous environmentS. ATLAS is a unified procedure that leverages the invariance principle to separate aligned invariant and unaligned heterogeneous factors, and further exploits supervision from auxiliary labels to extract prediction-invariant and transferable factors from those unaligned heterogeneous factors. ATLAS yields near-oracle performance for downstream latent factor regression, enables transferable prediction in new environments through the full latent signal when auxiliary labels are available, and reduces to robust invariant-factor-only prediction otherwise. We establish sharp non-asymptotic error bounds for recovering invariant and heterogeneous factors, identifying all the response-invariant factors, and estimating the invariant signal in $Y$.
\end{abstract}
\thispagestyle{empty}
\end{titlepage}

\newpage

\section{Introduction}

A central challenge in deploying predictive models across heterogeneous environments is {\em transportability}: ensuring that a model trained in one setting performs reliably in another \citep{reps2022learning,lasko2024probabilistic}. This is especially challenging when the shift involves not only the covariate distribution but also the structure of the feature space itself: different environments may encode the same underlying clinical state in systematically different ways. In healthcare,  clinical prediction models trained at one institution often degrade when deployed elsewhere or later in time \citep{young2022empirical,xiong2026adversarial}. Beyond population differences, structural variation in clinical practices, electronic health record (EHR) systems, and administrative coding schemes (e.g., ICD-9 vs ICD-10) creates representation drift across sites and time \citep{torab2023interoperability,apathy2022early,molloy2023assessing}. As a result, two patients with similar trajectories from two different environments may have very different feature vectors, making the {\em feature-to-state mapping}, not just the state distribution, a key source of cross-environment heterogeneity.

Latent factor models \citep{forni2000generalized,bai2003inferential,stock2002forecasting} provide low-dimensional representations for high-dimensional features, supporting effective interpretation, dimension reduction, and prediction. In multi-source settings, a common strategy is to decompose multi-environment data into shared and source-specific latent components by imposing a common loading matrix (i.e., shared feature embeddings) across environments \citep{yang2025estimating,lock2013joint}. This implicitly assumes a shared semantic space and a common mapping from observed features to latent states. However, in the setting above, feature embeddings themselves can vary by environment due to institutional conventions or temporal drift. Forcing a common loading matrix can therefore mix truly invariant signals with source-specific artifacts. This motivates a more flexible framework that allows the representation map to be partly environment-specific while still recovering latent factors whose semantic meaning remains stable across sources. 

This motivates a latent-factor view of transportability: although the observed clinical states may be encoded differently across environments, source-specific maps may recover latent factors with shared semantic meaning, whose relationship to the outcome remains invariant after alignment. The main challenge is to learn these maps without relying primarily on outcome labels, which often require expert annotation and are scarce in multi-site EHR studies. We therefore leverage abundant unlabeled covariates whose cross-environment covariance structure contains information about both shared latent structure and source-specific encoding mechanisms. Our key insight is that representation heterogeneity can be informative: variation in source-specific loading patterns helps distinguish environment-specific factors from invariant factors with shared loadings. This leads to the central question of the paper:
\begin{quote}
\it
\begin{center}
How can we uncover invariant, transferable latent signals across heterogeneous environments?
\end{center}
\end{quote}
We answer this question through ATLAS ({\bf A}uxiliary-label and invariance-guided {\bf T}ransfer framework via {\bf L}atent {\bf A}lignment across heterogeneous environment{\bf S}), designed to enable transportable prediction from partially labeled multi-environment data.

\subsection{The problem under study}
We study latent factor regression across heterogeneous environments $\mathcal{E}$ with partially observed labels. For each $e \in \mathcal{E}$, high-dimensional covariates $X\supe \in \mathbb{R}^d$ are observed in every environment, while auxiliary labels $Z\supe \in \mathbb{R}^q$ and target responses $Y\supe$ are available only in subsets of environments, $\mathcal{E}_z \subseteq \mathcal{E}$ and $\mathcal{E}_y \subseteq \mathcal{E}_z$, respectively. Concretely, we observe $2n_x$ i.i.d samples of covariates $\mathcal{D}_x\supe = \{X_i\supe\}_{i=1}^{2n_x}$ for all $e\in \Ecal$, $n_z$ i.i.d. observations $\mathcal{D}_z\supe = \{(X_{i+2n_x}\supe, Z_{i+2n_x}\supe)\}_{i=1}^{n_z}$ in $\Ecal_z$, and an additional $n_y$ i.i.d samples $\mathcal{D}_y\supe = \{(X_{i+2n_x+n_z}\supe, Y_{i+2n_x+n_z}\supe)\}_{i=1}^{n_y}$ in $\mathcal{E}_y$. In the motivating regime, $n_x \gg n_z \gg n_y$. All the samples $\{\mathcal{D}_x^{(e)}\}_{e\in \mathcal{E}}, \{\mathcal{D}_z^{(e)}\}_{e\in \mathcal{E}_z}, \{\mathcal{D}_y^{(e)}\}_{e\in \mathcal{E}_y}$ are independent. We assume equal sample sizes across environments for notational ease, but this is not required by the methods or theory.

ATLAS leverages the abundant covariates to disentangle invariant from environment-specific latent structure, uses auxiliary labels to identify response-relevant transferable components, and estimates a predictor for $Y$ that can be transported to new environments.

\myparagraph{Multi-environment linear factor model.} For each environment $e\in \mathcal{E}$, $X\supe$ is explained by $r\supe$ latent factors $F\supe = (F_1\supe,\ldots, F_{r\supe}\supe)^\top$ and idiosyncratic errors $U\supe \in \mathbb{R}^d$, i.e., $X\supe$ admits the following decomposition:
\begin{align}
\label{eq:model-x}
    X\supe = \underbrace{\begin{bmatrix}
        B & A\supe
    \end{bmatrix}}_{B^{(e)}:=} \cdot \begin{bmatrix}F_{\tti}\supe \\ F_{\tth}\supe \end{bmatrix} + U\supe ~~\text{with}~~\mathbb{E}[F^{(e)} (U^{(e)})^\top] = 0.
\end{align} Here $F_{\tti}\supe = (F_1\supe,\ldots, F_{r_\tti}\supe)^\top \in \mathbb{R}^{r_\tti}$ represents the unobserved \emph{invariant factors} with an unknown common loading matrix $B\in \mathbb{R}^{d\times r_{\tti}}$ shared across $\mathcal{E}$; $F_{\tth}\supe = (F_{r_\tti+1}\supe,\ldots, F_{r\supe}\supe)^\top \in \mathbb{R}^{r\supe_\tth}$ with $r_\tth\supe := r\supe - r_\tti$ represents the unobserved \emph{heterogeneous factors} with environment-specific loadings $A\supe \in \mathbb{R}^{d\times r\supe_\tth}$; $U\supe \in \mathbb{R}^d$ represents the idiosyncratic errors. Throughout, we allow the joint distribution of $(F_\tti\supe, F_\tth\supe, U\supe)$ to vary with $e \in \mathcal{E}$.

\myparagraph{Invariant latent factor regression.} The decomposition in \eqref{eq:model-x} supports transportable prediction only through latent components whose relationship with the response $Y\supe$ is invariant across environments. The invariant factors $F_\tti\supe$  are naturally aligned through the shared loading matrix $B$ and therefore provide a stable predictive signal. However, relying only on $F_\tti\supe$ can be overly conservative: some heterogeneous factors may also predict $Y\supe$ through invariant effects, but they must first be aligned across environments before they can be transferred. Thus, properly aligning the prediction-invariant part of  $F_{\tth}\supe$ can expand the transferable signal and improve transportability. We model this stable response mechanism by the following generalized linear model (GLM) with invariant parameter $\beta^\star = ((\beta_\tti^\star)^\top, (\beta_\tth^\star)^\top)^\top $\footnote{Here we assume a GLM model for ease of presentation. One can extend both the method and theory to non-parametric models using the FAR-NN estimator in \cite{fan2024factor}; see \cref{sec:discuss}. }: 
\begin{align}
\label{eq:model-invariant-y}
    Y\supe = \sigma_y \left( (\beta_{\tti}^\star)^\top F_{S_\tti^\star}\supe + (\beta_{\tth}^\star)^\top F_{S_\tth^\star}\supe\right) + \varepsilon_y\supe \qquad \text{with} \qquad \mathbb{E}\left[\varepsilon_y\supe \mid F_{\tti}\supe, F_{S_\tth^\star}\supe\right] = 0.
\end{align}
Here $\sigma_y: \mathbb{R}\to \mathbb{R}$ is a known mean function, $S_\tti^\star \subseteq [r_\tti]$ represents the response-relevant invariant factors, and $F_{S_\tth^\star}\supe$ denotes the prediction-invariant heterogeneous factors with $S_\tth^\star \subseteq [\min_{e\in \mathcal{E}} r^{(e)}] \setminus [r_\tti]$ in environment $e$. The coefficients  $\beta_{\tti}^\star$ and $\beta_{\tth}^\star$ are shared across environments, so once $F_{S_\tth^*}\supe$ is aligned, its contribution to prediction can also be transported. As in the literature on invariant and quasi-causal variable selection \citep{peters2016causal, fan2024environment}, the exogeneity condition is imposed only on the union of invariant factors and truly relevant latent factors. In particular, we do not assume full covariate exogeneity, i.e., $\mathbb{E}\big[\varepsilon_y\supe \mid F\supe\big] \neq 0$; other heterogeneous factors may still carry spurious and environment-dependent predictive power for $Y\supe$.

\myparagraph{Factor selection and alignment by auxiliary labels.} To use heterogeneous factors for transportable prediction, we must identify
which components of $F_\tth\supe$ are prediction-invariant and align them across environments. This is not possible from $X\supe$ alone,
because $F_\tth\supe$ is identifiable only up to an environment-specific
invertible transformation. Consequently, the prediction-invariant component
$F_{S_\tth^\star}\supe$ cannot generally be selected, aligned, or transferred using covariates alone.
We resolve this ambiguity using auxiliary labels $Z\supe$, observed in some environments. These labels are noisy but informative proxies, designed using domain knowledge to depend on the latent factors relevant for prediction while excluding spurious environment-specific variation. 
Let $F_{S^\star}\supe := [F_{S_\tti^\star}\supe,\, F_{S_\tth^\star}\supe]$. We assume that $F_{S^\star}\supe$ is sufficient for predicting $Z\supe$, namely $Z\supe \indep \big(F_{(S^\star)^c}\supe, U\supe\big)\mid F_{S^\star}\supe$, and that the association between $F_{S^\star}\supe$ and the auxiliary-label mechanism is invariant across environments:
\begin{align}
\label{eq:model-invariant-z}
    Z_k\supe = \sigma_z \left( (\xi_k^\star)^\top F_{S^\star}\supe \right) + \varepsilon_k\supe \qquad \text{with} \qquad \mathbb{E}\left[\varepsilon_k\supe \cdot \begin{bmatrix}F\supe\\ U\supe\end{bmatrix} \right] = 0.
\end{align} Here $\sigma_z$ is a known mean function, and $\Xi^\star = [(\xi_1^\star)^\top,\ldots, (\xi_q^\star)^\top]^\top \in \mathbb{R}^{q\times |S^\star|}$ is the invariant coefficient matrix linking the selected latent factors to the auxiliary labels. This structure lets auxiliary labels provide a common coordinate system for the prediction-invariant heterogeneous factors, enabling their contribution to be transferred across environments.

\myparagraph{Estimation goal.} Our goal is threefold: (1) recover the invariant and heterogeneous latent factors from $X$ across all environments $\mathcal{E}$; (2) use the auxiliary labels $Z$ to identify and align the prediction-invariant heterogeneous factors across environments $\mathcal{E}_z$; and (3) estimate an invariant predictor for $Y$ using data from $\mathcal{E}_y$ and transfer the prediction to a new environment\footnote{Note any new environment $e$ can be incorporated into $\mathcal{E}$ (and into $\mathcal{E}_z$ when $Z$ is observed); post-fit adaptation without rerunning ATLAS is discussed in Section 6.} with $Z$, i.e., $e\in \mathcal{E}_z \setminus \mathcal{E}_y$, through the full latent signal, and without $Z$, i.e., $e\in \mathcal{E} \setminus \mathcal{E}_z$, through the invariant latent signal alone. 

\subsection{New contributions}
We introduce a statistical framework for a multi-environment factor model and latent factor regression. We offer a thorough and minimal identification characterization in linear models, with key ideas that naturally extend beyond linearity and provide a foundation for nonlinear extensions. 
\begin{itemize}
\item We extend the invariance principle to multi-environment latent factors in a fully unsupervised setting, where no auxiliary labels are available. Under the maximum invariant subspace assumption, analogous to the maximum invariant set \citep{fan2024environment, gu2024causality} assumption for identifying invariant and quasi-causal variables, together with the standard block-uncorrelatedness in factor models, i.e., 
\begin{align}
\label{eq:ident-intro}
\overbrace{\bigcap_{e\in\mathcal E}\mathrm{col}\big([B,A\supe]\big)=\mathrm{col}(B)}^{\text{maximum invariant subspace}} \qquad \text{and} \qquad \overbrace{\forall e\in \mathcal{E}, ~\mathbb{E}\!\left[F_{\tti}\supe(F_{\tth}\supe)^\top\right]=0}^{\text{block uncorrelatedness}},
\end{align}
we can identify the invariant factors $F_\tti$ up to a common ambiguity matrix across environments, and the heterogeneous factors $F_\tth$ up to environment-specific ambiguity matrices. Similar to variable-level invariant prediction \citep{fan2024environment, gu2024causality, gu2025fundamental}, this identification only requires a fixed number of environments independent of the latent dimension; in particular, two environments already suffice when the observed heterogeneity is exhaustive. We do not impose any distribution matching assumptions. Furthermore, the assumption  \eqref{eq:ident-intro} is minimal at the instance level in that including any single distribution violating this assumption will make it unidentifiable.
\item We show how auxiliary labels can be used to align the latent factors $F_{S_\tth^\star}$ that predict $Y$ invariantly, even when they induce heterogeneous effects on the covariates. This alleviates the predictive conservativeness of a fully unsupervised invariance procedure, which can recover stable but potentially incomplete predictive signals.
\item We provide robust transfer learning guarantees characterizing how the invariant signals based on the full factor $F_{S^\star}$ or the partial factor $F_{S_\tti^\star}$ minimize certain worst-case out-of-sample risks in \cref{prop:robust1} and \cref{prop:robust2}, respectively. As illustrated in \cref{prop:robust2}, without auxiliary labels $Z$ to align $F_{S_\tth^\star}$, prediction based only on $F_{S_\tti^\star}$ is the unique worst-case optimal prediction in a new environment. This highlights the necessity of auxiliary labels for exploiting transferable heterogeneous factors.
\end{itemize}

Based on the above statistical modeling of the multi-environment factor model, we propose ATLAS, a computationally efficient estimator for learning disentangled factors and aligning transferable latent signals across environments. Methodologically, ATLAS builds on and extends the idea of diversified projection (DP) \citep{fan2022learning}, which uses linear maps to extract factors and is flexible for downstream prediction, to a multi-environment setting. ATLAS first uses a computation-efficient spectral method, referred to as invariance-heterogeneity decomposition (IHD), to learn the diversified projection matrices that map $X$ to aligned invariant factors and unaligned heterogeneous factors. Building on the first-stage DP matrices, it further constructs diversified projection maps from the invariant and heterogeneous factors to those prediction-invariant factors for predicting $Y$. 

We establish a sharp non-asymptotic analysis for all stages of ATLAS. For the IHD estimator, we establish $\ell_\infty$-like dimension-free recovery guarantees for the invariant and heterogeneous factors under the learned diversified projections. A key technical novelty is that, unlike in standard factor models where consistent estimation of the loading subspace suffices for optimal factor recovery error, the loading estimation error is no longer negligible here because one needs to separate two latent blocks. We develop careful first-order approximations for both the intermediate PCA and partialling-out solutions and show that the loading estimation error enters the final factor estimation error quadratically. Building on this analysis, we derive non-asymptotic bounds for recovering the prediction-invariant factors $F_{S_\tti^\star}$ and $F_{S_\tth^\star}$ using auxiliary labels, as well as for estimating the invariant predictive signal for $Y$. The resulting rates make explicit how the number of environments, auxiliary labels, and latent dimensions jointly determine statistical efficiency, and how auxiliary labels further reduce variance beyond aiding identification.

\subsection{Related work and comparisons}

The linear factor model \citep{forni2000generalized,bai2003inferential,hallin2007determining} is a canonical framework for modeling dependence among high-dimensional covariates by explaining their correlated variation through a linear combination of latent factors. Factor regression, or factor-augmented regression, further uses these latent factors to predict the response through linear or nonlinear models \citep{stock2002forecasting,bai2006confidence,bair2006prediction,bai2008forecasting,fan2017sufficient,fan2024factor}, with factors typically estimated by spectral methods \citep{stock2002forecasting,bai2003inferential,agarwal2012noisy,fan2013large}. Our non-asymptotic analysis covers weak factor models \citep{bai2023approximate,jiang2023revisiting,choi2025high,fan2024can}, where the classical pervasiveness condition \citep{bai2002determining} need not hold; equivalently, the loading strength relative to the noise scale may be weaker than the usual $\sqrt{d}$ scaling. In the multi-environment setting, our formalization is closely connected to the idea of multi-study factor analysis \citep{de2019multi,de2021bayesian, grabski2023bayesian, moran2026nonlinear} in the Bayesian literature, where ``multi-study'' corresponds to our ``multi-environment'' setup. For the linear multi-study factor analysis, \cite{de2019multi, de2021bayesian} use independence between shared and environment-specific factors, but mainly focus on independent Gaussian factors and idiosyncratic errors. More importantly, as noted by \citet{chandra2025inferring}, independence alone does not fully identify the shared factors. \citet{chandra2025inferring} resolve this issue for a related model under stronger structural conditions. In contrast, our maximum invariant subspace assumption provides a sharp condition for identifying the shared loading space, and we show that it is minimal at the instance level.

There is a considerable literature extending principal component analysis (PCA) methods and factor analysis to multi-environment data. For example, Common Principal Components (CPC) \citep{flury1984common, hallin2014efficient} studies covariance matrices from multiple sources that share the same orthogonal eigenspaces, i.e., $\Sigma^{(e)} = Q\Lambda^{(e)}Q^\top$. Distributed PCA \citep{fan2019distributed} considers a similar common subspace model and proposes a distributed estimator that achieves the same statistical efficiency as if all data were stored at a single source. These works, however, focus on settings where the loading space is common across environments and therefore do not address heterogeneity in the loading matrices. Another line of work allows heterogeneous loading structures and develops PCA-type methods to recover stable \citep{ma2026optimal, baharav2025stacked, wang2025stablepca} or shared components. Among them, several methods, including Personalized PCA \citep{shi2024personalized}, AJIVE \citep{feng2018angle, yang2025estimating} ---a variant of JIVE \citep{lock2013joint, palzer2022sjive} designed for multi-view data---and Anchor PCA \citep{seiter2026anchor}, though motivated differently, share a similar geometric principle to identify the shared loading space: they identify a shared subspace through the intersection of the loading spaces across different environments. This is closely related to our maximum invariant subspace assumption, i.e., the first part of \eqref{eq:ident-intro}, for identifying the column space of the shared loading matrix. Beyond the sharper non-asymptotic guarantees we obtain for estimating this shared loading space, the main conceptual difference lies in how the environment-specific component is identified; we provide a detailed comparison with a geometric illustration in \cref{remark:cp}.

Our framework is also connected to invariant representation learning in the machine learning literature. Motivated by the invariance principle for selecting quasi-causal variables from multi-environment data \citep{peters2016causal, buhlmann2020invariance, fan2024environment, gu2024causality}, a broad line of work seeks a representation $\Phi(X)$, rather than a subset of observed variables,  such that the prediction rule based on $\Phi(X)$ is stable across environments; such representations are widely used for domain adaptation and domain generalization \citep{arjovsky2019invariant, zhao2019learning, stojanov2021domain, chen2021domain, wang2022causal}. Our framework shares a similar spirit but differs in two important ways. First, instead of using a shared map from $X$ to the latent factor, we allow the representation map to be environment-specific, which is natural when there are latent variables affecting $X$ differently. Second, under exhaustive heterogeneity, our invariant factors can be identified from a fixed number of environments, whereas several related identification frameworks require the number of environments to grow with the latent dimension \citep{arjovsky2019invariant, rosenfeld2021risks, chen2022iterative}. Our use of block-uncorrelatedness between invariant and heterogeneous factors is also related to independent component analysis (ICA) and nonlinear ICA \citep{pfister2019robustifying, khemakhem2020variational, schell2023nonlinear}. These methods use side-information variables analogous to the environment indicator in our setting to identify latent factors up to permutation and componentwise transformations. Related nonlinear ICA ideas have also been used to decompose shared and environment-specific latent factors for domain generalization \citep{kong2022partial, li2023subspace}. Our work develops a finite-sample theory for the corresponding linear factor model, providing a statistical benchmark for future nonlinear extensions.

\subsection{Road map and notation}

This paper is organized as follows. In \cref{sec:ident}, we first present the identification results in the noiseless setting $U^{(e)} = 0$ to illustrate the key idea and motivate the method. We then provide our ATLAS procedure in \cref{sec:method} and present simplified non-asymptotic results in \cref{sec:theory}. We illustrate the theoretical findings and evaluate the method through simulations and a real-data application in \cref{sec:numerical}. In \cref{sec:discuss}, we conclude with a high-level discussion of several extensions of our ATLAS framework, including predicting in a new environment with partial $Z$ labels, non-parametric associations between factors and labels, and handling the setting without $Z$ labels. A rigorous analyses of these extensions are deferred to future work to keep the present paper focused. The supplemental material contains additional identification results, fully explicit non-asymptotic bounds, implementation details for the numerical studies, and all technical proofs.

\myparagraph{Notation.} Let $[m] = \{1, \ldots, m\}$, $a\lor b = \max\{a, b\}$ and $a\land b = \min\{a, b\}$. We will use the superscript $\supe$ to denote the random variables, parameters, or estimates in environment $e\in \mathcal{E}$. We use $a(n) \lesssim b(n)$, $b(n) \gtrsim a(n)$, or $a(n) = O(b(n))$ if there exists some constant $C>0$ such that $a(n) \le Cb(n)$. Denote $a(n) \asymp b(n)$ if $a(n)\lesssim b(n)$ and $a(n) \gtrsim b(n)$. We use $a(n) = o(b(n))$, $a(n) \ll b(n)$, $b(n) \gg a(n)$ if $\lim_{n\to \infty} a(n)/b(n) = 0$. For a vector $x = (x_1,\ldots, x_m)^\top\in \mathbb{R}^m$, we let $\|x\|_p=(\sum_{j=1}^m |x_j|^p)^{1/p}$ and $\|x\|_\infty = \max_{j\in [m]} |x_j|$. For given index set $S=\{j_1,\ldots, j_{|S|}\}\subseteq [m]$ with $j_1<\cdots<j_{|S|}$, we denote $[x]_S=(x_{j_1},\ldots, x_{j_{|S|}})^\top \in \mathbb{R}^{|S|}$ and abbreviate it as $x_S$ if there is no ambiguity. For matrix $M = \{M_{i,j}\}_{i\in [m_1], j\in [m_2]} \in \mathbb{R}^{m_1\times m_2}$, we use $\|M\|_2 = \sup_{\|v\|_2=1} \|Mv\|_2$ to denote the induced $\ell_2$ norm, and $[M]_{S, T} = \{M_{i,j}\}_{i\in S, j\in T}$ to denote the sub-matrix with $S \subseteq [m_1]$ and $T\subseteq [m_2]$; we also abbreviate $S$ (resp. $T$) in the subscript as ``$:$'' if $S=[m_1]$ (resp. $T=[m_2]$). Because many symbols encode distinct semantic roles, \aosversion{Appendix A}{\cref{tb:notation}} summarizes the key notation. 

For $m_1\ge m_2$, we use $\mathcal{O}_{m_1\times m_2} = \{O \in \mathbb{R}^{m_1\times m_2}: O^\top O = I_{m_2}\}$ to denote the set of matrices with orthonormal columns and abbreviate it as $\mathcal{O}_{m_1}$ if $m_1=m_2$. We use $\mathrm{rank}(M)$ to denote the rank of a matrix, $\mathrm{col}(M)$ represents the column subspace of $M$ in $\mathbb{R}^{m_1}$. For a subspace $H$ in $\mathbb{R}^m$, we will use $[H]_\perp$ to represent the orthogonal complement of the subspace in $\mathbb{R}^m$ when $\mathbb{R}^m$ is clear from context. When $O\in \mathcal{O}_{m_1\times m_2}$ with $m_1 \ge m_2$, we also use $O_\perp \in \mathcal{O}_{m_1\times (m_1-m_2)}$ to denote an orthonormal matrix such that $\col(O_\perp) = [\col(O)]_\perp$. We use $\lambda_1(M), \ldots, \lambda_m(M)$ to represents the eigenvalues of a positive semi-definite matrix $M\in \mathbb{R}^{m \times m}$ in descending order, and let $\lambda_{\max}(M) = \lambda_1(M)$, $\lambda_{\min}(M) = \lambda_m(M)$. We use $\nu_1(M), \ldots, \nu_{m_1\land m_2}(M)$ to denote the singular values of a matrix $M\in \mathbb{R}^{m_1\times m_2}$ in descending order, and let $\nu_{\max}(M) = \nu_{1}(M)$, $\nu_{\min}(M) = \nu_{m_1\land m_2}(M)$; we let $\nu_{\min}(M) = \infty$ if $m_1\land m_2 = 0$. We also use $[M_1, M_2]$ to denote the column-wise concatenation of two matrices $M_1 \in \mathbb{R}^{m_1\times m_2}$ and $M_2 \in \mathbb{R}^{m_1 \times m_3}$, and $[v_1, v_2] \in \mathbb{R}^{m_1 + m_2}$ to denote the concatenation of two vectors with dimensions $m_1$ and $m_2$. For two $W_1, W_2 \in \mathcal{O}_{d \times r}$ with $r\le d$, let $\sigma_1,\ldots, \sigma_r$ be the singular values of $W_1^\top W_2$ in descending order, and define $\sin\Theta(W_1, W_2) = \mathrm{diag}([\sin(\arccos(\sigma_1)), \ldots, \sin(\arccos(\sigma_r))]) \in \mathbb{R}^{r\times r}$.

\section{Identification and transfer learning implications}
\label{sec:ident}

In this section, we establish identification results in the noiseless setting, where the idiosyncratic error $U\supe$ is zero. This idealized setting isolates the structural assumptions under which the latent components are identifiable from $X\supe$ and from the auxiliary labels $Z\supe$. These results provide the population-level foundation for the estimation procedure in \cref{sec:method}, where the full noisy setting is handled.

\subsection{Identifying the invariant factors and prediction-invariant factors}

We first collect the conditions needed for identification. The first regularity condition assumes the latent factors are zero-mean, and imposes a nondegeneracy condition that rules out rank deficiencies in the loading matrices and factor covariances. The invariant space condition identifies the maximal loading subspace shared across environments, while the block-uncorrelatedness condition separates invariant from heterogeneous factors. Finally, \aosversion{Condition B.1}{\cref{cond:ident-z}} on auxiliary labels, stated in \aosversion{Appendix B.2}{\cref{appendix:ident-z}}, ensures that $Z\supe$ carries enough information to identify the prediction-invariant components.

\begin{condition}[Regularity condition for loadings and factors]
\label{cond:ident-reg} For any $e\in \mathcal{E}$, $\mathbb{E}[F^{(e)}] = 0$, the factor loadings have full rank, i.e., $\rank(B) = r_{\tti}$, $\rank(A\supe) = r_{\tth}\supe$, $\rank(B\supe) = r\supe$, and the covariance matrices of factors are positive definite, i.e., $\lambda_{\min}(\mathbb{E}[F\supe \{F\supe\}^\top]) > 0$.
\end{condition}

The next assumption contains the two structural requirements for separating invariant and heterogeneous latent components.

\begin{assumption}[Identification for invariant and heterogeneous factors]
\label{cond:ident} The following conditions hold:
\begin{itemize}
\item[(a)] (maximum invariant subspace) $\cap_{e\in \mathcal{E}} \col([B, A\supe]) = \col(B)$.
\item[(b)] (Block-uncorrelated factors) $\mathbb{E}[F\supe_{\tti} (F\supe_{\tth})^\top] = 0$ for each $e\in \mathcal{E}$. 
\end{itemize}
\end{assumption}

\cref{cond:ident} is fundamental in the sense that it cannot be falsified from observational data alone, akin to assumptions like unconfoundedness \citep{rubin1974estimating, rosenbaum1983central} in causal inference. 
\cref{cond:ident}(a) is needed for the identification of $\col(B)$, and is minimal. Similar conditions can be found in related literature, e.g., \cite{wang2024joint, shi2024personalized, yang2025estimating}. If it is violated, it is impossible even to identify $r_\tti$ and $\col(B)$. \cref{cond:ident}(b) follows the widely adopted convention in factor modeling that factors with different semantic meanings are independent, or at least uncorrelated; here we require only uncorrelatedness between invariant and heterogeneous components, since we only need to disentangle these two parts. 

To identify the prediction-invariant factors, we further require that the auxiliary-label model satisfy \aosversion{Condition B.1}{\cref{cond:ident-z}} stated in \aosversion{Appendix B.2}{\cref{appendix:ident-z}}. Write $\Xi^\star = [\Xi^\star_\tti, \Xi^\star_\tth]$, this condition imposes regularity on the mean function $\sigma_z(\cdot)$ and requires $\Xi_c^\star \in \mathbb{R}^{q\times|S_c^\star|}$, the concatenated linear coefficients on $F_{S_c^\star}$, to have rank $|S_c^\star|$ for $c\in\{\tti,\tth\}$. In particular, it implicitly requires $q \ge |S_\tti^\star| \lor |S_\tth^\star|$.

Together, these conditions yield the two identification statements in \cref{thm:ident-1}. In the noiseless model, \cref{cond:ident-reg} and \cref{cond:ident} imply that the invariant factors $F_{\tti}\supe$ are identifiable up to a common invertible transformation across environments, whereas the heterogeneous factors $F_{\tth}\supe$ are identifiable up to environment-specific transformations. With the additional \aosversion{Condition B.1}{\cref{cond:ident-z}} on auxiliary-label, the prediction-invariant factors $F_{S_\tti^\star}\supe$ and $F_{S_\tth^\star}\supe$ are identifiable up to common transformations across environments in $\mathcal E_z$.

\begin{theorem}[Identification]
\label{thm:ident-1}
    {\sc (Invariant and heterogeneous factors)} If $X\supe$ is generated via \eqref{eq:model-x} satisfying \cref{cond:ident-reg} and \cref{cond:ident} with $U\supe = 0$ for any $e\in \mathcal{E}$, then there exists a collection of matrices $\{\Phi\supe_{\tti}, \Phi\supe_{\tth}\}_{e\in \mathcal{E}}$ with $\Phi_{\tti}\supe \in \mathbb{R}^{d\times r_{\tti}}$ and $\Phi_{\tth}\supe \in \mathbb{R}^{d\times r_{\tth}\supe}$ that can be constructed by the covariance matrices of $\{X^{(e)}\}_{e\in \mathcal{E}}$ such that
    \begin{align}
    \label{eq:phi-ih}
        \forall e\in \mathcal{E} \qquad (\Phi\supe_{\tti})^\top X\supe = Q_{\tti} F\supe_{\tti} \qquad \text{and} \qquad (\Phi\supe_{\tth})^\top X\supe = Q_{\tth}\supe F\supe_{\tth},
    \end{align} where $Q_{\tti} \in \mathbb{R}^{r_{\tti} \times r_{\tti}}$ is the invertible transform matrix that is the same across $\mathcal{E}$, and $Q_{\tth}\supe \in \mathbb{R}^{r_{\tth}\supe \times r_{\tth}\supe}$ is the invertible transform matrix that is environment-specific. 

    \noindent {\sc (Prediction-invariant factors) } Suppose further that \aosversion{Condition B.1}{\cref{cond:ident-z}} holds for the model \eqref{eq:model-invariant-z}, then there exist matrices $\Phi_{S_\tti^\star}\in \mathbb{R}^{r_\tti \times |S_\tti^\star|}$ and $\Phi_{S_\tth^\star}\supe \in \mathbb{R}^{r_\tth\supe \times |S_\tth^\star|}$ that can be constructed by the joint laws of $\{(X^{(e)}, Z^{(e)})\}_{e\in \mathcal{E}_z}$ such that 
    \begin{align}
    \label{eq:phi-sish}
        \forall e\in \mathcal{E}_z \qquad (\Phi_{S_\tti^\star})^\top (\Phi_{\tti}\supe)^\top X\supe = Q_{S_\tti^\star} F_{S^\star_\tti}\supe ~~ \text{and} ~~ (\Phi_{S_\tth^\star}\supe)^\top (\Phi_{\tth}\supe)^\top X\supe = Q_{S_\tth^\star} F_{S^\star_\tth}\supe,
    \end{align} where $Q_{S_\tti^\star} \in \mathbb{R}^{|S_\tti^\star| \times |S_\tti^\star|}$ and $Q_{S_\tth^\star} \in \mathbb{R}^{|S_\tth^\star| \times |S_\tth^\star|}$ are the same invertible matrices across $e\in \mathcal{E}_z$.
\end{theorem}

\begin{remark}[Instance-level necessity for \cref{cond:ident}] Within the noiseless model class satisfying the regularity condition \cref{cond:ident-reg}, \cref{cond:ident} is necessary at the instance level -- for any data-generating process, the identification of latent structure is equivalent to the satisfaction of \cref{cond:ident}; see the rigorous statement in \aosversion{Theorem B.1}{\cref{thm:ident}}.
\end{remark}

\begin{remark}[Impossibility of using a common linear map] Though the invariant factors $F_\tti\supe$ produce the same effects on $X\supe$ via shared loading $B$, it is impossible to use the same linear map $\Phi_\tti$ to extract $F_\tti^{(e)}$ from $X^{(e)}$ in general; see a counterexample in \aosversion{Appendix B.1}{\cref{append:ident}}. 
\end{remark}

\subsection{Implications for robust transfer}

\myparagraph{Stable prediction across environments.} Even if directly regressing $Y\supe$ on $X\supe$ (or even on the full latent $F\supe$) yields different solutions across $\mathcal{E}_y$, the auxiliary labels $Z\supe$ provide enough supervision to extract a collection of \emph{aligned} maps $\{\Phi_{X\to S^\star}\supe\}_{e\in \mathcal{E}_z}$, where $\Phi_{X\to S^\star}\supe = [\Phi_{X\to S_\tti^\star}\supe, \Phi_{X\to S_\tth^\star}\supe]$ with $\Phi_{X\to S_\tti^\star}\supe=\Phi_{\tti}\supe \Phi_{S_\tti^\star}$ and $\Phi_{X\to S_\tth^\star}\supe=\Phi_{\tth}\supe \Phi_{S_\tth^\star}\supe$. These maps extract the prediction-invariant latent factors 
\begin{align*}
\bar{F}_{S^\star}\supe = (\Phi_{X\to S^\star}\supe)^\top X\supe= Q F_{S^\star}\supe 
\end{align*} 
on which regressing $Y\supe$ recovers the same coefficients $\bar{\beta} = Q^{-\top} \beta^\star$ across $e\in \mathcal{E}_y$, where $Q = \mbox{diag}\{Q_{S_\tti^\star}, Q_{S_\tth^\star}\}$. This identifies the invariant coefficients \eqref{eq:model-invariant-y} up to the same ambiguity matrix $Q$ across $e\in \mathcal{E}_y$ and crucially, the ambiguity cancels in the prediction signal: $(\bar{\beta})^\top \bar{F}_{S^\star}\supe = ({\beta}^\star)^\top Q^{-1} Q {F}_{S^\star}\supe  = ({\beta}^\star)^\top {F}_{S^\star}\supe$. 

\myparagraph{Transferability to new environments.} When moving to the new environment $t \in \mathcal{E}_z \setminus \mathcal{E}_y$ where only  $X^{(t)}$ and $Z^{(t)}$ are observed, the aligned maps  $\Phi_{X\to S^\star}^{(t)}$ from \cref{thm:ident-1} allow prediction of $Y^{(t)}$ via the invariant signals $\bar{\beta}^\top (\Phi_{X\to S^\star}^{(t)})^\top X^{(t)} = (\beta^\star)^\top F_{S^\star}^{(t)}$. Even without knowing the distribution of $Y \mid F, U$ in this new environment, the invariant predictor $g^\star=\sigma_y((\beta^\star)^\top F_{S^\star}^{(t)})$ is optimal in a worst-case sense: it minimizes the maximum prediction error over all plausible distributions of $Y$, where the uncertainty comes from spurious factors potentially correlating with $Y$ in unpredictable ways across environments. 

\begin{proposition}
\label{prop:robust1}
    Let $\nu$ be any fixed distribution of $(F, U)$ under which $U$ is not necessarily $0$, and $(B, A, \beta^\star, \sigma_y)$ be any fixed quantities with the same semantic meaning as before, $c>0$ be any fixed constant. Write $\theta = (\nu, B, A, \beta^\star, \sigma_y)$, and define the uncertainty set $\mathcal{U}_{\theta, c}$ as
    \begin{align*}
        \mathcal{U}_{\theta, c} := \Bigg\{\begin{split}(F, U, X, Y) \sim \mathbb{P}:&~\myblue{\mathbb{E}[Y|W]=\sigma_y(F_{S^\star}^\top\beta^\star)} ~\text{with}~ W=((F_{\tti})^\top, (F_{S_\tth^\star})^\top)^\top,  \\
        & (F,U)\sim \nu,~~ X=BF_\tti + AF_\tth + U, ~~\mathbb{E}[\myred{\mathrm{Var}(Y|W)}] \le c\end{split}\Bigg\},
    \end{align*} and worst-case out-of-sample $L_2$ risk $\mathsf{R}_{\mathtt{oos}}(g) = \sup_{\mathbb{P} \in \mathcal{U}_{\theta, c}} \mathbb{E}_{\mathbb{P}}[|g(X,F,U) - Y|^2]$. Then, under the law $(X, F, U) \sim \bar{\nu}$ induced by $\nu$ and $X=BF_\tti + AF_\tth + U$, 
    \begin{align*}
        \forall g \in L_2(\bar{\nu}), ~~~~ \mathsf{R}_{\mathtt{oos}}(g) - \mathsf{R}_{\mathtt{oos}}(g^\star) = \|g - g^\star\|_{L_2(\bar{\nu})}^2 + 2\sqrt{c} \|\Pi(g) - g\|_{L_2(\bar{\nu})}
    \end{align*} with $g^\star(x, f, u) = \sigma_y((\beta^\star)^\top f_{S^\star})$, and $\Pi(g) = \mathbb{E}[g(X,F,U)|F_{\tti}, F_{S_\tth^\star}]$. 
\end{proposition}

The joint distribution of $(F, U, X)$ is fixed, reflecting the practical setting of unsupervised learning with abundant unlabeled covariates. The perturbations arise from the potential adversarial association between $Y$ and the heterogeneous factors other than $F_{S_\tth^\star}^{(e)}$, which may explain $Y$ differently across environments $e\in \mathcal{E}$. 

\myparagraph{Generalization without alignment supervision. } When only $X^{(t)}$ is observed in the new environment $t \in \mathcal{E} \setminus \mathcal{E}_z$, \cref{thm:ident-1} provides the map $\Phi^{(t)}_{X\to S^\star_\tti} = \Phi^{(t)}_{\tti} \Phi_{S_\tti^\star}$, which extracts the invariant factors $\bar{F}_{S_\tti^\star}^{(t)} = (\Phi^{(t)}_{X\to S^\star_\tti})^\top X^{(t)} = Q_{S_\tti^\star} F_{S_\tti^\star}^{(t)}$. Without $Z$, the heterogeneous prediction-invariant factors are not identifiable under the current assumptions; \cref{prop:robust2} shows that relying solely on the invariant factors is nevertheless minimax-optimal over a rotation-based uncertainty class.

\begin{proposition}
\label{prop:robust2}
    Let $\nu$ be any fixed distribution of $\bar{F} = (F_\tti^\top, \bar{F}_\tth^\top)^\top$ such that $\mathbb{E}_\nu[\bar{F}_\tth \bar{F}_\tth^\top]=I_{r_\tth}$, $\mathbb{E}[F_\tti \bar{F}_\tth^\top] = 0$ and $F_\tti \indep \bar{F}_\tth$, $\beta^\star$ be any fixed parameter, $\sigma_y(\cdot)$ has strictly positive derivative $\sigma_y'(\cdot) > 0$, $c \in \mathbb{R}^+$. Under \aosversion{Condition B.2}{\cref{cond:fsi-worst-reg}}, denote $\mathcal{H} = \{h:  h~\text{increasing and measurable}\}$, and define the uncertainty set $\mathcal{U}_{\theta, c}$ for $\theta=(\nu, \beta^\star, \sigma_y)$ as 
    \begin{align*}
        \mathcal{U}_{\theta, c} := \Bigg\{\begin{split}&(F_\tti, F_\tth, Y) \sim \mathbb{P}: \myblue{\mathbb{E}[Y|F_{\tti}, F_{S_\tth^\star}]=\sigma_y(F_{S^\star}^\top\beta^\star)} \\
        &~~~~~~~~~ (F_\tti, \myred{OF_\tth}) = (F_\tti, \bar{F}_\tth)\sim \nu ~\text{with}~ O\in \mathcal{O}_{r_\tth}, \mathbb{E}[\myred{\mathrm{Var}(Y|F_{\tti})}] \le c\end{split}\Bigg\},
    \end{align*} and $\mathsf{R}_{\mathtt{oos}, \mathcal{H}}(\beta) := \sup_{\mathbb{P} \in \mathcal{U}_{\theta, c}} \inf_{h \in \mathcal{H}} \mathbb{E}_{\mathbb{P}}[|h(\beta^\top [F_\tti, \bar{F}_\tth]) - Y|^2]$, the worst-case calibrated out-of-sample $L_2$ risk. Let $\bar{\beta} \in \mathbb{R}^{r_\tti + r_\tth}$ with $\bar{\beta}_{S^\star_\tti} = \beta_\tti^\star$ and $\bar{\beta}_{(S_\tti^\star)^c} = 0$, then
    \begin{align}\label{eq:worst2}
        \{\bar{\beta} \cdot t: t>0\} = \argmin_{\beta} \mathsf{R}_{\mathtt{oos}, \mathcal{H}}(\beta).
    \end{align} 
\end{proposition}
\noindent \cref{prop:robust2} has two takeaways. First, the invariant signal from $F_\tti$ remains the best achievable under worst-case uncertainty, with a monotone calibration head $h$ to account for the nonlinearity of $\sigma_y(\cdot)$. Depending on the choice of $\sigma_y(\cdot)$, this corresponds to minimizing the worst-case $L_2$ risk (linear regression, $\sigma_y(t) = t$) or maximizing worst-case AUC (logistic regression $\sigma_y(t)=e^t/(e^t+1)$), since AUC is invariant under monotone transformation.

Moreover, the uncertainty in $\mathcal{U}$ has two sources: the unknown and varying distribution of $Y$ given the latent factors, and the unidentified rotation of $F_\tth$, since without $Z$ we only observe $\bar{F}_\tth = O F_\tth$ for some unknown $O \in \mathcal{O}_{r_\tth}$. This also indicates the necessity of $Z$ in yielding robust prediction of $Y$: without its supervision to identify $F_{S^\star_\tth}$ from $F_\tth$, incorporating any information from $F_\tth$ yields strictly worse predictions than just using $F_\tti$ alone.

\section{Method}
\label{sec:method}
We now describe the ATLAS estimation procedure, which builds on the idea of diversified projection (DP) \citep{fan2022learning, fan2024factor}. In a standard factor model with observed covariates $X=BF+U$, loading matrix $B\in\Rbb^{d\times r}$, latent factor $F\in\Rbb^r$, and idiosyncratic error $U$, DP constructs a matrix $\Phi\in\Rbb^{d\times \bar r}$, with $\bar r\ge r$, such that $\nu_{\min}(\Phi^\top B)\gg \nu_{\max}(\Phi)$. Hence, when the idiosyncratic errors are weakly dependent, 
\begin{align*}
\tilde F(X)=\Phi^\top X=(\Phi^\top B)F+\Phi^\top U,
\qquad
\|(\Phi^\top B)F\|_2\gg \|\Phi^\top U\|_2,
\end{align*}
so $\tilde F(X)$ can be used as a proxy for $F$ in downstream tasks such as latent factor regression. Throughout this section, we use $\tilde F$ and $\tilde F(X)$ interchangeably. Conditional on the independently estimated DP matrices, the projected factors can be treated as i.i.d. whenever the corresponding $X$ are i.i.d.
ATLAS adapts DP to the multi-environment setting in three stages. In \cref{sec:method1}, we use $\mathcal D_x\supe$ to construct separate DP matrices $\hat\Phi_\tti\supe$ and $\hat\Phi_\tth\supe$, yielding factor proxies $\tilde F_\tti\supe=(\hat\Phi_\tti\supe)^\top X\supe$ and $\tilde F_\tth\supe=(\hat\Phi_\tth\supe)^\top X\supe$ for the invariant and heterogeneous factors. In \cref{sec:method2}, we use $\mathcal D_z\supe$ to align the prediction-invariant components across environments. In \cref{sec:method3}, we use $\mathcal D_y\supe$ to estimate the invariant prediction rule and transfer it to new environments with or without auxiliary labels.

\subsection{Estimating the invariant and heterogeneous factors}
\label{sec:method1}

We first estimate the invariant loading space $\col(B)$. Throughout this subsection, for a generic object $T$, $T$ denotes an unobserved quantity, $\bar{T}$ an observed population functional, and $\hat{T}$ its empirical estimator. Under \cref{cond:ident-reg} and \cref{cond:ident}, we can write the loading $B\supe$ as 
\begin{align}
\label{eq:loading}
    B\supe := \begin{bmatrix} B & A\supe \end{bmatrix} =  \begin{bmatrix}
        W_{\tti} & W_{\tth}\supe 
    \end{bmatrix} \begin{bmatrix}
        \kii & \khie \\
         0 & \khhe
    \end{bmatrix} =: W\supe \cdot K\supe , 
\end{align} 
where $W_{\tti}\in \mathbb{R}^{d\times r_{\tti}}$ and $W_{\tth}\supe\in \mathbb{R}^{d\times r_{\tth}\supe}$ have orthonormal columns and are mutually orthogonal, $\kii \in \Rbb^{r_{\tti}\times r_{\tti}}$ and $\khhe\in \mathbb{R}^{r_{\tth}\supe\times r_{\tth}\supe}$ are invertible, while $\khie$ captures possible non-orthogonality between the shared loading $B$ and the environment-specific loading $A\supe$. 

To motivate the estimator, consider the population-level regime with known $r\supe$ and $r_\tti$. For exposition, suppose $U\supe \sim \mathcal{N}(0, I_d)$ for $d\times d$ identity matrix $I_d$; our estimator below does not rely on Gaussianity or on knowing $r_\tti$. The covariance of $X\supe$ can be written as
\begin{align}
\label{eq:method-decom-sigma-x}
    \bar{\Sigma}_X\supe = I_d + W\supe \left\{K\supe \Sigma_F\supe (K\supe)^\top \right\} (W\supe)^\top,
\end{align} 
where $\Sigma_F^{(e)}$ is the covariance matrix of $F^{(e)}$.
Therefore, for any orthonormal basis $\bar W\supe\in\mathcal{O}_{d\times r\supe}$ of the leading $r\supe$ eigenspace of $\bar{\Sigma}_X\supe$, we have $\bar W\supe=W\supe O\supe$ for some $O\supe\in\mathcal O_{r\supe}$. Averaging the projection matrices across environments gives
\begin{align*}
    \bar{P} = \frac{1}{|\mathcal{E}|} \sum_{e\in \mathcal{E}} \bar{W}\supe (\bar{W}\supe)^\top = W_{\tti} & W_{\tti}^\top + \frac{1}{|\mathcal{E}|}\sum_{e\in \mathcal{E}} W_{\tth}\supe (W_{\tth}\supe)^\top.
\end{align*} 
Under \cref{cond:ident}(a), the eigenspace of $\bar P$ associated with eigenvalue one is exactly $\col(W_{\tti})=\col(B)$, whereas all remaining eigenvalues are strictly smaller than one. Thus, the leading $r_\tti$ eigenvectors of $\bar P$ recover the invariant loading space. For each environment, the heterogeneous loading space can then be recovered from the residual projection $\bar W\supe(\bar W\supe)^\top-\bar W_\tti\bar W_\tti^\top$. Consequently, $\bar W_\tti=W_\tti O_\tti$ and $\bar W_\tth\supe=W_\tth\supe O_\tth\supe$ for orthonormal matrices $O_\tti$ and $O_\tth\supe$.

We next construct DP matrices $\bar{\Phi}_{\tti}\supe$ and $\bar{\Phi}_{\tth}\supe$ for $F_{\tti}\supe$ and $F_{\tth}\supe$, respectively. The DP matrix for $F_{\tth}\supe$ is relatively easy. Given the model \eqref{eq:model-x} and the decomposition \eqref{eq:loading} that
\begin{align*}
    (\bar{W}_{\tth}\supe)^\top X\supe 
    &= (O_{\tth}\supe)^\top \khhe F_{\tth}\supe + V^{(e)} ~~~\text{with}~~~ V^{(e)} \sim \mathcal{N}(0, I_{r_{\tth}\supe}),
\end{align*} one thus can adopt $\bar{\Phi}_{\tth}\supe = \bar{W}_{\tth}\supe/\sqrt{d}$. On the other hand, using $\bar{W}_{\tti}$ as the projection cannot screen out the effects of $F_{\tth}\supe$. Particularly, the orthogonality between $W_{\tti}$ and $W\supe_{\mathtt{H}}$ yields
\begin{align*}
    (\bar{W}_{\tti})^\top X\supe = O_{\tti}^\top \left[\kii F_{\tti}\supe + \khie F_{\tth}\supe \right] + V^{(e)} ~~~\text{with}~~~ V^{(e)} \sim \mathcal{N}(0, I_{r_{\tti}}).
\end{align*} We thus need to explicitly partial out the effects of $F_{\tth}\supe$ using the following diversified projection $\bar{\Phi}_{\tti} = [\bar{W}_{\tti} - \bar{W}_{\tth}\supe \bar{R}\supe]/\sqrt{d}$, where $\bar{R}\supe$ can be interpreted as a least-square solution
\begin{align}
\label{eq:method-partialing-out}
    \bar{R}\supe = \left[ (\bar{W}_{\tth}\supe)^\top \bar{\Sigma}_{X}\supe \bar{W}_{\tth}\supe\right]^{-1} \left[(\bar{W}_{\tth}\supe)^\top \bar{\Sigma}_{X}\supe  \bar{W}_{\tti} \right].
\end{align} 
This construction separates invariant and heterogeneous factor scores without imposing orthogonality between $B$ and $A\supe$.

\begin{algorithm}[!t]
\caption{Invariance-Heterogeneity Decomposition}
\algnewcommand{\IIf}[1]{\State\algorithmicif\ #1\ \algorithmicthen}
\algnewcommand{\EndIIf}{\unskip\ \algorithmicend\ \algorithmicif}
\begin{algorithmic}[1]
\State \textbf{Input:} $\mathcal{D}=\{X_i\supe\}_{e\in \mathcal{E}, i\in [2n_x]}$, factor dimension $\{r\supe\}_{e\in \mathcal{E}}$, hyper-parameter $\lambda_{\mathsf{ihd}} \in (0, 1)$.
\State Sample split $\mathcal{D} = \mathcal{D}_1 \cup \mathcal{D}_2$ with $\mathcal{D}_1 = \{X_i\supe\}_{e\in \mathcal{E}, i\in [n_x]}$ and $\mathcal{D}_2=\{X_i\supe\}_{e\in \mathcal{E}, i\in [2n_x] \setminus [n_x]}$. 
\State \Comment{{\sc Step 1. PCA on each $e\in \mathcal{E}$.}}
\For{$e \in \mathcal{E}$} 
    \State Calculate covariance matrix $\hat{\Sigma}\supe \gets \frac{1}{n_x} \sum_{i=1}^{n_x} X_i\supe [X_i\supe]^\top$ using data from $\mathcal{D}_1$.
    \State Compute the truncated top-$r\supe$ PCA $\hat{V}\hat{\Lambda}\hat{V}^\top$ of $\hat{\Sigma}\supe$, set $\hat{W}\supe \gets \hat{V} \in \mathbb{R}^{d\times r\supe}$.
\EndFor
\State \Comment{{\sc Step 2. Estimate Invariant Loading.}}
\State $\hat{P} \gets \frac{1}{|\mathcal{E}|} \sum_{e\in \mathcal{E}} \hat{W}\supe (\hat{W}\supe)^\top$.
\State Compute the truncated top-$\hat{r}_\tti$ PCA $\hat{V}\hat{\Lambda}\hat{V}^\top$ of $\hat{P}$ with $\hat{\Lambda}_{\hat{r}_\tti, \hat{r}_\tti} \ge 1-\lambda_{\mathsf{ihd}}$, set $\hat{W}_{\tti} \gets \hat{V} \in \mathbb{R}^{d\times \hat{r}_\tti}$.
\For{$e \in \mathcal{E}=\{e_1, \ldots, e_{|\mathcal{E}|}\}$} \Comment{{\sc Step 3. Estimate DP Matrices.}}
    \State $\hat{S}_\tth \gets \hat{W}\supe (\hat{W}\supe)^\top - \hat{W}_\tti \hat{W}_\tti^\top$, $\hat{r}_\tth\supe \gets r\supe - \hat{r}_\tti$.
    \State Compute the truncated top-$\hat{r}_\tth\supe$ PCA $\hat{V}\hat{\Lambda}\hat{V}^\top$ of $\hat{S}_\tth$, set $\hat{W}_{\tth}\supe \gets \hat{V} \in \mathbb{R}^{d\times \hat{r}_\tth\supe}$.
    \State Calculate covariance matrix $\hat{\Sigma}\supe_{\dagger} \gets \frac{1}{n_x} \sum_{i=n_x+1}^{2n_x} X_i\supe [X_i\supe]^\top$ using data from $\mathcal{D}_2$.
    \State $\hat{\Phi}\supe_{\tth} \gets \hat{W}_{\tth}\supe \left\{ (\hat{W}_{\tth}\supe)^\top \hat{\Sigma}\supe_\dagger \hat{W}_{\tth}\supe\right\}^{-1/2}$.
    \State $M \gets \hat{W}_{\tti} - \hat{W}_{\tth}\supe \left\{ (\hat{W}_{\tth}\supe)^\top \hat{\Sigma}\supe_\dagger \hat{W}_{\tth}\supe\right\}^{-1} \left\{(\hat{W}_{\tth}\supe)^\top \hat{\Sigma}\supe_\dagger \hat{W}_{\tti} \right\}$. \Comment{{\sc Partialling out $F_\tth\supe$}}
    \IIf{$e=e_1$} $G\gets M^\top \hat{\Sigma}_\dagger^{(e_1)} M$ \EndIIf 
    \State $\hat{\Phi}_\tti\supe \gets M G^{-1/2}$ \Comment{{\sc Fix a reference $e_1$ to pin down the scaling of the $F_\tti$ estimate.}}
\EndFor
\State \textbf{Output:} The diversified projection matrices $\{\hat{\Phi}_{\tti}\supe, \hat{\Phi}_{\tth}\supe\}_{e\in \mathcal{E}}$.
\end{algorithmic}
\label{algo:dp}
\end{algorithm}

\cref{algo:dp} gives the empirical implementation. For simplicity, the algorithm treats $r\supe$ as known and selects $r_\tti$ through the hyperparameter $\lambda_{\mathsf{ihd}}$; in practice, $r\supe$ may be selected by standard factor-number estimators \citep{bai2002determining, ahn2013eigenvalue}. We also explicitly standardize the factor $\tilde{F}_\tth\supe$ (line 16) for any $e\in \mathcal{E}$ and $\tilde{F}_\tti^{(e_1)}$ (line 18-21) for the first environment $e_1$ rather than the $1/\sqrt{d}$ scaling to avoid large condition numbers in their covariance matrices. 

\begin{remark}[Comparison to other multi-environment dimension reduction methods] \label{remark:cp} Our method differs from approaches that directly use the estimated shared loading space for downstream prediction. To see the distinction geometrically, consider two environments 
\begin{align*}
    X^{(1)} = b \cdot F_1^{(1)} + a^{(1)} F_2^{(1)} + U^{(1)} ,\qquad X^{(2)} = b \cdot F_1^{(2)} + a^{(2)} F_2^{(2)} + U^{(2)}, 
\end{align*} 
Let $w_1,w_2,w_3$ be orthonormal vectors in $\mathbb{R}^d$ with $d=1024$. For each environment $e\in\{1,2\}$, let $F_1^{(e)}$ and $F_2^{(e)}$ be independent and uniformly distributed on $[-1,1]$, and let $U^{(e)}\sim\mathcal N(0,I_d)$ be independent of $F^{(e)}$, and set $b=w_1$, $a^{(e)}=0.5w_1+0.9w_{1+e}$ for $e\in [2]$. 

In \cref{fig:vis}, the \myyellow{yellow} and \mygreen{green} planes represent the loading spaces $\mathrm{span}(b,\myyellow{a^{(1)}})$ and $\mathrm{span}(b,\mygreen{a^{(2)}})$, respectively, with points showing observations of $X$ from environments \myyellow{$e=1$} and \mygreen{$e=2$}. Panel (a) shows that our method estimates the shared loading direction $\hat w_\tti$ as the intersection of the two environment-specific planes, thereby recovering the shared loading subspace. In contrast, PoolPCA runs PCA on the pooled data and finds the direction that minimizes the overall projection error, which need not coincide with the shared direction. Panel (b) highlights a further distinction from methods\footnote{Examples include AJIVE \citep{feng2018angle, yang2025estimating}, Personalized PCA \citep{shi2024personalized}, Anchor PCA \citep{seiter2026anchor}, and JICO \citep{wang2024joint}. Although motivated differently, these methods share the idea of estimating the intersection of the leading principal subspaces, equivalently loading spaces in factor models, and using it directly for downstream prediction.}, exemplified by AJIVE, that use the estimated shared loading direction directly for dimension reduction. Such an approach is appropriate when the environment-specific loading space is orthogonal to the shared loading space. ATLAS instead allows $A\supe$ to be non-orthogonal to $B$ and separates invariant from heterogeneous factors using only the block-uncorrelatedness of $F_\tti$ and $F_\tth$. This matters for prediction: JICO, for example, imposes orthogonality on both loadings and factors to identify shared response signals, whereas our uncorrelatedness-based formulation is closer to standard factor-model assumptions and extends more naturally to nonlinear factor models, where uncorrelatedness can be strengthened to independence but loading-space orthogonality has no direct analogue.
\end{remark}

\begin{figure}[!t]
\begin{center}
\begin{tikzpicture}
    \node[anchor=west] at (0, 0) {\includegraphics[width=6cm]{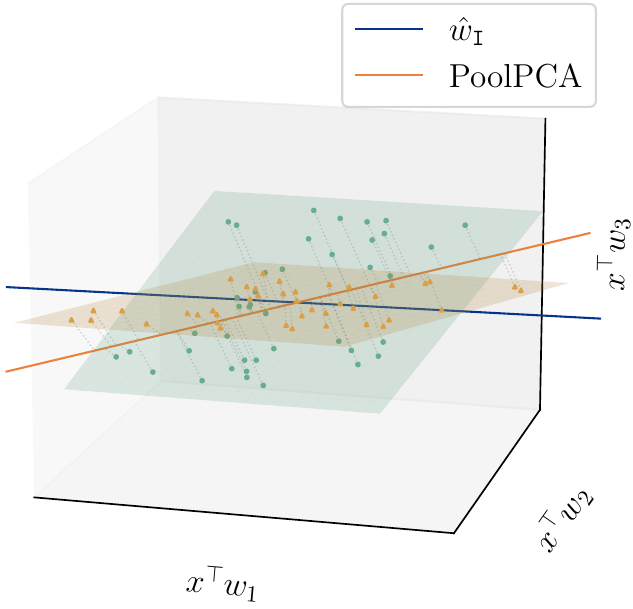}};
    \node[anchor=west] at (3.5, -2.6) {(a)};
    \node[anchor=west] at (7.5, -0.) {\includegraphics[width=6.2cm]{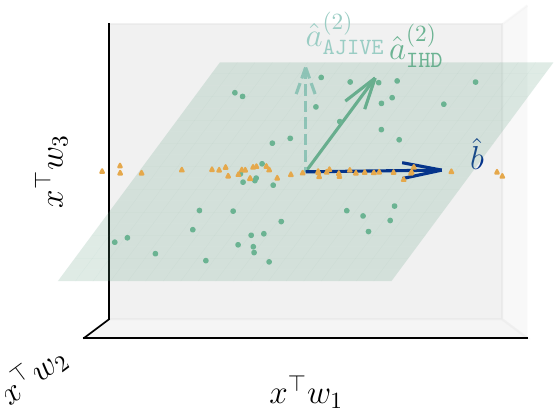}};
    \node[anchor=west] at (7.5+3.5, -2.6) {(b)};
\end{tikzpicture}
\end{center}
\caption{A visualization of the solutions yielded by our proposed invariance-heterogeneity decomposition (IHD), PCA on all the data (PoolPCA), and AJIVE. Panel (a) compares the shared direction identified by \myblue{IHD/AJIVE} with that obtained by \myorange{PoolPCA}; Panel (b) compares the environment-specific loading direction estimated by \mylightblue{IHD} with that estimated by \mygreen{AJIVE} in environment $e=2$. }
\label{fig:vis}
\end{figure}

\subsection{Estimating the prediction-invariant factors using auxiliary labels}
\label{sec:method2}

From \cref{sec:method1}, we obtain factor proxies  $\tilde{F}_\tti\supe = (\hat{\Phi}_\tti\supe)^\top X\supe$ and $\tilde{F}_\tth\supe = (\hat{\Phi}_\tth\supe)^\top X\supe$, which estimate $F_\tti\supe$ and $F_\tth\supe$ up to invertible transformations. We next use the auxiliary labels $Z$ to extract the prediction-invariant components
$F_{S_\tti^\star}\supe$ and $F_{S_\tth^\star}\supe$ from these proxies and align them across environments in $\mathcal E_z$.
A natural approach is to minimize the quasi-likelihood objective for $\Xi_\tti \in \mathbb{R}^{q\times |S_\tti^\star|}, \Phi_{S_\tti} \in \mathbb{R}^{\hat{r}_\tti \times |S_\tti^\star|}$, $\Xi_\tth \in \mathbb{R}^{q\times |S_\tth^\star|}, \Phi_{S_\tth}^{(e)} \in \mathbb{R}^{\hat{r}_\tth^{(e)} \times |S_\tth^\star|}$:
\begin{align}
\label{eq:method-mle}
    \hat{\Xi}_\tti, \hat{\Xi}_\tth,\hat{\Phi}_{S_\tti}, \{\hat{\Phi}_{S_\tth}\supe\}_{e\in \Ecal_z} = \argmin_{\substack{\Xi_\tti, \Xi_\tth, \Phi_{S_\tti}, \\  
    \{{\Phi}_{S_\tth}\supe\}_{e\in \Ecal_z}}} \sum_{\substack{e\in \mathcal{E}_z\\i-2n_x \in [n_z]}} L\left(\Xi_\tti \Phi_{S_\tti}^\top \tilde{F}_{\tti, i}\supe + \Xi_\tth (\Phi_{S_\tth}\supe)^\top \tilde{F}_{\tth, i}\supe, Z_i\supe\right).
\end{align} 
Here $L(v, z) = \sum_{k=1}^q \ell_z(v_k, z_k)$,  $\ell_z(v, z) = b_z(v) - v z$, $b_z'(\cdot) = \sigma_z(\cdot)$, and $\tilde{F}_{c, i}^{(e)} = \tilde{F}_c^{(e)}(X^{(e)}_i)$ for $c\in \{\tti, \tth\}$. Under the identification conditions, the resulting $\hat{\Phi}_{S_\tti}^\top \tilde{F}_\tti\supe$ and $(\hat{\Phi}_{S_\tth}\supe)^\top \tilde{F}_\tth\supe$ estimate $F_{S_\tti^\star}\supe$ and $F_{S_\tth^\star}\supe$ up to common transformation across $\mathcal{E}_z$, respectively. However, \eqref{eq:method-mle} is generally non-convex and computationally difficult. We therefore use a spectral approximation. 

We first fit the generalized linear model:
\begin{align}
\label{eq:est-psi}
    \hat{\Psi}_\tti, \{\hat{\Psi}_\tth\supe\}_{e\in \Ecal_z} = \argmin_{\Psi_\tti, \{\Psi_\tth\supe\}_{e\in \mathcal{E}_z}} \frac{1}{n_z \cdot |\mathcal{E}_z|}  \sum_{\substack{e\in \mathcal{E}_z\\ i-2n_x \in [n_z]}}
    L\left(\Psi_\tti \tilde{F}_{\tti,i}\supe + \Psi_\tth\supe \tilde{F}_{\tth,i}\supe, Z_i\supe\right).
\end{align} 
The coefficient $\Psi_{\tti} \in \mathbb{R}^{q \times \hat{r}_i}$ is shared across $\mathcal{E}_z$ because $\tilde F_\tti\supe$ is already aligned, whereas $\Psi_\tth\supe \in \mathbb{R}^{q \times \hat{r}_\tth^{(e)}}$ is environment-specific because $\tilde F_\tth\supe$ remains unaligned after the first stage.

For the invariant block, the right singular space of $\hat{\Psi}_\tti$ is approximately low rank, and its leading right singular subspace is close to $\col(\tilde{\Phi}_{S_\tti})$, where $\tilde{\Phi}_{S_\tti} \in \mathbb{R}^{\hat{r}_\tti \times |S_\tti^\star|}$ such that $\tilde{\Phi}_{S_\tti}^\top \tilde{F}_\tti \approx F_{S_\tti^\star}$. Let $U_{\hat{\Psi}_\tti} \Sigma_{\hat{\Psi}_\tti} V_{\hat{\Psi}_\tti}^\top$ be the truncated top-$\hat{s}_\tti$ SVD of $\hat{\Psi}_\tti$, i.e., $[\Sigma_{\hat{\Psi}_\tti}]_{\hat{s}_\tti, \hat{s}_\tti} \ge \lambda_{\mathsf{sel}}$ for some threshold hyperparameter $\lambda_{\mathsf{sel}}>0$, we use the map to extract $F_{S_\tti^\star}$ from $\tilde{F}_\tti$:
\begin{align}
\label{eq:est-si}
    \hat{\Phi}_{S_\tti} = V_{\hat{\Psi}_\tti} \in \mathbb{R}^{\hat{r}_\tti \times \hat{s}_\tti}
\end{align} 

For the heterogeneous block, the right singular spaces of $\hat\Psi_\tth\supe$ are environment-specific and must first be aligned. Let $U_{\hat{\Psi}_\tth\supe} \Sigma_{\hat{\Psi}_\tth\supe} V_{\hat{\Psi}_\tth\supe}^\top$ be the truncated top-$\hat{s}_\tth$ SVD of $\hat{\Psi}_\tth\supe$ in \eqref{eq:est-psi} with threshold $\lambda_{\mathsf{sel}}$. We estimate the column space of $\Xi_\tth^\star = [(\xi_{1,\tth}^\star)^\top, \ldots, (\xi_{q, \tth}^\star)^\top]^\top$, where $\xi_{k,\tth}^\star$ is the coefficient of $F_{S_\tth^\star}$ in the GLM on $Z_k$, by
\begin{align}
\label{eq:est-xi-h}
    \hat{\Xi}_\tth = V_{M} , \quad \mbox{where} \quad V_{M} \Lambda_M V_M^\top \text{ is top-}\hat{s}_\tth ~\text{PCA of } ~M = \frac{1}{|\mathcal{E}_z|} \sum_{e\in \mathcal{E}_z} U_{\hat{\Psi}_\tth\supe} (U_{\hat{\Psi}_\tth\supe})^\top .
\end{align} 
We then define the heterogeneous prediction-invariant projection as:
\begin{align}
\label{eq:est-sh}
    \hat{\Phi}_{S_\tth}\supe = (\hat{\Psi}_\tth\supe)^\top \hat{\Xi}_\tth \in \mathbb{R}^{\hat{r}_\tth\supe \times \hat{s}_\tth}.
\end{align} 

\subsection{Estimation of invariant signals and prediction in new environments}
\label{sec:method3}
We now estimate the prediction rule using environments in 
$\mathcal{E}_y$, where both $Y$ and $Z$ are observed. When the dimensions $|S_\tti^\star|$ and $|S_\tth^\star|$ are correctly specified by $\tilde{F}_{S_\tti}^{(e)}$ and $\tilde{F}_{S_\tth}^{(e)}$, for each $e \in \Ecal_y$, define the aligned factor representation
\begin{align*}
    \tilde{F}\supe_S(X^{(e)}) := \begin{bmatrix}
        \tilde{F}_{S_\tti}\supe \\
        \tilde{F}_{S_\tth}\supe
    \end{bmatrix} \approx \begin{bmatrix}
        Q_{S_\tti} & 0 \\
        0 & Q_{S_\tth}
    \end{bmatrix} \begin{bmatrix}
        F_{S_\tti^\star}\supe \\
        F_{S_\tth^\star}\supe
    \end{bmatrix},
\end{align*} 
Because the ambiguity matrix $Q\in\mathbb{R}^{|S^\star|\times |S^\star|}$ is shared across environments, we can pool all data in $\mathcal{E}_y$ and regress $Y\supe$ on $\tilde{F}\supe_S$ to estimate the invariant prediction parameter $\beta^\star$, up to the common transformation $Q$. Thus, we may estimate $\beta = [\beta_\tti, \beta_\tth]$
\begin{align}
\label{eq:est-y}
    \hat{\beta} = \left[\hat{\beta}_\tti, \hat{\beta}_\tth\right] = \argmin_{\beta =[\beta_\tti,\beta_\tth]} \frac{1}{|\mathcal{E}_y| \cdot n_y} \sum_{\substack{e\in \mathcal{E}_y \\ i-2n_x-n_z \in [n_y]}} \ell_y(\beta_\tti^\top \tilde{F}_{S_\tti,i}\supe + \beta_\tth^\top\tilde{F}_{S_\tth,i}\supe, Y\supe_i)
\end{align} where $\ell_y(v, y) = b_y(v) - v \cdot y$ with $b_y(\cdot)$ being such that $b_y'(\cdot) = \sigma_y(\cdot)$. Because the GLM parameter $\beta^\star$ between $Y\supe$ and $F_{S^\star}\supe$ as in \eqref{eq:model-invariant-y} and the ambiguity transform $[\tilde{F}_{S_\tti^\star}\supe, \tilde{F}_{S_\tth^\star}\supe] \approx [Q_{S_\tti} F_{S_\tti^\star}\supe, Q_{S_\tth} F_{S^\star_\tth}\supe]$ are both invariant across $\mathcal{E}_y$, we can expect $\hat{\beta}_\tti \approx Q_{S_\tti}^{-\top} \beta_\tti^\star$ and $\hat{\beta}_\tth \approx Q_{S_\tth}^{-\top} \beta_\tth^\star$. 

\myparagraph{Predictions in new environments.} For the environment $e \in \mathcal{E}_z \setminus \mathcal{E}_y$ such that one can use the auxiliary labels $Z^{(e)}$ to construct DP matrix that extract $F_{S_\tth^\star}\supe$ from $F_\tth\supe$, we can use the score 
\begin{align}
    \hat{g}^{(e)}(X^{(e)}) = \hat{\beta}^\top_\tti \tilde{F}_{S_\tti}^{(e)}(X^{(e)}) + \hat{\beta}^\top_\tth \tilde{F}_{S_\tth}^{(e)}(X^{(e)})
\end{align} to predict $Y^{(e)}$. For the environment $e \in \mathcal{E} \setminus \mathcal{E}_z$ without auxiliary labels so that $F_{S_\tth^\star}^{(e)}$ is not identifiable, we can turn to conservative predictions of the score that only use $F_{S_\tti^\star}^{(e)}$, i.e.,
\begin{align}
    \hat{g}^{(e)}(X^{(e)}) = \hat{\beta}^\top_\tti \tilde{F}_{S_\tti}^{(e)} (X^{(e)}). 
\end{align}

\section{Theory}
\label{sec:theory}
Throughout the theoretical analysis, we assume $(F\supe, U\supe)$ are zero-mean for ease of presentation, and present main results under a common signal-to-noise (SNR) ratio across environments,
\begin{align}
\label{eq:def-lambda}
    \forall e\in \mathcal{E}, \qquad C^{-1} \snrmain \le \frac{\lambda_{\min}\left[(B\supe)^\top B\supe\right]}{\|\Sigma_U\supe\|_2} \le C \snrmain,  \quad
    \mbox{for some universal constant $C$,}
\end{align} where $\Sigma_{U}\supe := \mathbb{E}[U^{(e)}(U^{(e)})^\top]$. 
The results with varying SNR ratios are deferred to \aosversion{Appendix C.1}{\cref{appendix:factor-est}}. We also use the following convention: a random vector $W \in \mathbb{R}^m$ is \emph{jointly sub-Gaussian} with parameter $s$ if $\mathbb{E} [\exp\left(v^\top W\right)] \le e^{s^2 \|v\|_2^2 / 2}$ for any $v\in \mathbb{R}^m$. This definition implies that $W$ has mean zero.  

\subsection{Single-environment benchmark for weak factor recovery}

\begin{condition}
\label{cond:x-reg}
    There exists a constant $c_1\ge 1$ such that for any $e\in \mathcal{E}$:
    \begin{itemize}[itemsep=0pt]
        \item[(a)] The observed $X\supe$ follows model \eqref{eq:model-x} with $\nu_{\min}( B^{(e)}) > 0$ and $\mathbb{E}[U\supe|F\supe] = 0$. 
        \item[(b)] The covariance matrices $\Sigma_{F}\supe$ have eigenvalues ranging from $c_1^{-1}$ to $c_1$, $F\supe$ is jointly sub-Gaussian with parameter $\sqrt{c_1}$.
        \item[(c)] $\Sigma_{U}\supe = \mathbb{E}[U^{(e)}(U^{(e)})^\top]$ is invertible and $(\Sigma_{U}\supe)^{-1/2} \cdot U\supe| F\supe$ is sub-Gaussian with parameter $\sqrt{c_1}$ almost surely.
    \end{itemize}
\end{condition}

\cref{cond:x-reg} collects standard conditions for estimating factor loading spaces from $X$ \citep{chen2021spectral} and directly implies \cref{cond:ident-reg}. Condition (a) imposes full conditional exogeneity of $U\supe$ on $F\supe$. This condition, or a stronger version like independence, is also imposed in other analyses of weak factors \citep{jiang2023revisiting, choi2025high, fan2024can}, and is slightly stronger than the marginal exogeneity in \eqref{eq:model-x}. This is imposed to get tighter non-asymptotic error bounds, whereas a consistent but relatively crude error bound can also be attained under $\mathbb{E}[F^{(e)} (U^{(e)})^\top] = 0$. Condition (b) assumes well-conditioned and sub-Gaussian latent factors, (c) allows dependence across coordinates of $U\supe$ while requiring the standardized error vector to be jointly sub-Gaussian. 

Recall $W\supe\in \mathcal{O}^{d\times r\supe}$ in \eqref{eq:loading} satisfies that $\col(W\supe) = \col(B\supe)$. We first record a single-environment benchmark for estimating the factor loading and factor. Our result is based on the prerequisite that $\col(W\supe) = \col(B\supe)$ can be consistently estimated well in each environment $e\in \mathcal{E}$.

\begin{lemma}[Single environment loading and factor estimation error]
\label{lemma:preliminary}
    Under \cref{cond:x-reg}, if $n_x\ge \tilde{C}_1 \cdot \sup_{e\in \mathcal{E}} [r^{(e)} + \log(|\mathcal{E}|\cdot n_x)]$ for some constant $\tilde{C}_1$, then with probability at least $1-n_x^{-100}$
    \begin{align}
    \label{eq:loading-error}
        \forall e\in \mathcal{E}, \qquad \left\|\sin \Theta(\hat{W}\supe, W\supe) \right\|_2 \le \tilde{C}_1 \left[\sqrt{\frac{d + \log(n_x\cdot |\mathcal{E}|)}{\snrmain \cdot n_x}} + \frac{1}{\snrmain} \right] =: \tilde{C}_1 \delta_{W} .
    \end{align} Moreover, as long as the $\delta_{W} \le 1/[\tilde{C}_1]^2$, under the same event that \eqref{eq:loading-error} holds, if we use the $\hat{\Phi}\supe = d^{-1/2} \cdot \hat{W}\supe$ as the diversified projection matrix for $F\supe$, then for any $e\in \mathcal{E}$ and $\mathcal{D}_x^{(e)} = \{X_i^{(e)}\}_{i=1}^{n_x}$,
    \begin{align}
    \label{eq:factor-error}
        \inf_{Q\supe \in \mathbb{R}^{r\supe\times r\supe}} \left\{\mathbb{E} \left[ \left\| Q\supe (\hat{\Phi}\supe)^\top X\supe - F\supe \right\|_2^2 \mid \mathcal{D}_x^{(e)}\right] \right\}^{1/2} \le \tilde{C}_1 \cdot \sqrt{\frac{r\supe}{\snrmain}}.
    \end{align}
\end{lemma}

The quantity $\delta_{W}$ in \eqref{eq:loading-error} measures the error in estimating the loading space $\col(B\supe)$. Its dependence on $(n_x, d, \snrmain)$ is optimal up to a constant depending on $c_1$, and $\delta_{W} = o(1)$ is the minimal requirement for consistent loading-space estimation \citep{yan2024inference}.
The second part of \cref{lemma:preliminary} shows that once $\col(W\supe)$ is consistently estimated, the resulting factor estimates achieve the optimal weak factor recovery rate. The lower bound of $\lambda_\star^{-1/2}$ is due to the existence of idiosyncratic error. Thus, once $n_x \gg d/\snrmain$,  increasing $n_x$ no longer improves factor recovery beyond this noise-limited rate of $(r\supe / \lambda_\star)^{1/2}$; see Lemma 1 of \cite{fan2024factor}. 

\subsection{Estimation errors for invariant and heterogeneous factors}

We next analyze the factor estimates produced by the invariance-heterogeneity decomposition. Recall from \eqref{eq:loading} that $W_\tti \in \mathcal{O}^{d\times r_\tti}$ spans the invariant loading space, i.e., $\col(W_\tti)=\col(B)$, while $W_\tth\supe\in \mathcal{O}^{d\times r_\tth\supe}$ represents the projection of the heterogeneous loadings into the subspace $(W_\tti)_\perp$, i.e., $\col(W_\tth\supe) = \col((I_d - \mathsf{P}_{B}) A\supe)$, where $\mathsf{P}_M = M (M^\top M)^{-1} M^\top$ is the projection matrix. The following \cref{cond:separate} is a quantitative analogue of \cref{cond:ident}.
\begin{condition}
\label{cond:separate}
We have $\mathbb{E}[F_\tti^{(e)} (F_\tth^{(e)})^\top] = 0$ for any $e\in\mathcal{E}$, and
\begin{align}
    \lambda_{\max}\left(\frac{1}{|\mathcal{E}|} \sum_{e\in \mathcal{E}} W\supe_{\mathtt{H}} (W_{\mathtt{H}}\supe)^\top\right) \le 1-\epsilon_A ~~\text{for some constant } \epsilon_A \in (0, 1].
\end{align}
\end{condition}
The parameter $\epsilon_A$ measures the separation between the invariant loading space and the average heterogeneous loading spaces. A larger $\epsilon_A$ corresponds to an easier problem. This condition implies \cref{cond:ident} and provides the eigen-gap that appears in the rates below. 

Denote the condition number and error rate for estimating the invariant loading space by
\[
\kappa\supe = \frac{\lambda_{\max}[(B\supe)^\top B\supe]}{\lambda_{\min}[(B\supe)^\top B\supe]} \quad \mbox{and} \quad
    \delta_{W_\tti} = \frac{1}{\epsilon_A} \left[\frac{\delta_W}{\sqrt{|\mathcal{E}|}} + \delta_W^2 + \frac{1}{\lambda_\star}\right],
\] 
respectively. The term $\delta_W/\sqrt{|\mathcal{E}|}$ reflects variance reduction from averaging loading-space estimates across environments, while $\delta_W^2 + \lambda_\star^{-1}$ is the bias from estimating each environment-specific space. When all environments share the same loading space, this reduces to the setting of \cite{fan2019distributed}. Unlike the error bounds in \cite{fan2019distributed,he2025transpca, yang2025estimating}, our bound of $\delta_{W_\tti}$ is independent of the condition number $\kappa\supe$. This follows from the first-order characterization of the SVD solution, which is a deterministic result that is of independent interest; see \aosversion{Theorem D.1}{\cref{thm:first-order-svd}}. 

\begin{proposition}
\label{prop:loading-est}
Suppose all the conditions in \cref{thm:factor-est} are satisfied. Under the choice of $\lambda$ and the same event defined therein, we have $\hat{r}_\tti = r_\tti$, and
    \begin{align}
        \left\|\sin \Theta(\hat{W}_\tti, W_\tti) \right\|_2 \le \tilde{C}_2 \delta_{W_\tti}, \qquad \left\|\sin \Theta(\hat{W}_\tth\supe, W_\tth\supe) \right\|_2 \le \tilde{C}_2 \left(\delta_{W} + \delta_{W_\tti}\right)
    \end{align}
\end{proposition}

We next state the main result on invariant and heterogeneous factor estimation errors. 

\begin{theorem}
\label{thm:factor-est}
    Suppose Conditions \ref{cond:x-reg} and \ref{cond:separate} hold. Denote $\tilde{r}\supe = r\supe + \log(|\mathcal{E}| \cdot n_x)$, $C_A = [\epsilon_A \min\{1, |\mathcal{E}| \cdot \epsilon_A\}]^{-1}$. There exist constants $\tilde{C}_1, \tilde{C}_2, \tilde{C}_3 = \poly(c_1)$ such that if 
    \begin{align}
    \label{eq:factor-est-cond-n}
        \max\left\{n_x^{-1}\sup_{e\in \mathcal{E}} \tilde{r}\supe, 2\delta_{W_\tti}/\epsilon_A, \delta_W\right\} \le \tilde{C}_1^{-1}
    \end{align}
    then by picking $\lambda_{\mathsf{ihd}} \in [\tilde{C}_2 \delta_{W_\tti}, \epsilon_A - \tilde{C}_2 \delta_{W_\tti}]$, the following event occurs with probability at least $1-n_x^{-10}$: \cref{algo:dp} returns DP matrices for which there exists ambiguity matrices $Q_\tti, \{Q_{\tth}\supe\}_{e\in \mathcal{E}}$ such that all singular values lie in $[1/(2\sqrt{c_1}), 2\sqrt{c_1}]$; and for any $e\in \mathcal{E}$, the two error vectors
    \begin{align*}
    \Delta_{F_\tth}\supe := Q_\tth\supe (\hat{\Phi}_\tth\supe)^\top X\supe - F_\tth\supe \qquad \text{ and } \qquad \Delta_{F_\tti}\supe := Q_\tti (\hat{\Phi}_\tti\supe)^\top X\supe - F_\tti\supe
    \end{align*} are jointly sub-Gaussian with parameters $\delta_{F_\tth}^{(e)}$ and $\delta_{F_\tti}^{(e)}$, respective, where
    \begin{align}
        \delta_{F_\tth}\supe &= \tilde{C}_3  \left[\frac{1}{\sqrt{\snrmain}} + \sqrt{\kappa\supe} \left( \sqrt{\frac{\tilde{r}\supe}{\snrmain \cdot n_x}} + C_A (\delta_W)^2 \right) \right], \label{eq:delta-f-h}\\
        \delta_{F_\tti}\supe &= \tilde{C}_3  \left[\frac{1}{\sqrt{\snrmain}} + \sqrt{\frac{\tilde{r}\supe}{n_x}} + \sqrt{\kappa\supe} \left( \sqrt{\frac{\tilde{r}\supe}{\snrmain \cdot n_x}} + C_A (\delta_W)^2 \right) \left(1 + \sqrt{\frac{\kappa\supe\tilde{r}\supe}{n_x}}\right)\right]. \label{eq:delta-f-i}
    \end{align} 
\end{theorem}
The sub-Gaussian error bound provides isometric control: 
$\mathbb{E}[|u^\top \Delta_{F_\tti}\supe|^2]\le (\delta_{F_\tti}\supe)^2$ for any unit vector $u$.
Thus, the estimation error is uniformly small in every direction, rather than only in aggregate. This is stronger than a full $\ell_2$ bound, which scales as
$\mathbb{E}[\|\Delta_{F_\tti}\supe\|_2^2]\le r_\tti(\delta_{F_\tti}\supe)^2$; see \cref{coro:fifh}. Such direction-wise control prevents the downstream error from accumulating over all $r_\tti$ factors, yielding sharper bounds for estimating the prediction-invariant factors and $\beta^\star$, especially when $r_\tti$ diverges. When $\epsilon_A\asymp 1$, the heterogeneous factor error consists of the intrinsic weak-factor recovery term $1/\sqrt{\snrmain}$ and a higher-order term,
$\sqrt{\kappa\supe}\delta_W\{\delta_W+\sqrt{\tilde r\supe/d}\}$,
arising from estimating and separating the heterogeneous loading space. The invariant factor error has the same intrinsic term, an additional first-order term $\sqrt{\tilde r\supe/n_x}$ from separating ${F}_\tti^{(e)}$ and $F_\tth^{(e)}$ via uncorrelatedness in each environment $e\in \mathcal{E}$, plus higher-order terms from the partialling-out step:
$\sqrt{\kappa\supe}\delta_W\{\delta_W+\sqrt{\tilde r\supe/d}\}
(1+\sqrt{\kappa\supe\tilde r\supe/n_x}).$
These higher-order terms simplify under the mild condition $n_x\ge \kappa\supe\tilde r\supe$, leading to the cleaner rates in the following corollary. In the standard regime where $\lambda_{\max}[(B\supe)^\top B\supe] \asymp d$ and $\tilde{r}\supe \asymp 1$ up to poly-log factors, this condition matches the condition for consistent estimation of loading space: $n_x \gtrsim d/\snrmain$.  

\begin{corollary}
\label{coro:fifh}
    Under the setting of \cref{thm:factor-est}, further assume $n_x \ge \kappa\supe \tilde{r}\supe$, then the factor estimation errors can be simplified as
    \begin{align*}
        \delta_{F_\tth}\supe \lesssim \sqrt{ \frac{1}{\snrmain} +  \kappa\supe\left(C_A \delta_W^2\right)^2 } \qquad \text{and} \qquad \delta_{F_\tti}\supe \lesssim \sqrt{ \frac{1}{\snrmain} + \frac{\tilde{r}\supe}{n_x} + \kappa\supe  \left(C_A \delta_W^2\right)^2 },
    \end{align*}
    and $\mathbb{E}[\|\Delta_{F_\tth}\supe\|_2^2] \le r_\tth\supe (\delta_{F_\tth}\supe)^2$ and $\mathbb{E}[\|\Delta_{\tti}\supe\|_2^2] \le r_\tti (\delta_{F_\tti}\supe)^2$.
\end{corollary}

\begin{remark}[Technical challenge and novelty of analysis] The main technical challenge, beyond the single environment \cref{lemma:preliminary} benchmark, is that the DP matrix $\hat{\Phi}_\tth\supe$ needs to diversify both the effects of the idiosyncratic errors $U\supe$ and leakage from the invariant factors $F_\tti\supe$. To see this, using $X\supe = A\supe F_\tth\supe + B F_\tti\supe + U\supe$ with $A\supe = W_\tth\supe \khhe + W_\tti \khie$ and $B = W_\tti \kii$, we have
\begin{align*}
    (\hat{W}_\tth\supe)^\top X\supe &= \hat{H}\supe F\supe_\tth + (\hat{W}_\tth\supe)^\top W_\tti \kii F_\tti\supe + (\hat{W}_\tth\supe)^\top U\supe \\
    &= \hat{H}\supe F\supe_\tth + O_{\mathbb{P}}\left(\| W_\tti^\top \hat{W}_\tth\supe\|_2 \nu_{\max}(\kii)\right) + O_{\mathbb{P}}\left(\|\Sigma_U\supe\|_2^{1/2}\right)
\end{align*} 
for some matrix $\hat{H}\supe := (\hat{W}_\tth^{(e)})^\top A^{(e)}$ with $\nu_{\min}(\hat{H}\supe) \gtrsim \nu_{\min}(\khhe)$. 
Thus, unlike in standard factor estimation, the error in $\hat{W}_\tth\supe$ is not negligible when $W_\tth\supe$ is consistently estimated. 
Under $\epsilon_A\asymp 1$, a crude bound would control this leakage by $\| W_\tti^\top \hat{W}_\tth\supe\|_2 \le \|\sin\Theta(\hat{W}_\tth\supe, W_\tth\supe)\|_2 \lesssim \delta_W$. We obtain a sharper, higher-order bound by using a first-order expansion of the PCA solution:
\begin{align*}
    \hat{W}_\tth\supe = \left[{W}_\tth\supe + [{W}_\tth\supe]_\perp f(\hat{\Delta}) \right] \hat{Q} + O(\|\hat{\Delta}\|_2^2)
\end{align*} 
where $\hat{\Delta}:= [\hat{W}\supe (\hat{W}\supe)^\top - \hat{W}_\tti \hat{W}_\tti^\top] - W_\tth\supe (W_\tth\supe)^\top$ satisfying $\|\hat{\Delta}\|_2 \lesssim \delta_W$,  $f: \mathbb{R}^{d\times d} \to \mathbb{R}^{(d-r_\tth\supe) \times r_\tth\supe}$ is a linear map, and $\|\hat{Q}\|_2 \lesssim 1$. This expansion yields the dimension-free concentration error bound $\|W_\tti^\top [{W}_\tth\supe]_\perp f(\hat{\Delta})\|_2 \allowbreak \lesssim \delta_W \sqrt{\tilde{r}\supe/d}$, which further gives $\| W_\tti^\top \hat{W}_\tth\supe\|_2 \lesssim \delta_W \left[\delta_W + \sqrt{\tilde{r}\supe/d}\right]$. 

The invariant factor analysis is more involved because it also requires controlling the ``partialling-out" step in \eqref{eq:method-partialing-out}. The proof maintains first-order approximations throughout this sequence of spectral and least-squares operations; see \aosversion{Appendix C.1}{\cref{appendix:factor-est}}.
\end{remark}

\subsection{Estimation errors for the prediction-invariant factors and invariant signals}

We next analyze the estimation of the prediction-invariant factors and the final prediction rule. The following condition contains standard GLM-type regularity assumptions. We require local, rather than global, strong convexity at the true coefficients $\xi^\star_k$; global strong convexity assumptions would simplify some technical details. The full conditional exogeneity $\mathbb{E}[\varepsilon_k\supe|F\supe, U\supe] = 0$ in (a) is imposed to simplify the proof; see further discussions in \aosversion{Appendix C.2}{\cref{appendix:factor-est-star}}. 

\begin{condition}
\label{cond:z-reg}
    There exists a constant $c_2\ge 1$ such that for any $e\in \mathcal{E}_z$:
    \begin{itemize}
        \item[(a)] We observe $\mathcal{D}_z\supe = \{(X_i\supe, Z_{i}\supe)\}_{i=2n_x+1}^{2n_x+n_z}$ from \eqref{eq:model-invariant-z} with $\mathbb{E}[\varepsilon_k\supe|F\supe, U\supe] = 0$. 
        \item[(b)] $\sigma_z$ has uniformly bounded first-order and second-order derivatives: $\|\sigma_z'\|_\infty \lor \|\sigma_z''\|_\infty \le c_2$. $\sigma'_z(\cdot) \ge 0$ such that the loss function is convex. 
        \item[(c)] For each $k$, the eigenvalues of the following matrices range from $c_2^{-1}$ to $c_2$:
        \begin{align*}
            H\supe_k = \mathbb{E} \left[ \sigma_z'\left((\xi^\star_k)^\top F_{S^\star}\supe\right) F\supe (F\supe)^\top\right]
        \end{align*}
        \item[(d)] The noise $(\varepsilon_1^{(e)},\ldots, \varepsilon_q^{(e)})|F^{(e)},U^{(e)}$ is $\sqrt{c_2}$ jointly sub-Gaussian almost surely.
    \end{itemize}
\end{condition}

The next result builds on \cref{thm:factor-est}. To keep the presentation simple, we state it under balanced error scales across environments $e\in \mathcal{E}$: 
\begin{align*}
    \sqrt{2}[\delta_{F_\tti}\supe + \delta_{F_\tth}\supe] =: \delta_{F}\supe \asymp \bar{\delta}_{F} := \sqrt{\frac{1}{|\mathcal{E}_z|}\sum_{e\in \mathcal{E}_z} (\delta_F\supe)^2}, \qquad r\supe \asymp \bar{r} := \frac{1}{|\mathcal{E}_z|} \sum_{e\in \mathcal{E}_z} r\supe .
\end{align*}  
For clarity, we also state the result in the regime where the coefficients for auxiliary labels are well-conditioned. The general result is given in \aosversion{Appendix C.2}{\cref{appendix:factor-est-star}}.  

\begin{theorem}
\label{thm:factor-s-est}
    Under the setting of \cref{thm:factor-est} and the event defined therein, suppose further that \cref{cond:z-reg} holds, $\sup_{k\in [q]} \|\xi^\star_k\|_2 \asymp 1$, and $\lambda_{\min}((\Xi^\star)^\top \Xi^\star) \asymp q$. Denote $\tilde{r} = \bar{r} \cdot \log(n_z) + \log(|\mathcal{E}_z| \cdot q)$. There exist constants $\tilde{C}_1, \tilde{C}_2, \tilde{C}_3=\poly(c_1, c_2)$ such that if 
    \begin{align}
    \label{eq:cond-z-text}
    n_z \ge \tilde{C}_1 \tilde{r} (|\mathcal{E}_z| + \sqrt{\tilde{r}}) \qquad \text{and} \qquad (\bar{\delta}_F)^{-1} \ge \tilde{C}_1 |\mathcal{E}_z|
    \end{align} then, by choosing of $\lambda_{\mathsf{sel}} \in (\tilde{C}_2 \lambda_{\mathsf{sel}, \star}, \sqrt{q} /\tilde{C}_2 - \tilde{C}_2 \lambda_{\mathsf{sel}, \star})$ with $\lambda_{\mathsf{sel}, \star} = \sqrt{\frac{q + \tilde{r}}{n_z}} + \bar{\delta}_F + \frac{\tilde{r}}{n_z} = o(1)$, there exist invertible maps $Q_{S_\tti}$ and $Q_{S_\tth}$ such that with probability at least $1-n_z^{-10}$: 
    \begin{align*}
        \frac{Q_{S_\tti} (\hat{\Phi}_{S_\tti})^\top (\hat{\Phi}_{\tti}\supe)^\top X\supe - F_{S_\tti^\star}\supe}{\delta_{S_\tti}} ~~~~\text{and}~~~~ \frac{Q_{S_\tth} (\hat{\Phi}^{(e)}_{S_\tth})^\top (\hat{\Phi}_{\tth}\supe)^\top X\supe - F_{S_\tth^\star}\supe}{\delta_{S_\tth}}
    \end{align*} are jointly sub-Gaussian with parameter $\tilde{C}_3$ for any $e\in \mathcal{E}$ and any $e\in \mathcal{E}_z$, respectively, where
    \begin{align}
        \delta_{S_\tti} = \sqrt\frac{r_\tti + \log(n_z)}{q \cdot n_z \cdot |\mathcal{E}_z|} + \bar{\delta}_F+ \frac{\tilde{r}}{n_z} ~~\text{and}~~
        \delta_{S_\tth} = \sqrt\frac{\bar{r} - r_\tti + \log(n_z)}{q \cdot n_z} + \bar{\delta}_F+ \frac{\tilde{r}}{n_z} \label{eq:fs-error}.
    \end{align}
\end{theorem}

The errors in \cref{thm:factor-s-est} again have an isometric form, as in
\cref{thm:factor-est}. Under the pervasive-in-$Z$ regime, the sample-size condition
\eqref{eq:cond-z-text} is comparable to the usual requirement for consistently estimating the GLM coefficients and the factors $(F_\tti, F_\tth)$ when
$|\Ecal_z|\asymp 1$. Each rate has three components: a first-order GLM
estimation term of order $n_z^{-1/2}$, the inherited factor-estimation error
$\bar\delta_F$, and a second-order GLM term $\tilde r/n_z$. The latter two
terms affect both $F_{S_\tti^\star}$ and $F_{S_\tth^\star}$. The invariant
component has a faster first-order rate because it is already aligned across
environments, whereas the heterogeneous component must be aligned separately
in each environment since the ambiguity matrices for $F_\tth\supe$ differ
across environments. Importantly, the first-order rates depend only on the dimensions of the corresponding prediction-invariant components. Thus, when the
parametric term dominates, estimating $F_{S_\tti^\star}$ is not affected by a
large heterogeneous dimension $\bar{r}-r_\tth$, and estimating $F_{S_\tth^\star}$ is not affected
by a large invariant dimension $r_\tti$, vice versa. Moreover, using more auxiliary labels does not hurt efficiency: in the (up to poly-log factors) balanced regime
$\tilde r\asymp r_\tti\asymp \bar r-r_\tti$ and $|\Ecal_z|\asymp 1$, the errors
decrease with $q$ as long as $q=o(n_z/\bar r)$, after which the second-order term dominates.

We finally analyze the estimation error for the invariant parameter $\beta^*$. The next theorem establishes the $\ell_2$  error of $\hat\beta$ up to the shared ambiguity matrix $Q_S = \diag\{Q_{S_\tti}, Q_{S_\tth}\}$, and shows how the learned signal generalizes to other environments. The regularity condition is analogous to \cref{cond:z-reg} and is deferred to the \aosversion{Appendix C.2}{\cref{appendix:factor-est-star}}. 

\begin{theorem}
    Under the setting of \cref{thm:factor-s-est} and the event defined therein, suppose further that \aosversion{Condition C.1}{\cref{cond:y-reg}} with constant $c_3$ holds. There exists some constant $\tilde{C}_1, \tilde{C}_2, \tilde{C}_3 = \poly(c_1,c_2,c_3)$ such that if $|\mathcal{E}_y| \cdot n_y \ge \tilde{C}_1 (|S^\star| + \log(n_y))$ and $\delta_{S_\tti} \lor \delta_{S_\tth} \le \tilde{C}_1^{-1}$, then
    \begin{align*}
        \frac{\left\|Q_S^{-\top} \hat{\beta} - \beta^\star\right\|_2}{\tilde{C}_2} \le \sqrt{\frac{|S^\star| + \log(n_y)}{n_y \cdot |\mathcal{E}_y|}} + \delta_{S_\tti} + \delta_{S_\tth} := \delta_y \qquad \text{with} \qquad Q_S = \begin{bmatrix}
            Q_{S_\tti} & 0 \\
            0 & Q_{S_\tth} \\
        \end{bmatrix}
    \end{align*} occurs with probability at least $1-n_y^{-10}$, under which we have the following generalization of invariant signals
    \begin{align*}
        \forall e\in \mathcal{E}_z \setminus \mathcal{E}_y, \qquad &\sqrt{\mathbb{E}\left[\left| \hat{\beta}^\top \tilde{F}\supe_S(X\supe) -  (\beta^\star)^\top F_{S^\star}\supe \right|_2^2\right]} \le \tilde{C}_3 \delta_y \\
        \forall e\in \mathcal{E} \setminus \mathcal{E}_z, \qquad & \sqrt{\mathbb{E}\left[\left| \hat{\beta}_\tti^\top \tilde{F}\supe_{S_\tti}(X\supe) -  (\beta_\tti^\star)^\top F_{S^\star_\tti}\supe \right|_2^2\right]} \le \tilde{C}_3 \delta_y
    \end{align*}
\end{theorem}

When $\lambda_\star \asymp \lambda_{\min}((B\supe)^\top B\supe) / \|\Sigma_U^{(e)}\|_2 \gtrsim d$, $\epsilon_A \asymp 1$ and $\bar{r}$ grows, the second-order error in estimating $[F_\tti, F_\tth]$ is dominated by the leading terms and the final prediction error can be expanded (up to log factors) as
\begin{align}
\label{eq:error-pervasive}
    \delta_y \asymp \underbrace{\sqrt{\frac{|S_\star|}{n_y \cdot |\mathcal{E}_y|}}}_{\delta_\beta} + \underbrace{\sqrt{\frac{(\bar{r} - r_\tti) \lor (r_\tti / |\mathcal{E}_z|)}{n_z \cdot q}} + \frac{\bar{r}}{n_z}}_{\delta_{Z}} + \underbrace{\frac{1}{\sqrt{d}} + \sqrt{\frac{\bar{r}}{n_x}}}_{\delta_F}.
\end{align}
Here $\delta_\beta$ is the error from estimating the prediction parameter $\beta^\star$, $\delta_Z$ is the error from using the auxiliary labels to identify prediction-invariant factors $F_{S^\star}$, and $\delta_F$ is the error rate from estimating and disentangling the latent factors. The term $1/\sqrt{d}$ is the intrinsic factor recovery error, and $\sqrt{\bar{r}/n_x}$ comes from disentangling invariant factors from the heterogeneous ones. 

\begin{remark}[Benefits of auxiliary labels $Z$ in sample efficiency]
Auxiliary labels help not only by identifying the transferable heterogeneous signal  $(\beta_\tth^\star)^\top F_{S_\tth^\star}$, but also by reducing variance even when $|S_\tth^\star|=0$. In this case, the prediction error becomes
\begin{align*}
    \delta_y \asymp \sqrt{\frac{|S_\tti^\star|}{n_y \cdot |\mathcal{E}_y|}} + \sqrt{\frac{r_\tti}{n_z \cdot |\mathcal{E}_z| \cdot q}} + \frac{\bar{r}}{n_z} + \delta_F
\end{align*} 
which is much smaller than the error from directly regressing $Y$ against $F_\tti$, whose error scales as $\delta_F + \sqrt{r_\tti/(n_y \cdot |\mathcal{E}_y|)}$, when $n_z$ is large enough and $|S^\star| \ll r_\tti$. 
\end{remark}

\section{Simulation studies and real-world validation}
\label{sec:numerical}

\subsection{Simulation studies}
We set $\mathcal{E}_z = \mathcal{E}$ with $|\mathcal{E}| = 3$, $r_\tti=8$, $r_\tth\supe=8$, $|S_\tti^\star| = |S_\tth^\star| = 3$, with binary auxiliary labels $Z^{(e)} \in \{0, 1\}^{q}$ (logistic link) and continuous response  $Y^{(e)}$ (identity link). For simplicity, we assume $r_\tti, r\supe, |S_\tti^\star|, |S_\tth^\star|$ are pre-known. Detailed data-generating process and implementation details are given in \aosversion{Appendix E.1}{\cref{subsec:simulation}}. Below, we summarize key results from the simulation studies. 

\myparagraph{Alignment of invariant factors.} We first compare how different methods can estimate $F_\tti\supe$ across $e\in \mathcal{E}$ when $d=1024$. Given a latent factor map $\tilde{F}_{\tti}^{(e),\mathtt{m}}$ produced by method $\mathtt{m}$, we measure alignment quality by 
\begin{align} \label{eq:simu-f}
    \mathrm{MSE}_{F_\tti}(\mathtt{m}) = \inf_{Q\in \mathbb{R}^{r_\tti \times r_\tti}} \frac{1}{|\mathcal{E}|} \sum_{e\in \mathcal{E}} \frac{1}{n_{test}} \sum_{i=1}^{n_{test}} \left\|Q \tilde{F}_\tti^{(e), \mathtt{m}}(X_i\supe) - F_{\tti,i}\supe \right\|_2^2 / r_\tti
\end{align} 
with $n_{test}=30000$ held-out observations in each environment. The oracle estimator has access to $B\supe$, and hence its performance characterizes the intrinsic lower bound due to idiosyncratic errors. IHD uses $\widetilde{F}_I^{(e),\text{IHD}}(x) = (\widehat{\Phi}_I^{(e)})^\top x$ from \cref{algo:dp}. PoolPCA uses $\tilde{F}_{\tti}^{(e),\mathtt{PoolPCA}}(x) = \hat{W}^\top x$ where $\hat{W}$ represents the top-$r_\tti$ eigenvectors of the pooled covariance matrix of $X$. Note when $r_\tti$ is known, the dimension reduction map AJIVE \citep{feng2018angle, yang2025estimating} used matches $\tilde{F}_{\tti}^{(e),\mathtt{AJIVE}}(x) = \hat{W}_\tti^\top x$ with the intermediate $\hat{W}_\tti$ in \cref{algo:dp}.

\cref{fig:exp1} (a) shows how $\text{MSE}_{F_I}$ scales with $n_x \in \{2^l\}_{l=7}^{14}$. Both PoolPCA and AJIVE cannot consistently estimate the invariant factors $F_\tti\supe$, whereas IHD achieves errors decaying at a rate of $n_x^{-1}$ when $n_x$ is relatively small, as the slope of the \myblue{blue} curve is roughly $-1$ when $\log_2(n_x) \le 10$, after which it saturates near the oracle error. This matches the theoretical rate from \cref{thm:factor-est} 
\begin{align}
\delta_{F_\tti}^2 \asymp \frac{r_\tti}{n_x} + \frac{1}{\lambda_{\min}((B\supe)^\top B\supe)} \asymp n_x^{-1} + d^{-1}.
\end{align} 
The saturation at $d^{-1}$ reflects the intrinsic noise floor from idiosyncratic errors.

\myparagraph{Alignment of prediction-invariant factors.} Given $d=1024$ and $n_x = 2^{14}$, we next examine whether the algorithm described in \cref{sec:method2} can construct feature maps $\tilde{F}_{S_\tti}(x)$ and $\tilde{F}_{S_\tth}(x)$ that extract and align $F_{S_\tti}$ and $F_{S_\tth}$ across $e\in \mathcal{E}$ using the analogous alignment metric to \eqref{eq:simu-f}. 
\begin{figure}[!t]
\centering
\begin{tabular}{cc}
\subfigure[]{
\includegraphics[height=5cm]{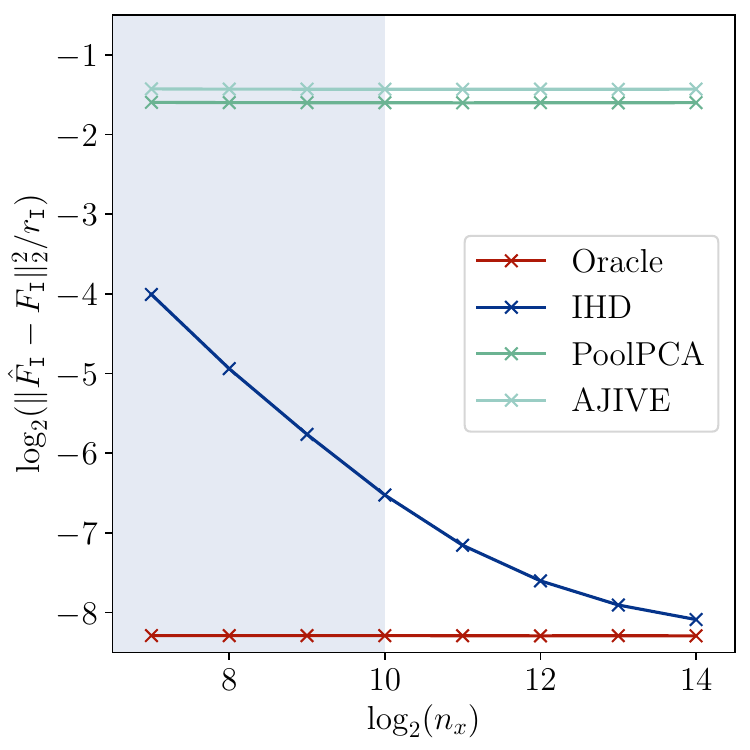}
}&\subfigure[]{
\includegraphics[height=5cm]{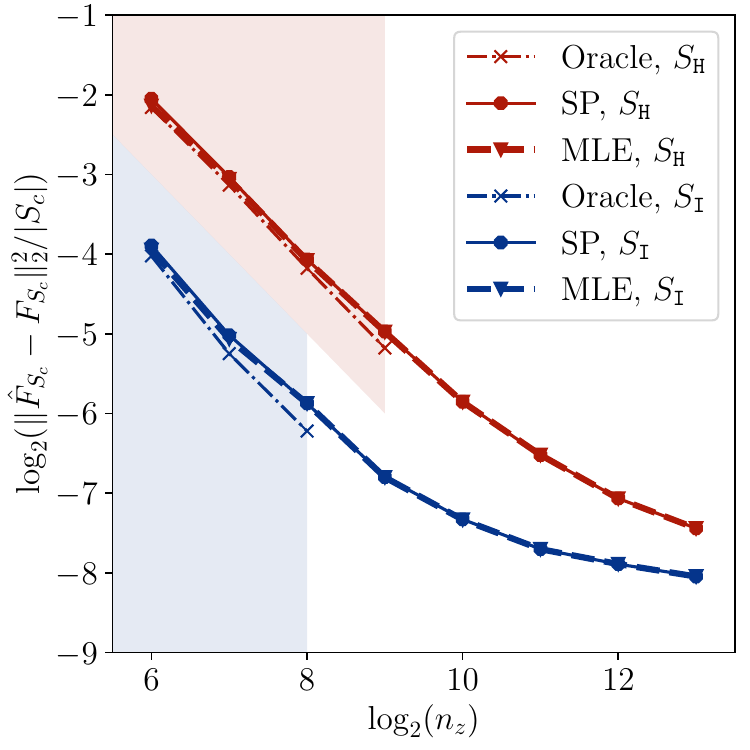}
}\\
\subfigure[]{
\includegraphics[height=5cm]{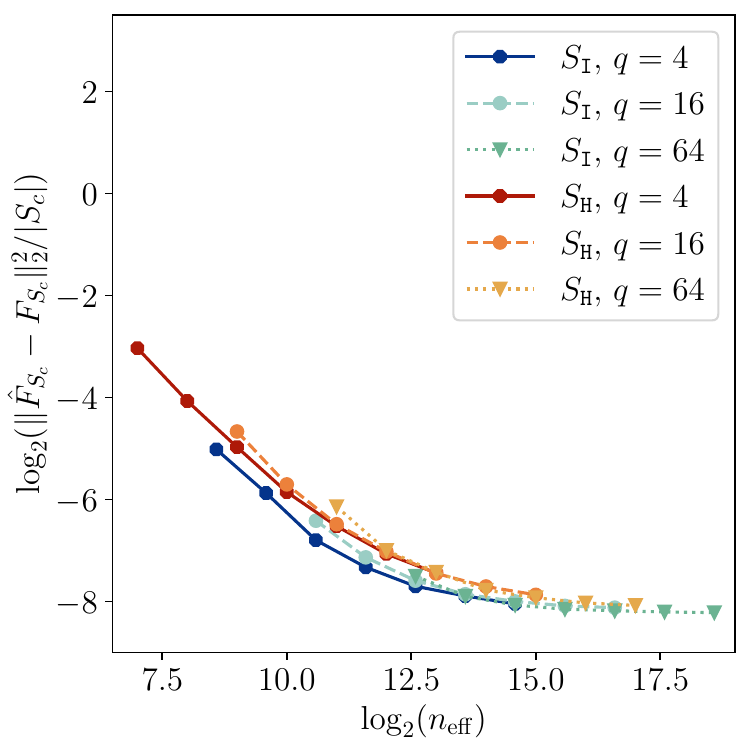}
}&
\subfigure[]{
\includegraphics[height=5cm]{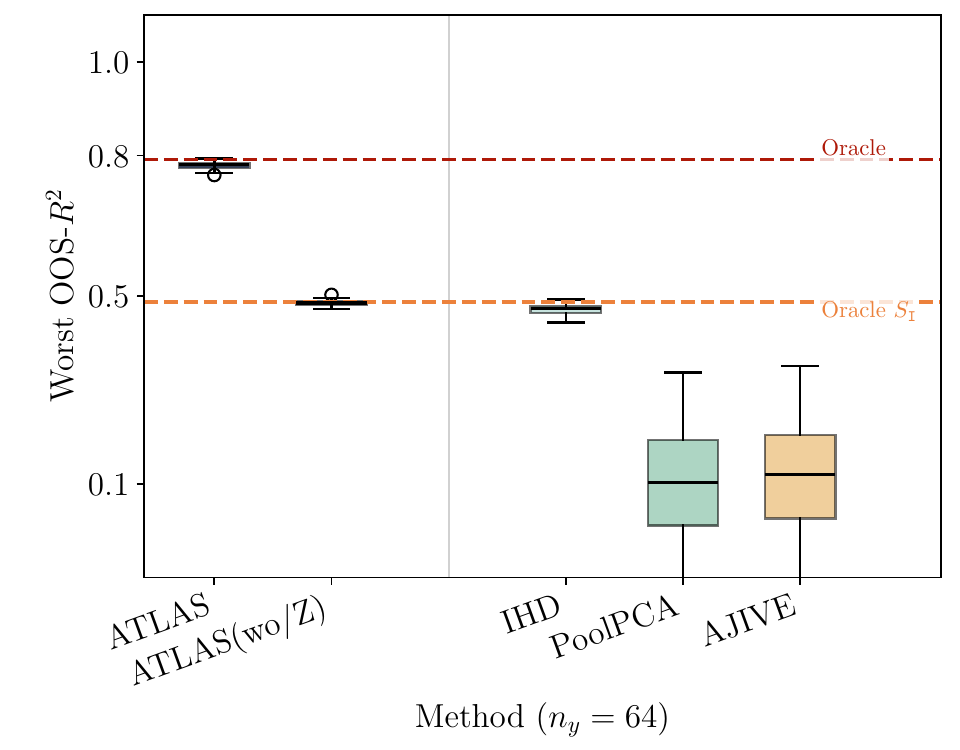}
}
\end{tabular}

\caption{Panel (a) (resp. (b)--(c)) depict the results on factor estimation of $F_\tti$ (resp. $F_{S_\tti}, F_{S_\tth}$): both $x$ and $y$ axes are $\log_2(\cdot)$ scaled. Different curves in the figures (a)--(c) show how the various methods estimate the corresponding factors under different settings. The shaded area represents the regime where the empirical error curve has slopes nearly $-1$; (d) depicts the worst-environment out-of-sample $R^2$: the oracle methods are represented by dashed lines, the methods with the supervision of $Z$ in $\mathcal{E}$ are in the left block, while the methods blind to $Z\supe$ in $\mathcal{E}$ are in the right block.}
\label{fig:exp1}
\end{figure}

\cref{fig:exp1} (b) shows results for $q=4$ and $n_z \in \{2^l\}_{l=6}^{13}$. For both $S_\tti$ (\myblue{blue curves}) and $S_\tth$ (\myred{red curves}), the proposed spectral methods (SP) depicted in solid lines achieve nearly identical performance to the non-convex MLE objective \eqref{eq:method-mle} (bold dashed lines): error decay at rate $n_z^{-1}$ in the shaded regime before saturating at the floor inherited from the first stage factor estimates, which confirms \cref{thm:factor-s-est}:
\begin{align*}
    \delta_{F_{S_\mathtt{c}}}^2 \asymp \frac{r}{n_z \cdot q \cdot L_{\mathtt{c}}} + \delta_F^2 \asymp \frac{1}{n_z \cdot q \cdot L_{\mathtt{c}}} + n_x^{-1} + d^{-1} ~~~\text{with}~~~ L_{\mathtt{c}} = 1 + \indicator\{\mathtt{c} = \tti\} (|\mathcal{E}_z| - 1).
\end{align*} 
\cref{fig:exp1} (c) further confirms this by showing that estimation error curves for different $q\in \{4, 16, 64\}$ and $\mathtt{c} \in \{\tti, \tth\}$ collapse when plotted against the effective sample size $n_{\text{eff}} = n_z \cdot q \cdot L_c$, illustrating that increasing $q$ is equivalent to increasing $n_z$ in the pervasive regime.

\myparagraph{Generalization in new environments. } Fixing $d=1024$, $n_x=2^{14}$, and $n_z=2^{10}$, we evaluate how well different methods transfer to three new environments $\mathcal{E}_o$ outside $\mathcal{E}$, using the worst-case out-of-sample $R^2$ defined on $\mathcal{E}_o$. All methods run pooled linear least squares on the extracted factors and make predictions using the learned predictor in new environments; see the descriptions for different methods and evaluation metrics in \aosversion{Appendix E.1}{\cref{subsec:simulation}}. 
\cref{fig:exp1} (d) shows that ATALS with auxiliary labels $Z\supe$ has performance similar to the oracle predictor built on the true $F_{S^\star}^{(e)}$, confirming \cref{prop:robust1}. Without $Z$ in the new environment, both {ATLAS(wo/Z)} and {IHD} have worst-case OOS $R^2$ similar to that of the \myorange{Oracle $F_{S_\tti}$} based on $F_{S_\tti^\star}^{(e)}$. Moreover, ATLAS(wo/Z) slightly outperforms IHD, confirming that the auxiliary labels also reduce variance beyond identifying $F_{S_\tth}\supe$. All the other methods generalize poorly in the new environments. 

\subsection{Real data application: temporally generalizable prediction of future disease activity}
\label{sec:realdata}

Rheumatoid arthritis (RA) is a chronic autoimmune disease affecting $\sim 1\%$ of the global population \citep{smolen2016eular}. Accurate assessment of disease activity is central to RA management, because persistent moderate or high disease activity may lead to progressive joint damage, impaired physical function, and the need for treatment escalation \citep{aletaha2018diagnosis}. Disease activity score (DAS) and its variants are widely used to summarize RA activity and guide clinical decisions, but disease activity information is often incompletely recorded in routine EHR data and may appear across both structured fields and unstructured clinical notes \citep{cheng2025inferring}. Reliable DAS labels often require manual chart review and are relatively scarce, while routine EHR covariates are broadly available. At the same time, unstructured clinical notes contain rich DAS-related information, including joint counts, inflammatory markers, and physician global assessments, which can be extracted by an LLM-based agent and used as auxiliary variables. This creates a natural setting for ATLAS: abundant historical EHR covariates $X$, limited manually curated target labels $Y$, and richer LLM-derived auxiliary variables $Z$ that can help identify prediction-relevant latent structure.

Building a future DAS prediction model from these data requires not only good predictive accuracy, but also stability over time. Treatment protocols, clinical practice, documentation patterns, and measurement conventions for inflammatory markers have evolved substantially over time, potentially shifting both the observed EHR features and their associations with disease activity. As a result, models trained on historical data may perform well during training but degrade when applied to current or future patients. This makes future DAS prediction a natural empirical setting for evaluating ATLAS, since the data have both partially observed labels and heterogeneous temporal environments.

We use the Mass General Brigham (MGB) EHR system to study 8,095 RA-positive patients who first initiated DMARDs after 2006-01-01, contributing 175,436 labeled patient visits with 174,532 $(X, Z)$ samples and 904 $(X, Y)$ samples. For each labeled visit at calendar date $t_0$, the prediction goal is to determine whether the patient's RA disease activity at $t_0$ is at least moderate using only EHR information recorded before $t_0-360$ days. To reflect temporal heterogeneity, we define environments by calendar-time partitions of these labeled visits. The first three environments, covering 2009--2018, are used for training, while the last three environments, covering 2019--2021, are held out to evaluate temporal generalizability. \aosversion{Table 2}{\cref{table:real_data}} in \aosversion{Appendix E.2}{\cref{subsec:real-data}} reports the sample sizes. 

We construct the historical EHR covariates from RA-related structured codes, including diagnosis codes, medication prescriptions, and laboratory tests, as well as clinical concepts extracted from narrative notes using natural language processing (NLP). We aggregate historical EHR information over the one-year covariate window $[t_0-720,t_0-360)$, which is split into two 180-day sub-windows, $[t_0-720,t_0-540)$ and $[t_0-540,t_0-360)$, referred to as the past and near windows, respectively; see the figure in \aosversion{Appendix E.2}{\cref{subsec:real-data}} for an illustration. We first select RA-related structured codes and NLP-derived clinical concepts using an existing knowledge graph \citep{xiong2023knowledge}. These base variables are then expanded across the two temporal windows and combined with demographic variables and derived feature representations, yielding the final $d=3539$ dimensional covariate vector. Details of the feature construction are provided in \aosversion{Appendix E.2}{\cref{subsec:real-data}}. For ATLAS, these historical EHR covariates serve as $X$, the LLM-derived DAS-related quantities serve as auxiliary variables $Z$, and the binary indicator of moderate-or-high disease activity at the future prediction time serves as the target outcome $Y$. We compare ATLAS with Lasso on the original covariates, PoolPCA, AJIVE, and IHD using only the invariant representation, and an ATLAS ablation with $r_\tti=0$, which uses auxiliary labels to align the latent factors without first separating invariant and heterogeneous components. For each method, we construct the corresponding representation using the training environments, fit the same downstream prediction model for $Y$, and evaluate AUC separately in each held-out test environment.

\begin{table}
\begin{center}
    \begin{tabular}{c|l|l|l|l|l|l}
        \hline
        \hline
        Year & Lasso & PoolPCA & AJIVE & IHD & ATLAS ($r_\tti=0$) & ATLAS \\
        \hline
        2019 & 0.473 & 0.495 & 0.502 & 0.575 & 0.612 & 0.601 \\
        2020 & 0.573 & 0.522 & 0.531 & 0.541 & 0.574 & 0.554 \\
        2021 & 0.445 & 0.459 & 0.483 & 0.682 & 0.649 & 0.736 \\
        \hline
        Average & 0.497 & 0.492 & 0.505 & 0.600 & 0.612 & 0.630 \\
        \hline
        \hline
    \end{tabular}
\end{center}
\caption{Out-of-sample prediction performance measured by AUC for different methods, including the one that directly regresses $Y$ on the full covariate $X$ with $L_1$ penalty (Lasso), and five methods that regress $Y$ on extracted latent factors similar to the simulation. The implementation details can be found in \aosversion{Appendix E.2}{\cref{subsec:real-data}}.}
\label{table:results}
\end{table}

\cref{table:results} reports the held-out AUCs for 2019--2021. Lasso, PoolPCA, and AJIVE achieve average AUCs close to 0.5, suggesting that directly using the original covariates or relying only on pooled/shared low-dimensional structure does not provide stable temporal transfer in this application. IHD improves the average AUC to 0.600, indicating that invariant factors extracted from the multi-environment covariance structure are more transportable. ATLAS achieves the best average AUC of 0.630, further improving over IHD by using LLM-derived auxiliary variables to align prediction-invariant heterogeneous factors. The ablation ATLAS with $r_I=0$ also improves over PoolPCA and AJIVE, but performs worse than the full ATLAS procedure, supporting the value of first disentangling invariant and heterogeneous factors before applying auxiliary supervision.

\def\ebreve{\breve{e}}
\section{Extensions and Discussion}
\label{sec:discuss}

In this paper, we consider the multi-environment factor model and factor regression. We propose ATLAS, which can uncover the invariant factors that have common loading effects on the covariates in an unsupervised manner, and further extract additional prediction-invariant factors for the response with the supervision of auxiliary labels. We provide a thorough statistical analysis of the proposed method. In particular, under the instance-level necessary assumption \eqref{eq:ident-intro}, together with some regularity conditions, we can efficiently estimate the latent prediction-invariant factors and the corresponding invariant associations, yielding robust transfer in unseen environments.

The main analysis assumes that all the environments are available at fitting time, the associations between the factors and labels are (generalized) linear models, and that auxiliary labels are observed at least in some environments. We next discuss extensions beyond these settings.

\myparagraph{Prediction in a new environment. }
We briefly describe how to make predictions in a new environment $\breve{e} \notin \mathcal{E}$ without re-running the full algorithm on $\Ecal\cup \{\ebreve\}$. Suppose that in environment $\ebreve$, we observe $2\breve{n}_x$ samples of $X$ and $\breve{n}_z$ samples of $(X, Z_{S_Z})$ observations with only a subset $S_Z$ of auxiliary labels with $S_Z \subseteq [q]$ and $|S_Z| \ge |S_\tth^\star|$. Let $\hat{W}^{(\breve{e})}$ be the matrix of the top-$r^{(\breve{e})}$ eigenvectors of the covariance matrix of $X$ in environment $\breve{e}$. Using the previously learned loading space $\hat{W}_\tti$ together with $\hat{W}^{(\breve{e})}$, we construct $\hat{W}_\tth^{(\breve{e})}$ and the DP matrices $\hat{\Phi}_\tti^{(\breve{e})}$ and $\hat{\Phi}_\tth^{(\breve{e})}$ as in line 12-15 and line 16-18 in \cref{algo:dp}. The shared projection 
 $\hat{\Phi}_{S_\tti}$ in \eqref{eq:est-si} can be reused to obtain $\tilde{F}_{S_\tti}^{(\breve{e})}$ from $\tilde{F}_\tti^{(\breve{e})}$. 
 
It remains to align the heterogeneous prediction-invariant factors in the new environment. Let  $\hat{\Psi}_\tth^{(\breve{e})}$ solve
\begin{align*} 
\hat{\Psi}_\tth^{(\breve{e})} = \argmin_{\Psi_\tth \in \mathbb{R}^{q \times \hat{r}_\tth^{(\breve{e})}}}  \sum_{i-2\breve{n}_x \in [\breve{n}_z]} \sum_{k\in S_Z} \ell_z\left([\hat{\Psi}_\tti]_{k,:} \tilde{F}_{\tti,i}^{(\breve{e})} + [\Psi_\tth]_{k, :} \tilde{F}_{\tth,i}^{(\breve{e})}, Z_{i, k}^{(\breve{e})}\right).
\end{align*} 
where $\hat{r}_\tth^{(\breve)} = r^{\breve{e}} - \hat{r}_\tti$, $\hat{\Psi}_\tti$ is defined in \eqref{eq:est-psi} and $\ell_z$ is the MLE loss for the mean function $\sigma_z(\cdot)$. We then set
$$\hat{\Phi}_{S_\tth}^{(\breve{e})} = ([\hat{\Psi}_\tth^{(\breve{e})}]_{S_Z,:})^\top [\hat{\Xi}_\tth]_{S_Z,:} ([\hat{\Xi}_\tth]_{S_Z,:}^\top [\hat{\Xi}_\tth]_{S_Z,:})^{-1},$$ 
where $\hat{\Xi}_\tth$ is defined in \eqref{eq:est-xi-h}. These projection matrices can then be used to recover the aligned factors $\tilde{F}_S^{(\breve{e})}$ and make predictions with the score 
$\hat{\beta}^\top\tilde{F}_S^{(\breve{e})}$. 
Under analogous regularity conditions and 
$\lambda_{\min}([\Xi^\star_{\tth}]_{S_Z,:}^\top [\Xi^\star_{\tth}]_{S_Z,:}) \asymp |S_Z|$, we expect that the $L_2$ prediction error of the resulting predictor satisfies
\begin{align*}
    \sqrt{\mathbb{E}\left[\left\|\hat{\beta}^\top \tilde{F}_S^{(\breve{e})} - (\beta^\star)^\top F_{S^\star}^{(\breve{e})} \right\|_2^2\right]} \lesssim \delta_y + \sqrt{\frac{r^{(\breve{e})} - r_\tti}{\breve{n}_z \cdot |S_Z|}} + \frac{r^{(\breve{e})}}{\breve{n}_z} + \underbrace{\frac{1}{\sqrt{d}} + \sqrt{\frac{r^{(\breve{e})}}{\breve{n}_x}}}_{\breve{\delta}_F},
\end{align*} 
in the pervasiveness regime considered by \eqref{eq:error-pervasive}. A rigorous analysis of this extension is left for future work.

\myparagraph{Nonlinear relationship between $Y$ and $F_{S^\star}$.}
Our method and the robust transfer-learning interpretations in \cref{prop:robust1} and \cref{prop:robust2} extend naturally to nonlinear additive GLMs. Consider
\begin{align}
\label{eq:model-y-nonlinear} 
    \mathbb{E}\left[Y\supe \mid F_{\tti}\supe, F_{S_\tth^\star}\supe\right] = \sigma_y\left( g_\tti (F_{S_\tti^\star}\supe) + g_\tth (F_{S_\tth^\star}\supe) \right)
\end{align} where $\sigma_y$ is the given mean function, and $g_\tti, g_\tth$ are both unknown nonparametric functions. This reduces to  \eqref{eq:model-invariant-y} when $g_\tti(v) = (\beta_\tti^\star)^\top v$ and $g_\tth(v) = (\beta_\tth^\star)^\top v$. 
When $g_\tti(\cdot)$ and $g_\tth(\cdot)$ are generic nonparametric functions, all earlier steps of the pipeline remain unchanged; only the last regression step \eqref{eq:est-y} is replaced by  empirical risk minimization over nonlinear hypothesis classes $\mathcal{G}_\tti$ and $\mathcal{G}_\tth$: 
\begin{align*}
    \hat{g}_\tti, \hat{g}_\tth \in \argmin_{g_\tti \in \mathcal{G}_\tti, g_\tth \in \mathcal{G}_\tth} \frac{1}{|\mathcal{E}_y| \cdot n_y} \sum_{\substack{e\in \mathcal{E}_y \\ i-2n_x-n_z \in [n_y]}} \ell_y\left(g_\tti(\tilde{F}_{S_\tti}\supe(X\supe_i)) + g_\tth(\tilde{F}_{S_\tth}\supe(X\supe_i)), Y\supe_i \right)
\end{align*} 
where $\tilde{F}_{S_\tti}^{(e)}(\cdot) = (\hat{\Phi}_{S_\tti})^\top (\hat{\Phi}_\tti\supe)^\top (\cdot)$ and $\tilde{F}_{S_\tth}^{(e)}(\cdot) = (\hat{\Phi}_{S_\tth}\supe)^\top (\hat{\Phi}_\tth\supe)^\top (\cdot)$. Taking $\mathcal{G}_\tti$ and $\mathcal{G}_\tth$ to be neural network classes gives a direct generalization the FAR-NN estimator in \citep{fan2024factor}. Following the proof strategy of \cite{fan2024factor}, when both $g_\tti(\cdot)$ and $g_\tth(\cdot)$ lie in the hierarchical composition model \citep{schmidt2020nonparametric, kohler2021rate}, which is a class of finite compositions of low-dimensional smooth functions, we can obtain that the estimated score satisfies
\begin{align*}
    \left[\frac{1}{|\mathcal{E}_y|} \sum_{e\in \mathcal{E}_y} \mathbb{E}\left[\left|\hat{g}^{(e)}(X^{(e)}) - g_\star(F\supe) \right|^2\right] \right]^{1/2} \lesssim \underbrace{\left(|\mathcal{E}_y| \cdot n_y \right)^{-\frac{\gamma^\star}{2\gamma^\star+1}}}_{\delta_g} + \delta_{S_\tti} + \delta_{S_\tth}
\end{align*} where $\hat{g}^{(e)}(x) = \hat{g}_\tti(\tilde{F}_{S_\tti}\supe(x)) + \hat{g}_\tth(\tilde{F}_{S_\tth}\supe(x))$, $g_\star(f) = g_\tti(f_{S_\tti^\star}) + g_\tth(f_{S_\tth^\star})$, $\gamma^\star$ is the effective smoothness, which is the smallest ratio of smoothness to dimension in the composition, and the factor estimation errors $(\delta_{S_\tti}, \delta_{S_\tth})$ are those defined in \eqref{eq:fs-error}. It replaces the parametric estimation error of estimating the $(|S_\tti^\star|+|S_\tth^\star|)$-dimensional parameter by a non-parametric rate $\delta_g$ with the same sample size $(n_y \cdot |\mathcal{E}_y|)$.

\myparagraph{Nonlinear relationship between $Z$ and $F_{S^\star}$.}
Our main analysis assumes a generalized linear association between $Z$ and $F_{S^\star}$ as in \eqref{eq:model-invariant-z}. Our algorithm continues to apply to the case where the conditional mean of $Z$ given $F_{S^\star}$ is nonparametric, provided that $Z\indep F| F_{S^\star}$ and the joint law of $(Z, F_{S^\star})$ is the same across $e\in \mathcal{E}_z$. In this case, the condition $q \ge |S_\tti^\star| \lor |S_\tth^\star|$ may be relaxed by constructing a feature map of $Z$. To see this, for any map $h: \mathbb{R}^q \to \mathbb{R}^{\bar{q}}$, the linear coefficient satisfies
\begin{align*}
    [\bar{\Xi}^\star]^{(e)} \equiv \mathbb{E}[h(Z) F_{S^\star}^\top ] \left\{ \mathbb{E} [F_{S^\star} F_{S^\star}^\top ] \right\}^{-1} . 
\end{align*} 
If the two cross-moment matrices
$\mathbb{E}[h(Z)F_{S_\tti^\star}^\top]$ and
$\mathbb{E}[h(Z)F_{S_\tth^\star}^\top]$ have ranks
$|S_\tti^\star|$ and $|S_\tth^\star|$, 
respectively, then the introduced feature map $h$ together with the invariance of the joint law of $(Z, F_{S^\star})$ reduces the nonlinear case to the linear alignment problem we studied.
For a fully nonparametric model $\mathbb{E}[Z^{(e)}|F^{(e)}] = m_Z(F_{S^\star}^{(e)})$, it may be possible to develop a multi-environment counterpart of the projection pursuit \citep{huber1985projection} or multi-index model \citep{xia2008multiple} to align $F_{S_\tth^\star}^{(e)}$. Such an analysis would require additional overlap conditions on the factor distributions across environments and is left for future work.

\myparagraph{No auxiliary labels $Z$. }
When no auxiliary labels are available but $Y$ is observed in some environments $\mathcal{E}_y \subseteq \mathcal{E}$, our method can still recover the invariant signals in the regression setting with $\sigma_y(v) = v$, even for a nonparametric regression function and $S_\tth^\star \neq \emptyset$.  Under the non-linear model and $F_\tti^{(e)} \indep F_\tth^{(e)}$, 
\begin{align*}
    \mathbb{E}[Y^{(e)}|F_\tti^{(e)}] = \mathbb{E} \left[ g_\tti(F_{S_\tti^\star}\supe) + g_\tth(F_{S_\tth^\star}\supe)|F_\tti^{(e)}\right] = g_\tti(F_{S_\tti^\star}\supe) + b^{(e)} ~~~\text{for}~~ b^{(e)} = \mathbb{E}\left[g_\tth(F_{S_\tth^\star}\supe)\right]
\end{align*} 
Then one can estimate the invariant signals $g_\tti(F_{S_\tti^\star}\supe)$ up to a constant by fitting environment-specific intercepts, i.e.,
\begin{align*}
    \hat{g}_\tti, \{\hat{b}^{(e)}\}_{e\in \mathcal{E}_y} \in \argmin_{g_\tti, \{b^{(e)}\}_{e\in \mathcal{E}_y}} \sum_{e\in \mathcal{E}_y, i\in [n_y]} \left|Y_i^{(e)} - b^{(e)} - g_\tti(\tilde{F}_{\tti, i}^{(e)})\right|^2.
\end{align*}


\appendix
\appendixrefs

\newpage 

\noindent{\LARGE \textbf{Appendix}}
\vspace{10pt}

\section{Organization and key notations}
\label{append:notation}

The supplement is organized as follows:
\begin{description}
    \item[\cref{append:notation}] presents a table of main notations used in the main paper.
    \item[\cref{append:ident-all}] presents further identification theorems and discussions omitted in \cref{sec:ident}.
    \item[\cref{append:factor}] presents the full the non-asymptotic results in \cref{sec:theory} and corresponding discussions. 
    \item[\cref{append:first-order-approx}] presents the auxiliary results on the first-order approximation for SVD, PCA, and GLM. 
    \item[\cref{append:numerical}] presents the implementation details for numerical studies. 
    \aosversion{\item[\cref{append:proof-ident}] collects the proof for identification results.
    \item[\cref{append:proof-f}] collects the proof for the IHD procedure, i.e., \cref{thm:factor-est-full} that is a general version of \cref{thm:factor-est}.
    \item[\cref{append:proof-fstar}] collects the proof for the other non-asymptotic results, i.e. \cref{thm:z-step1} and \cref{thm:beta-est}. 
    \item[\cref{append:proof-pca}] collects the proof for auxiliary results in \cref{append:first-order-approx}.}{} 
\end{description}

See the table of notations in \cref{tb:notation}.

\begin{table}[!htb]
\footnotesize
\begin{center}
\begin{tabular}{l|l|l|l}
\hline
\hline
Notation & Meaning & Type & Definition \\
\hline
$\mathcal{E}$, $\mathcal{E}_z$, $\mathcal{E}_y$ & Set of all/with-z-label/with-y-label environments & Set & ---  \\
$\mathcal{D}_x, \mathcal{D}_z, \mathcal{D}_y$ & Observed data in $\mathcal{E}$, $\mathcal{E}_z$, $\mathcal{E}_y$& Data & --- \\
$F^{(e)}$, $F_\tti^{(e)}$, $F_\tth^{(e)}$ & Full/invariant/heterogeneous factor $\in \mathbb{R}^{r^{(e)}}$ & Random variable & \eqref{eq:model-x} \\
$B^{(e)}$, $B$, $A^{(e)}$ & Full/invariant/heterogeneous loading & Unknown quantity & \eqref{eq:model-x} \\
$r^{(e)}$, $r_\tti$, $r_\tth^{(e)}$ & Dimension of full/invariant/heterogeneous factor & Unknown quantity & \eqref{eq:model-x} \\
$S_\tti^\star, S_\tth^\star$ & Index of subset of factors $\subseteq [r^{(e)}]$ & Unknown quantity & \eqref{eq:model-invariant-y} \\
$Y^{(e)}$ & Response label $\in \mathbb{R}$ & Random variable & \eqref{eq:model-invariant-y} \\
$\beta^\star$, $\beta_\tti^\star$, $\beta_\tth^\star$ & Invariant parameter for response $\in \mathbb{R}^{|S_\star|}$ & Unknown quantity & \eqref{eq:model-invariant-y} \\
$Z^{(e)}$ & Auxiliary label $\in \mathbb{R}^q$ & Random variable & \eqref{eq:model-invariant-z} \\
$\Xi^\star$ & Invariant parameter for auxiliary label $\in \mathbb{R}^{q\times |S_\star|}$ & Unknown quantity & \eqref{eq:model-invariant-z} \\
$\sigma_{c}(\cdot)$ & GLM mean function for $c\in \{y, z\}$ & Function & \eqref{eq:model-invariant-y}, \eqref{eq:model-invariant-z} \\
\hline
$\Phi_{\tti}^{(e)}$, $\Phi_\tth^{(e)}$ & The linear projection from $X^{(e)}$ to $F_\tti^{(e)}, F_\tth^{(e)}$ & Identification &  \eqref{eq:phi-ih} \\
$\Phi_{S_\tti}^{(e)}$, $\Phi_{S_\tth}^{(e)}$ & The linear projection from $F_c^{(e)}$ to $F_{S_c^\star}^{(e)}$ for $c\in \{\tti, \tth\}$ & Identification & \eqref{eq:phi-sish} \\
$Q_c$ & The ambiguity matrix for factors $c\in \{\tti, \tth, S_\tti^\star, S_\tth^\star\}$ & Identification & \eqref{eq:phi-sish} \\
\hline
$W^{(e)}$, $W_\tti$, $W^{(e)}_\tth$ & Orthogonal basis spanning the factor loading & Unknown quantity & \eqref{eq:loading}\\
$K^{(e)}$ & Weights of basis $W^{(e)}$ to the loading $B^{(e)}$  & Unknown quantity  & \eqref{eq:loading} \\
$\hat{\Phi}_{\tti}^{(e)}$, $\hat{\Phi}_\tth^{(e)}$ & Learned DP matrix for $F_\tti, F_\tth$ & Method & \cref{algo:dp}\\ 
$\tilde{F}_c(X)$ & Estimated factors $F_c$ for $c\in \{\tti, \tth\}$ & Method & --- \\
$\hat{\Xi}_\tth$ & Estimated coefficients of $\Theta_\tti^\star$ & Method & \eqref{eq:est-xi-h} \\
$\hat{\Phi}_{S_\tti}^{(e)}$, $\hat{\Phi}_{S_\tth}^{(e)}$ & Learned DP matrix for $F_{S_\tti^\star}, F_{S_\tth^\star}$ & Method & \eqref{eq:est-si}, \eqref{eq:est-sh}\\ 
$\tilde{F}_{S_c}(X)$ & Estimated factors $F_{S_c^\star}$ for $c\in \{\tti, \tth\}$ & Method & --- \\
$\hat{\beta}, \hat{\beta}_\tti, \hat{\beta}_\tth$ & Estimation of invariant coefficients $\beta^\star$ & Method & \eqref{eq:est-y} \\
\hline
\hline
\end{tabular}
\end{center}
\caption{The semantic meanings of the notations in this paper}
\label{tb:notation}
\end{table}

\section{Further results and discussions in Section \ref{sec:ident}}
\label{append:ident-all}

\subsection{Identification of invariant and heterogeneous factors}

\label{append:ident}

We first provide an intuitive explanation of the necessity of \cref{cond:ident}. If \cref{cond:ident} (a) does not hold, let $\bar{A} \in \mathbb{R}^{d\times r'}$ with $r'\ge 1$ such that $\col(\bar{A}) = \cap_{e\in \mathcal{E}} \col([B, A^{(e)}]) \cap (\col(B))^\perp$ and without loss of generality we write $A^{(e)} = [\bar{A}, \breve{A}^{(e)}]$. In this case, we can write \eqref{eq:model-x} as
\begin{align*}
   X^{(e)} = \left(B \cdot F_{\tti}^{(e)} + \bar{A} \cdot F_{\{r_\tti+1,\ldots,r+r'\}}^{(e)} \right) + \breve{A}^{(e)} \cdot  F_{\{r_\tti+r'+1,\ldots,r^{(e)}\}}^{(e)}.
\end{align*} This means the set of distributions $\{X^{(e)}\}_{e\in \mathcal{E}}$ admits two fundamentally different latent structure decompositions: one with invariant loadings $B$ and one with invariant loadings $[B, \bar{A}]$. 

If \cref{cond:ident} (b) fails, we can write \eqref{eq:model-x} as
\begin{align*}
    X^{(e)} = B \cdot \left(F_{\tti}^{(e)} + S \cdot F_{\tth}^{(e)} \right) + \left[A^{(e)} - (B \cdot S)\right] F_{\tth}^{(e)}
\end{align*} for almost arbitrary matrix $S \in \mathbb{R}^{r_{\tti} \times r_{\tth}^{(e)}}$.

The following theorem makes the necessity argument rigorous at an instance level by the equivalence between identification and \cref{cond:ident} for \emph{any} data-generating process $P$ satisfying standard regularity conditions in \cref{cond:ident-reg}. The first part of the following theorem provides the rigorous statement of the sufficiency and necessity. The first identification result of \cref{thm:ident-1} directly follows from the second part {\sc (Linear map recovery)}. In the first part, we use different ambiguity matrices $Q_\tti^{(e)}$ in the statement (b), which illustrates that the necessity of \cref{cond:ident} is not due to the need to align $F_\tti^{(e)}$ with the same ambiguity transform matrix $Q_\tti$. 

\begin{theorem}
\label{thm:ident}
    {\sc (Sufficiency and necessity)} Let $\mathcal{L}(W)$ denote the distribution of the random variable $W$. For any latent structure $P = \{\mathcal{L}(F_{\tti}^{(e)}, F_{\tth}^{(e)}), B, A^{(e)}\}_{e\in \mathcal{E}}$ satisfying \cref{cond:ident-reg}, the two statements are equivalent. 

    \noindent (a) \underline{Satisfaction of \cref{cond:ident}}: $P$ satisfies \cref{cond:ident}.
    
    \noindent (b) \underline{Identification of latent structures}: For any $\tilde{P} = \{\mathcal{L}(\tilde{F}_{\tti}^{(e)}, \tilde{F}_{\tth}^{(e)}), \tilde{B}, \tilde{A}^{(e)}\}_{e\in \mathcal{E}}$ satisfying \cref{cond:ident-reg} and \cref{cond:ident}. If $\tilde{X}^{(e)} \overset{d}{=} X^{(e)}$ where ${X}^{(e)}$ (resp. $\tilde{X}^{(e)}$) is generated by $P$ (resp. $\tilde{P}$) via \eqref{eq:model-x} with $U^{(e)} = 0$ (resp. $\tilde{U} = 0$), then we have $r_{\tti} = \tilde{r}_{\tti}$, and for each environment $e\in \mathcal{E}$, $r_{\tth}^{(e)} = \tilde{r}_{\tth}^{(e)}$, 
    \begin{align*}
        [\tilde{F}^{(e)}_{\tti}, \tilde{F}^{(e)}_\tth] \overset{d}{=} [Q_{\tti}^{(e)} \cdot F^{(e)}_{\tti},  Q_{\tth}^{(e)} \cdot F_{\tth}^{(e)}],
    \end{align*} where both $Q_{\tti}^{(e)} \in \mathbb{R}^{r_{\tti} \times r_{\tti}}$ and $Q_{\tth}^{(e)} \in \mathbb{R}^{r_{\tth}^{(e)} \times r_{\tth}^{(e)}}$ are invertible matrices.

    \noindent {\sc (Linear map recovery)} Moreover, if $P$ satisfies \cref{cond:ident}, then the latent structure can be recovered by a linear map on $X^{(e)}$ under the noiseless setting. Particularly, suppose $X^{(e)}$ is generated by $P$ via \eqref{eq:model-x} with $U^{(e)} = 0$ for any $e\in \mathcal{E}$, then there exists a series of matrices $\{\Phi^{(e)}_{\tti}, \Phi^{(e)}_{\tth}\}_{e\in \mathcal{E}}$ with $\Phi_{\tti}^{(e)} \in \mathbb{R}^{d\times r_{\tti}}$ and $\Phi_{\tth}^{(e)} \in \mathbb{R}^{d\times r_{\tth}^{(e)}}$ determined by the covariance matrix of $X$ such that
    \begin{align*}
        \forall e\in \mathcal{E} \qquad (\Phi^{(e)}_{\tti})^\top X^{(e)} = Q_{\tti} F^{(e)}_{\tti} \qquad \text{and} \qquad (\Phi^{(e)}_{\tth})^\top X^{(e)} = Q_{\tth}^{(e)} F^{(e)}_{\tth}.
    \end{align*} where $Q_{\tti} \in \mathbb{R}^{r_{\tti} \times r_{\tti}}$ is the invertible ambiguity transform matrix that is also invariant across different environments, and $Q_{\tth}^{(e)} \in \mathbb{R}^{r_{\tth}^{(e)} \times r_{\tth}^{(e)}}$ is the invertible ambiguity transform matrix that may be different.
\end{theorem}

Here, we provide a counterexample showing that under \cref{cond:ident}, it is impossible to find a map $\Phi_\tti \in \mathbb{R}^{d\times r_\tti}$ such that $\Phi_\tti X^{(e)} \equiv Q^{(e)} F_\tti^{(e)}$ for any $e\in \mathcal{E}$ in the noiseless setting. This confirms the impossibility of ``invariant representation'' and further implies that one must use different projection matrices to obtain the aligned ``invariant representation''. 
\begin{example} Consider the setting with $\mathcal{E} = \{1, 2, 3, 4, 5\}$, $d=5$, let $w_1,\ldots, w_5$ be a set of orthogonal basis in $\mathbb{R}^5$, consider the model
\begin{align*}
    \forall e\in [5], \qquad X^{(e)} = b F_1^{(e)} + a^{(e)} F_2^{(e)} 
\end{align*} with
\begin{align*}
    \begin{bmatrix}
        a^{(1)} & \ldots & a^{(5)}
    \end{bmatrix} = \underbrace{\begin{bmatrix}
        1 & 1 & 1 & 0 & 0 \\
        1 & 0 & 1 & 1 & 0 \\
        1 & 0 & 0 & 1 & 1 \\
        0 & 1 & 0 & 0 & 1 \\
        0 & 0 & 1 & 0 & 1
    \end{bmatrix}}_M \underbrace{\begin{bmatrix}
        w_1 & \ldots & w_5
    \end{bmatrix}}_{W}
\end{align*} with invariant factor $F_1$ and heterogeneous factor $F_2$. It is easy to verify that 
\begin{align*}
    \bigcap_{e\in [5]} \mathrm{col}([b, a^{(e)}]) = \mathrm{span}\{b\}
\end{align*} and $M$ is invertible. 

We use the proof-by-contradiction argument. Suppose $\phi: \mathbb{R}^{5\times 1}$ satisfies
\begin{align*}
    \forall e\in \mathcal{E}, \qquad \phi^\top X^{(e)} = q^{(e)} F_1^{(e)}
\end{align*} with some $q^{(e)} \in \mathbb{R}$, then we must have $\phi^\top a^{(e)} = 0$ for any $e\in [5]$, this gives
\begin{align*}
    0 = \phi^\top \begin{bmatrix}
        a^{(1)} & \ldots & a^{(5)}
    \end{bmatrix} = \phi^\top M W,
\end{align*} which directly implies that $\phi = 0$ because $M$ and $W$ are both invertible matrices. So there doesn't exist some $\phi \in \mathbb{R}^{5\times 1}$ such that $\phi^\top X^{(e)} = q^{(e)} F_1^{(e)}$ for any $e\in \mathcal{E}$. 
\end{example}

\subsection{Identifying the prediction-invariant factors for response}
\label{appendix:ident-z}

We first present the regularity condition for further identifying (aligned) $F_{S_\tti^\star}^{(e)}$, $F_{S_\tth^\star}^{(e)}$ across $e\in \mathcal{E}$. The first condition imposes a standard regularity condition on the mean function $\sigma_z(\cdot)$. It is easy to check that many mean functions, like identity $\sigma(t) = t$ for regression and sigmoid $\sigma(t) = 1/(1+e^{-t})$ for classification, satisfy this. The second condition ensures the auxiliary labels $Z\supe$ contain sufficient information for identifying the prediction-invariant factors for $Y$. This condition is testable with the knowledge of $|S_\tti^\star|$ and $|S_\tth^\star|$.

\begin{condition} \label{cond:ident-z} The following holds:
\begin{itemize}
\item[(a)] \emph{(Regularity on GLM)} $\sigma_z(\cdot)$ is differentiable and $\sigma_z'(t) > 0$ for any $t\in \mathbb{R}$.
\item[(b)] \emph{(Exhaustiveness)} Recall the model \eqref{eq:model-invariant-z} with parameter $\Xi^\star = [\xi_1^\top,\ldots, \xi_q^\top]^\top \in \mathbb{R}^{q\times |S^\star|}$. Let $\Xi^\star_\tti$ (resp. $\Xi^\star_\tth$) be the first $|S^\star_\tti|$ (resp. last $|S^\star_\tth|$) columns of $\Xi^\star$, i.e., its $k$-th row represents the linear coefficients of $F_{S^\star_\tti}$ (resp. $F_{S^\star_\tth}$) of GLM on $Z_k$. We assume $\rank(\Xi^\star_\tti) = |S^\star_\tti|$ and $\rank(\Xi^\star_\tth) = |S^\star_\tth|$.
\end{itemize}
\end{condition}

Under the above conditions, we can find the map that can recover $F_{S_\tti^\star}\supe$ and $F_{S_\tth^\star}\supe$ up to the same transform matrix across environments $\mathcal{E}_z$ where the auxiliary labels are accessible. Here, we explicitly state it separately. 

\begin{theorem}
\label{thm:ident-pi}
    Under the setting of \cref{thm:ident}, assume \cref{cond:ident} holds, suppose further \cref{cond:ident-z} holds for the model \eqref{eq:model-invariant-z}, then there further exists a series of projection matrices $\Phi_{S_\tti^\star}\in \mathbb{R}^{r_\tti \times |S_\tti^\star|}$ and $\Phi_{S_\tth^\star}\supe \in \mathbb{R}^{r_\tth\supe \times |S_\tth^\star|}$ determined by the joint law of $X,Z$ such that 
    \begin{align*}
        \forall e\in \mathcal{E}_z \qquad (\Phi_{S_\tti^\star})^\top (\Phi_{\tti}\supe)^\top X\supe = Q_{S_\tti^\star} F_{S^\star_\tti}\supe ~~ \text{and} ~~ (\Phi_{S_\tth^\star}\supe)^\top (\Phi_{\tth}\supe)^\top X\supe = Q_{S_\tth^\star} F_{S^\star_\tth}\supe 
    \end{align*} where $Q_{S_\tti} \in \mathbb{R}^{|S_\tti^\star| \times |S_\tti^\star|}$ and $Q_{S_\tth} \in \mathbb{R}^{|S_\tth^\star| \times |S_\tth^\star|}$ are the same invertible matrices across $e\in \mathcal{E}_z$.
\end{theorem}

\subsection{Regularity conditions and discussions for Proposition \ref{prop:robust2}}
\label{append:worst2}

We first present the regularity condition.
\begin{condition}
\label{cond:fsi-worst-reg}
The following conditions hold
\begin{itemize}
    \item[(1)] $\nu$ has finite second-order moments, $\mathbb{E}_{\bar{F} \sim \nu}[\bar{F} \bar{F}^\top] < \infty$.
    \item[(2)] $\sigma_y(\cdot)$ has uniformly bounded first-order derivative $\|\sigma_y'\|_\infty < \infty$.
    \item[(3)] $c \ge \tilde{C}$ for some constant $\tilde{C}$ depending on $\nu, \sigma_y, \|\beta^\star\|_2$.
    \item[(4)] For any $\beta$ that is linearly independent of $\bar{\beta}$, i.e., $\beta \in \mathbb{R}^{r_\tti + r_\tth} \setminus \{t\cdot \bar{\beta}: t\neq 0\}$, we have
    \begin{align*}
        \mathbb{E}_\nu \left[\mathrm{Var}_\nu(\bar{\beta}^\top \bar{F}|\beta^\top \bar{F})\right] > 0.
    \end{align*}
\end{itemize}
\end{condition}

Here we comment on the conditions in \cref{cond:fsi-worst-reg}. The first two conditions are standard and hold for many canonical models like linear regression $\sigma_y(t) = t$ and logistic regression $\sigma_y(t) = e^t/(1+e^t)$. The condition (3) is imposed to make sure the uncertain set $\mathcal{U}$ has full ambiguity on the orthogonal transform, that is, for any $O\in \mathcal{O}_{r_\tth}$, we can find a potential distribution of $(F_\tti, F_\tth, Y) \sim \mu \in \mathcal{U}$ such that $(F_\tti, F_\tth) \overset{d}{=} (F_\tti, O \bar{F}_\tth)$.

The last condition (4) is for uniqueness: if it is violated, we can find a $\beta \in \mathbb{R}^d$ such that
\begin{align*}
    \bar{\beta}^\top \bar{F} = \beta^\top \bar{F} \qquad \nu\text{-a.s.}
\end{align*} such that $\beta$ also minimizes the loss because their predictions are the same. This condition can be easily satisfied, for example, when $\bar{F}$ is a continuous random vector such that the corresponding probabilistic density function exists. Without this condition, we can have the same result except for replacing ``$=$'' by ``$\subseteq$'' in \eqref{eq:worst2}; see the justifications in the proof.

\section{Main and intermediate results in Section \ref{sec:theory}}
\label{append:factor}

\subsection{Invariant and heterogeneous factors}
\label{appendix:factor-est}

We first define some notations. For a matrix $A\in \mathbb{R}^{d_1\times d_2}$ with full rank $d_2$, we denote $\psf{A} = A(A^\top A)^{-1}A^\top \in \mathbb{R}^{d_1 \times d_1}$ as the projection matrix that can project any vector $x$ to the column space of $A$. In particular, we have $\psf{A}= AA^\top$ for $A \in \mathcal{O}_{d_1 \times d_2}$. For a rank $r$ matrix $A\in \mathbb{R}^{d_1\times d_2}$, we let $A = U_A \Sigma_A V_A^\top$ be the $r$-truncated SVD decomposition of $A$ such that $U_A \in \mathbb{R}^{d_1\times r}$, $V_A\in \mathbb{R}^{d_2\times r}$, $\Sigma_A \in \mathbb{R}^{r\times r}$, we denote $\rsf{A} = U_A (\Sigma_A)^{-1} V_A^\top$. 

Denote
\begin{align}
\label{eq:snr-factor}
    \lambda^{(e)} = \frac{\lambda_{\min}\left((B^{(e)})^\top B^{(e)}\right)}{\|\Sigma_U^{(e)}\|_2} \qquad \text{and} \qquad \bar{\lambda} = \left[\frac{1}{|\mathcal{E}|} \sum_{e\in \mathcal{E}} \frac{1}{\lambda\supe} \right]^{-1}.
\end{align}

We are working in the case where $0<\lambda^{(e)} < \infty$, i.e., $B^{(e)}$ has full rank $r^{(e)}$. Recall that $W^{(e)} \in \mathcal{O}_{d\times r^{(e)}}$ satisfies $\col(W^{(e)}) = \col(B^{(e)})$, $W_\tti \in \mathcal{O}_{d\times r_\tti}$ satisfies $\col(W_\tti) = \cap_{e\in \mathcal{E}} \col(B^{(e)}) = \col(B)$ by \cref{cond:separate}, and $W_\tth^{(e)} \in \mathcal{O}_{d\times r_\tth^{(e)}}$ satisfies $\col(W_\tth^{(e)}) = \col(B^{(e)}) \cap \col(W_\tti)_{\perp}$. In this case, we can follow the decomposition in \eqref{eq:loading}:
\begin{align*}
    B\supe := \begin{bmatrix} B & A\supe \end{bmatrix} =  \begin{bmatrix}
        W_{\tti} & W_{\tth}\supe 
    \end{bmatrix} \begin{bmatrix}
        \kii & \khie \\
         0 & \khhe
    \end{bmatrix} = W\supe \cdot K\supe.
\end{align*}

We first define the factor loading estimation error in each environment $e\in \mathcal{E}$.
\begin{align}
\label{eq:delta-w-e-t}
    \delta_{W}^{(e)}(t) := \sqrt{\frac{d + \log(|\mathcal{E}|)+t}{n_x \cdot \snre}} + \frac{1}{\snre},
\end{align} and the factor loading estimation error averaged across different environments $e\in \mathcal{E}$,
\begin{align}
\label{eq:delta-wi}
    \delta_{W_{\tti}}(t) := \sqrt{\frac{d + t}{ |\mathcal{E}| \cdot n_x \cdot \snrbar}} + \frac{1}{\snrbar} \left[1 + \frac{d + \log(4|\mathcal{E}|) + t}{n_x}\right].
\end{align} 

\subsubsection{Step 1. Factor loading space estimation in each environment}

Throughout this step, we will adopt the notation
\begin{align*}
    \mbxe = \begin{bmatrix}
        X_1^{(e)} & \cdots  & X_{n_x}^{(e)}
    \end{bmatrix}^\top \qquad \mble = \begin{bmatrix}
        B^{(e)} F_1^{(e)} &\cdots &B^{(e)} F_{n_x}^{(e)}
    \end{bmatrix}^\top \qquad \mbue = \begin{bmatrix}
        U_1^{(e)} & \cdots & U_{n_x}^{(e)}\end{bmatrix}^\top
\end{align*} with $\mbxe, \mble, \mbue \in \mathbb{R}^{n_x \times d}$. We have the simple low-rank plus noise model $\mbxe = \mble + \mbue$, and will omit the dependence on $e$ if it is clear from the context. 

Denote the empirical covariance matrix for $F$ as $\hat{\Sigma}_F^{(e)} = \frac{1}{n_x} \sum_{i=1}^{n_x} F_i^{(e)} (F_{i}^{(e)})^\top$. The following lemma establishes two standard high-probability events.
\begin{lemma}
\label{lemma:concentrate-1}
    Under \cref{cond:x-reg}, there exists a constant $\tilde{C}_1=\mathrm{poly}(c_1)$ such that the following holds. For any $t\ge 1$ such that $n_x \ge (r^{(e)} + t)$, the following events:
    \begin{align}
        \label{eq:u-sigma-bound}
        \begin{split}
        \mathcal{A}_{1,F}^{(e)}(t) &:= \left\{\left\|\hat{\Sigma}_F^{(e)} - \Sigma_F^{(e)}\right\|_2 \le \tilde{C}_1 \sqrt{\frac{r^{(e)} + t}{n_x}} \right\}\\
        \mathcal{A}_{1,U}^{(e)}(t) &:= \left\{\|\mbu^{(e)}\|_2 \le \tilde{C}_1 \sqrt{\|\Sigma_U^{(e)}\|_2 }\sqrt{d + n_x + t} \right\}\\
        \end{split}
    \end{align}
    satisfy $\min\{\mathbb{P}[\mathcal{A}_{1,F}^{(e)}(t)], \mathbb{P}[\mathcal{A}_{1,U}^{(e)}(t)]\} \ge 1-0.25 e^{-t}$.
\end{lemma}

\aosversion{\begin{proof}[Proof of \cref{lemma:concentrate-1}] See \cref{sec:proof-concentration}.\end{proof}}{}

Based on the high-probability events in \cref{lemma:concentrate-1}, we can establish \cref{prop:first-order-approx-e}, a first-order approximation of $\hat{W}^{(e)}$. The proof is based on a tight first-order approximation of the SVD solution \cref{thm:first-order-svd} with the low-rank component $M_\star \gets \mble$ and noise component $E \gets \mbue$. The result is stated for a given fixed $e\in \mathcal{E}$, and we omit the dependency on $\supe$ for ease of presentation.

\begin{proposition}
\label{prop:first-order-approx-e}
    Under \cref{cond:x-reg}, there exists a constant $\tilde{C}_1=\mathrm{poly}(c_1)$ such that the following holds. For each fixed environment $e$, let $\hat{V} = \hat{W}^{(e)}\in \mathbb{R}^{d\times r^{(e)}}$ be the top-$r^{(e)}$ eigenvectors of $[\mbxe]^\top \mbxe$, and $V_\star = W^{(e)} \in \mathcal{O}_{d\times r^{(e)}}$ such that $\col(V_\star) = \col(B^{(e)}) = \col(W^{(e)})$, $V_\perp \in \mathcal{O}_{d\times(d-r^{(e)})}$ be the orthogonal basis matrix for the complement of $\col(V_\star)$ in $\mathbb{R}^d$, i.e., $\col(V_\perp) = \col(V_\star)^\perp$. Let $t \ge 1$ be arbitrary such that $n_x \ge \tilde{C}_1 (r^{(e)} + t)$. Under the event $\mathcal{A}_{1,F}^{(e)}(t)$, we have $\nu_{r^{(e)}}(\mbl^{(e)}) \ge \sqrt{\snre \|\Sigma_U^{(e)}\|_2 \cdot n_x / 2c_1}$. Suppose further that $\delta_W^{(e)}(t) \le 1/\tilde{C}_1$ and  $\mathcal{A}_{1,U}^{(e)}(t)$ occurs, then
    \begin{align}
        \label{eq:first-order-appro-e}
        \left\|\hat{V} \hat{V}^\top - V_\star V_\star^\top - G - G^\top \right\|_2 &\le \tilde{C}_1 \cdot \frac{d + n_x + t}{n_x \cdot \snre} \qquad \text{where} \qquad G = \psf{V_{\perp}} [\mbue]^\top \rsf{\mble}
    \end{align}
\end{proposition}

\aosversion{\begin{proof}[Proof of \cref{prop:first-order-approx-e}] see \cref{sec:proof-first-order-approx-e}. \end{proof}}{}

Now we are ready to prove \cref{lemma:preliminary}, we first state it with the explicit dependence on $t$. \cref{lemma:preliminary} directly follows from \cref{lemma:factor-loading-appendix1} by setting $t = 100\log(n_x)$. 

\begin{lemma}
\label{lemma:factor-loading-appendix1}
    Under \cref{cond:x-reg}, there exists some constant $\tilde{C}_1 = \poly(c_1)$ such that if $n_x\ge \tilde{C}_1 \cdot [\sup_{e\in \mathcal{E}} r^{(e)} + \log(|\mathcal{E}|)+ t]$, then for simultaneously all the environment $e\in \mathcal{E}$, 
    \begin{align}
    \label{eq:loading-error-append}
        \left\|\sin \Theta(\hat{W}^{(e)}, W^{(e)}) \right\| \big/ \tilde{C}_1 \le \sqrt{\frac{d + \log(|\mathcal{E}|) + t}{\snre \cdot n_x}} + (\snre)^{-1}
    \end{align} occurs with probability at least $1-e^{-t}$. Moreover, as long as the R.H.S. of \eqref{eq:loading-error-append} $\le 1/\tilde{C}_1^2$, under the same event, if we use $\hat{\Phi}^{(e)} = \frac{1}{\sqrt{d}} \cdot \hat{W}^{(e)}$ as the diversified projection matrix for $F^{(e)}$, then
    \begin{align}
    \label{eq:factor-error-append}
        \inf_{Q^{(e)} \in \mathbb{R}^{r^{(e)}\times r^{(e)}}, \text{invertible}} \mathbb{E} \left[ \left\| Q^{(e)} (\hat{\Phi}^{(e)})^\top X^{(e)} - F^{(e)} \right\|_2^2 \mid \mathcal{D}_{x,1}^{(e)} \right] \le \frac{4 r^{(e)}}{\snre}.
    \end{align} where $\mathcal{D}_{x,1}^{(e)} = \{X_{i}^{(e)}\}_{i=1}^{n_x}$.
\end{lemma}

\begin{proof}[Proof of \cref{lemma:factor-loading-appendix1}] 
It suffices to prove \eqref{eq:loading-error-append} and \eqref{eq:factor-error-append} for any given fixed $e \in \mathcal{E}$ and then apply the union bound. 

Now given fixed $e\in \mathcal{E}$, the proof proceeds conditioned on the events defined in \cref{lemma:concentrate-1} occurring, under which the first-order approximation in \cref{prop:first-order-approx-e} holds provided $\delta^{(e)}_W(t) \le \tilde{C}_1^{-1}$.  It follows from \cref{lemma:concentrate-2} (applying \cref{lemma:concentrate-2} with $\mathcal{E} = \{e\}$) that $\|G^{(e)}\|_2 \le C_1 \sqrt{(d + \log(|\mathcal{E}|) + t)/ (n_x \cdot \lambda^{(e)})}$ for some $C_1 = \mathrm{poly}(c_1)$ with probability at least $1-e^{-t}$. This concludes the proof of \eqref{eq:loading-error-append} by absorbing the second-order error $\frac{d+t+\log(|\mathcal{E}|)}{n_x\lambda^{(e)}}$ in \eqref{eq:first-order-appro-e} into the first-order error $\sqrt{\frac{d+t+\log(|\mathcal{E}|)}{n_x\lambda^{(e)}}}$. 

For \eqref{eq:factor-error-append}, we pick $Q^{(e)} = [(1/\sqrt{d}) [\hat{W}^{(e)}]^\top B^{(e)}]^{-1}$. As long as $\|\sin\Theta(\hat{W}^{(e)}, W^{(e)})\|_2 \le 1/2$, the smallest singular value of $(1/\sqrt{d}) [\hat{W}^{(e)}]^\top B^{(e)} = (1/\sqrt{d}) [\hat{W}^{(e)}]^\top W^{(e)} K^{(e)}$ is lower bounded by 
\begin{align*}
\frac{1}{2} \cdot \left\{\sqrt{\lambda_{\min}[(B^{(e)})^\top B^{(e)}]} / \sqrt{d} \right\},
\end{align*} thus its inverse is well defined and we have $\|Q^{(e)}\|_2 \le 2 \sqrt{d} / \sqrt{\lambda_{\min}[(B^{(e)})^\top B^{(e)}]}$. Further substituting the model $X^{(e)} = B^{(e)} F^{(e)} + U^{(e)}$, we have, for any new observation of $X^{(e)}_0, F^{(e)}_0$, conditioned on $\mathcal{D}_{x,1}^{(e)}$ 
\begin{align*}
    \mathbb{E}_{X^{(e)}_0, F^{(e)}_0} \left[ \left\| Q^{(e)} (\hat{\Phi}^{(e)})^\top X^{(e)} - F^{(e)} \right\|_2^2 \right] &= \mathrm{Tr}(\frac{1}{d} Q^{(e)} (\hat{W}^{(e)})^\top \Sigma_U^{(e)} \hat{W}^{(e)} (Q^{(e)})^\top) \\
    &\le r^{(e)} \cdot \left\|Q^{(e)} / \sqrt{d} \right\|_2^2 \cdot \left\|\Sigma_U^{(e)} \right\|_2 \\
    &\le \frac{4r^{(e)} \|\Sigma_U^{(e)}\|_2}{\lambda_{\min}[(B^{(e)})^\top B^{(e)}]} = \frac{4r^{(e)}}{\lambda^{(e)}}.
\end{align*}
\end{proof}

\subsubsection{Step 2. Invariant and heterogeneous loading estimation}

Denote $W_{\perp}^{(e)} := (W^{(e)})_{\perp}$, the first-order term $G^{(e)} := \psf{W_{\perp}^{(e)}} [\mbue]^\top \rsf{\mble}$ in \cref{prop:first-order-approx-e}, and the second-order residual of the projection matrix in each environment 
\begin{align*}
    \hat{\Delta}_{P,2}^{(e)} := \hat{W}^{(e)} (\hat{W}^{(e)})^\top - {W}^{(e)} ({W}^{(e)})^\top - G^{(e)} - (G^{(e)})^\top.
\end{align*} Also define $\hat{\Delta}_P := \frac{1}{|\mathcal{E}|} \sum_{e\in \mathcal{E}} \hat{W}^{(e)} (\hat{W}^{(e)})^\top - \frac{1}{|\mathcal{E}|} \sum_{e\in \mathcal{E}} W^{(e)} (W^{(e)})^\top$. We can decompose $\hat{\Delta}_P$ into 
\begin{align}
\label{eq:first-order-approx-p}
    \hat{\Delta}_P = \underbrace{\frac{1}{|\mathcal{E}|} \sum_{e\in \mathcal{E}} \hat{\Delta}_{P,2}^{(e)}}_{=:\hat{\Delta}_{P,2}} + \underbrace{\frac{1}{|\mathcal{E}|} \sum_{e\in \mathcal{E}} G^{(e)}}_{=:\hat{\Delta}_{P,1}} + \underbrace{\frac{1}{|\mathcal{E}|} \sum_{e\in \mathcal{E}} (G^{(e)})^\top}_{=:\hat{\Delta}_{P,1}^\top}.
\end{align}

The following lemma combines the concentration results in \cref{lemma:concentrate-1} and also establishes concentration on $\hat{\Delta}_{P, 1}$. 
\begin{lemma}\label{lemma:concentrate-2}
Under \cref{cond:x-reg}, there exists a large enough constant $\tilde{C}_1 = \poly(c_1)$ such that the following holds. Define the event
\begin{align}
\begin{split}
    \mathcal{A}_{2,\tti}(t) := & \left[\bigcap_{e\in \mathcal{E}} \mathcal{A}_{1,F}^{(e)}(t + \log(4|\mathcal{E}|)) \cap \mathcal{A}_{1,U}^{(e)}(t + \log(4|\mathcal{E}|)) \right] \\
    &\qquad \cap \left\{\left\|\hat{\Delta}_{P,1} \right\|_2 \le \tilde{C}_1 \sqrt{\frac{d + t}{|\mathcal{E}| \cdot n_x \cdot \snrbar}} \right\} \\
    &\qquad \cap \left\{\left\|\hat{\Delta}_{P,2}\right\|_2 \le \tilde{C}_1 \frac{1}{\snrbar} \left[1 + \frac{d + \log(4|\mathcal{E}|) + t}{n_x}\right] \right\}
\end{split}
\label{eq:event-2i}
\end{align}
If $\max\left\{\frac{t + \sup_{e\in \mathcal{E}}r^{(e)} + \log(|\mathcal{E}|)}{n_x}, \sup_{e\in \mathcal{E}} \delta_W^{(e)}(t + \log(4|\mathcal{E}|)) \right\}\le 1/ \tilde{C}_1$, then $\mathbb{P}\left[\mathcal{A}_{2,\tti}(t)\right] \ge 1-e^{-t}$. Under event $\mathcal{A}_{2,\tti}(t)$ in \eqref{eq:event-2i}, this immediately yields that
\begin{align}
\label{eq:first-order-approx-p-error}
    \|\hat{\Delta}_P\|_2 \le \tilde{C}_1 \cdot \delta_{W_\tti}(t)
\end{align} where $\delta_{W_\tti}(t)$ is defined in \eqref{eq:delta-wi}.
\end{lemma}
\aosversion{\begin{proof}[Proof of \cref{lemma:concentrate-2}] See \cref{sec:proof-concentration}. \end{proof}}{}

We are ready to provide error bounds under the above event $\mathcal{A}_{2,\tti}(t)$ in \eqref{eq:event-2i}. \cref{prop:loading-est} immediately follows from the bounds \eqref{eq:wihat-wi}, \eqref{eq:whhat-wh} in \cref{prop:wi-wh} below by specifying $t = 10 \log(n_x)$. It is worth noting that the error bounds in \eqref{eq:whhat-wi} and \eqref{eq:wihat-wh} require additional high-probability events. We first introduce some quantities of interest: for any $t\ge 1$, denote
\begin{align}
    \delta_{\tti, 2} &= \frac{1}{\epsilon_A \cdot \snrbar} \left[1 + \frac{d + \log(|\mathcal{E}|) + t}{\{1 \land (\epsilon_A \cdot |\mathcal{E}|)\} \cdot n_x}\right] \asymp \left[\frac{\delta_{W_\tti}(t)}{\epsilon_A}\right]^2 + \frac{1}{\epsilon_A} \left[\frac{1}{|\mathcal{E}|} \sum_{e\in \mathcal{E}} \delta^{(e)}_W(t)\right]^2 \label{eq:deltai2}\\
    \delta_{\tti, 1}^{(e)} &= \frac{1}{\epsilon_A} \sqrt{\frac{r^{(e)} + \log(|\mathcal{E}|) + t}{|\mathcal{E}| \cdot n_x \cdot \snrbar}} \\
    \delta_{\tth, 2}^{(e)} &= \frac{1}{\snre} \left[1 + \frac{d + \log(|\mathcal{E}|) + t}{ n_x}\right] + \delta_{\tti, 2}\\
    \delta_{\tth, 1}^{(e)} &= \sqrt{\frac{r^{(e)} + \log(|\mathcal{E}|) + t}{\snre \cdot n_x}} + \delta_{\tti, 1}^{(e)} \label{eq:deltah1}
\end{align}
These errors are of high order. When $\epsilon_A \asymp 1$, $\lambda^{(e)} \asymp \lambda$ and $\delta_W^{(e)}(t) \asymp \delta_W(t)$, we have
\begin{align*}
    \delta_{\tti, 2} \asymp \delta_{\tth, 2}^{(e)} \asymp [\delta_W(t)]^2 \qquad \text{and} \qquad \delta_{\tti, 1}^{(e)} \asymp \delta_{\tth, 1}^{(e)} \asymp \delta_W(t) \cdot \sqrt\frac{r^{(e)}}{d}
\end{align*}

\begin{proposition} 
\label{prop:wi-wh}
    Suppose \cref{cond:x-reg} and \cref{cond:separate} hold, there exists a large enough constant $\tilde{C}_1 = \poly(c_1)$ such that the following holds. If
    \begin{align*}
        \max\left\{\frac{t + \sup_{e\in \mathcal{E}}r^{(e)} + \log(|\mathcal{E}|)}{n_x}, \sup_{e\in \mathcal{E}} \delta_W^{(e)}(t+\log(4|\mathcal{E}|)), \delta_{W_\tti}(t) / \epsilon_A \right\}\le 1/ (10\tilde{C}_1),
    \end{align*} and we pick $ \lambda_{\mathtt{ihd}} \in [\tilde{C}_1 \delta_{W_{\tti}}(t), \epsilon_A - \tilde{C}_1 \delta_{W_{\tti}}(t)]$, then the following event $\mathcal{A}_3(t)$ holds with probability at least $1-3e^{-t}$. Here $\mathcal{A}_3(t)$ is defined as the event where we have $\hat{r}_\tti = r_\tti, \hat{r}_\tth^{(e)} = r_\tth^{(e)}$, and all of these error bounds hold
    \begin{align}
        &\frac{1}{\tilde{C}_1}\left\|\sin\Theta(\hat{W}_{\tti}, W_{\tti})\right\|_2 \le \frac{\delta_{W_\tti}(t)}{\epsilon_A}, \label{eq:wihat-wi}\\
        \forall e\in \mathcal{E}, \qquad &\frac{1}{\tilde{C}_1}\left\|\sin\Theta(\hat{W}_{\tth}^{(e)}, {W}_{\tth}^{(e)})\right\|_2 \le \frac{\delta_{W_\tti}(t)}{\epsilon_A} + \delta_{W}^{(e)}(t + \log(4|\mathcal{E}|)), \label{eq:whhat-wh}\\
        \forall e\in \mathcal{E}, \qquad &\frac{1}{\tilde{C}_1}\left\| \hat{W}_{\tti}^\top {W}_{\tth}^{(e)}\right\|_2 \le \delta_{\tti, 1}^{(e)} + \delta_{\tti, 2}, \label{eq:wihat-wh}\\
        \forall e\in \mathcal{E}, \qquad &\frac{1}{\tilde{C}_1}\left\| {W}_{\tti}^\top \hat{W}_{\tth}^{(e)}\right\|_2 \le \delta_{\tth, 2}^{(e)} + \delta_{\tth, 1}^{(e)} 
        \label{eq:whhat-wi}, \\
        & \frac{1}{\tilde{C}_1}\left\|W_\tti^\top \hat{W}_\tti \hat{W}_\tti^\top W_\tti - I_{r_\tti} \right\|_2 \le \delta_{\tti, 2} \label{eq:wi-wi}\\
        \forall e\in \mathcal{E}, \qquad & \frac{1}{\tilde{C}_1}\left\|(W_\tth^{(e)})^\top \hat{W}_\tth^{(e)} (\hat{W}_\tth^{(e)})^\top W_\tth^{(e)} - I_{r_\tth^{(e)}} \right\|_2 \le \delta_{\tth, 2}^{(e)} \label{eq:wh-wh}
    \end{align} and the all the R.H.S. in \eqref{eq:wihat-wi} -- \eqref{eq:wh-wh} are smaller than $1/(4\tilde{C}_1)$.
\end{proposition}
\aosversion{\begin{proof}[Proof of \cref{prop:wi-wh}] See \cref{sec:proof-wi-wh}.\end{proof}}{}

\subsubsection{Step 3. Invariant and heterogeneous factors estimation error}

We additionally define the following quantities,
\begin{align}
\label{eq:hatk-def}
\begin{split}
    \hatkii &= \hat{W}_\tti^\top B  \qquad\qquad \hatkihe = \hat{W}_\tti^\top A^{(e)} \\
    \hatkhie &= (\hat{W}_\tth^{(e)})^\top B  \qquad ~\hatkhhe = (\hat{W}_\tth^{(e)})^\top A^{(e)} \\
\end{split}
\end{align} 
The first lemma is deterministic and wraps the result in \cref{prop:wi-wh} for the proof of the main theorem: it casts the first-order error in \cref{prop:wi-wh} to the first-order approximation of compositions of the quantities in \eqref{eq:hatk-def}, with approximation errors that are of high order in \eqref{eq:deltai2} -- \eqref{eq:deltah1}.

\begin{lemma}
\label{lemma:hatk}
    Under the setting of \cref{prop:wi-wh} and the event $\mathcal{A}_{3}(t)$ defined therein, suppose further $8\sqrt{\kappa^{(e)}} \cdot \|{W}_\tti^\top \hat{W}_\tth^{(e)}\|_2 \le 1$, then $\hatkii$, $\hatkhhe$ are invertible, and there exist invertible matrices $\hat{Q}_{1,1}$ and $\hat{Q}_{1,2}^{(e)}$ with singular values in $[1/2, 2]$ such that 
    \begin{align}
    \label{eq:k-inverse}
        \hatkii = \hat{Q}_{1,1} K_\tti \qquad \text{and} \qquad  \hatkhhe = \hat{Q}_{1,2}^{(e)} \khhe 
    \end{align} and
    \begin{align}
    \label{eq:delta-i-i}
        \hatkii^{-1} \hatkii^{-\top} &= \kii^{-1} (I_{r_\tti} + \hat{\Delta}_{1,1}) \kii^{-\top } \qquad ~~~~~ \text{with} ~~ \|\hat{\Delta}_{1,1}\|_2 \le \tilde{C}_1 \delta_{\tti, 2} \\
    \label{eq:delta-h-h}
        (\hatkhhe)^{-1} (\hatkhhe)^{-\top} &= (\khhe)^{-1} (I_{r_{\tth}^{(e)}} + \hat{\Delta}_{1,2}^{(e)}) (\khhe)^{-\top }   ~ \text{with} ~~ \|\hat{\Delta}_{1,2}^{(e)}\|_2 \le \tilde{C}_1 \sqrt{\kappa^{(e)}} \delta_\tth^{(e)} \\
    \label{eq:delta-h-i}
        (\hatkhhe)^{-1} \hatkhie &= (\khhe)^{-1} \hat{\Delta}_{1,3}^{(e)} \kii \qquad\qquad\qquad ~~ \text{with} ~~ \|\hat{\Delta}_{1,3}^{(e)}\|_2 \le \tilde{C}_1 \delta_\tth^{(e)}\\
    \label{eq:delta-i-h}
        (\hatkii)^{-1} \hatkihe &= \kii^{-1} \khie + \kii^{-1} \hat{\Delta}_{1,4}^{(e)} \khhe \qquad   \text{with} ~~ \|\hat{\Delta}_{1,4}^{(e)}\|_2 \le \tilde{C}_1 (\delta_{\tti, 2} + \delta_{\tti, 1}^{(e)})
    \end{align} where $\delta_\tth^{(e)} = \delta_{\tth,1}^{(e)} + \delta_{\tth,2}^{(e)}$, the constant $\tilde{C}_1=8C_1$, with $C_1$ being the constant $\tilde{C}_1$ in \cref{prop:wi-wh}.
\end{lemma}
\aosversion{\begin{proof}[Proof of \cref{lemma:hatk}] See \cref{sec:proof-lemma:hatk}.\end{proof}}{}

Recall the sample split $\mathcal{D}_1 \cup \mathcal{D}_2$ in \cref{algo:dp}. We also need the following concentration result showing that for $\hat{W}_\tti$ and $\hat{W}_{\tth}^{(e)}$ estimated using $\mathcal{D}_1$, the spectral norm for the multiplication of $(\hat{W}_\tti, \hat{W}_{\tth}^{(e)})$ and the different covariance structures estimated from $\mathcal{D}_2$ are bounded. The concentration is done on the randomness of $\mathcal{D}_2$, thus are standard given $\hat{W}_\tti$ and $\hat{W}_\tth^{(e)}$ are regarded as fixed. For any random vector $Z, Z' \in \{F_\tti, F_\tth, U\}$, we use the notation 
\begin{align*}
    \hat{\Sigma}_{Z,Z'}^{\dagger,(e)} = \frac{1}{n_x} \sum_{i=n_x+1}^{2n_x} Z_{i}^{(e)} {({Z_i'}^{(e)})}^\top
\end{align*} with the superscript $\dagger$ to emphasize its dependency on $\mathcal{D}_2$, and abbreviate $Z'$ when $Z'=Z$. 

\begin{lemma}
\label{lemma:concentrate-3}
Under \cref{cond:x-reg} and \cref{cond:separate}, there exists some constant $\tilde{C}_1=\mathrm{poly}(c_1)$ such that assume further that $n_x \ge \tilde{C}_1 \left[r^{(e)} + \log(|\mathcal{E}|) + t\right]$ for any $e\in \mathcal{E}$, denote 
\begin{align}
\label{eq:delta-sigma}
\begin{split}
    \delta_{\Sigma}^{(e)} = \max\Bigg\{&\frac{\left\|(\hat{W}_\tth^{(e)})^\top (\hat{\Sigma}_U^{\dagger, (e)} - \Sigma_U^{(e)}) (\hat{W}_\tth^{(e)}) \right\|_2}{\|\Sigma_U^{(e)}\|_2},  \frac{\left\|(\hat{W}_\tth^{(e)})^\top (\hat{\Sigma}_U^{\dagger, (e)} - \Sigma_U^{(e)}) \hat{W}_\tti\right\|_2}{{\|\Sigma_U^{(e)}\|_2}}, \\ 
    &\frac{\max\left(\left\|(\hat{W}_{\tth}^{(e)})^\top \hat{\Sigma}_{U, F_{\tti}}^{\dagger, (e)}\right\|_2, \left\|(\hat{W}_{\tth}^{(e)})^\top \hat{\Sigma}_{U, F_{\tth}}^{\dagger, (e)}\right\|_2, \left\|\hat{W}_{\tti}^\top \hat{\Sigma}_{U, F_{\tti}}^{\dagger, (e)}\right\|_2, \left\|\hat{W}_{\tti}^\top \hat{\Sigma}_{U, F_{\tth}}^{\dagger, (e)}\right\|_2\right)}{(\|\Sigma_U^{(e)}\|_2)^{1/2}}, \\
    & \frac{\left\|\hat{W}_\tti^\top (\hat{\Sigma}_U^{\dagger, (e)} - \Sigma_U^{(e)}) \hat{W}_\tti\right\|_2}{{\|\Sigma_U^{(e)}\|_2}}, \left\|\hat{\Sigma}_{F_\tti, F_\tth}^{\dagger, (e)} \right\|_2, \left\|\hat{\Sigma}_{F_\tti}^{\dagger, (e)} - \Sigma_{F_\tti}^{(e)}\right\|_2, \left\|\hat{\Sigma}^{\dagger, (e)}_{F_\tth} - \Sigma_{F_\tth}^{(e)}\right\|_2\Bigg\}.
\end{split}
\end{align} The following event
\begin{align}
    \mathcal{A}_4(t):= \left\{\forall e\in \mathcal{E}, ~~ \delta_\Sigma^{(e)} \le \tilde{C}_1 \sqrt{\frac{r^{(e)} + \log(|\mathcal{E}|) + t}{n_x}} \le \frac{1}{5c_1^2} \right\}
\end{align} occurs with probability at least $1-e^{-t}$ for any $t\ge 1$. 
\end{lemma}

\aosversion{\begin{proof}[Proof of \cref{lemma:concentrate-3}] See \cref{sec:proof-concentration}.\end{proof}}{}

The following deterministic result casts the factor estimation error in a single environment $e\in \mathcal{E}$ to the loading estimation errors and the covariance estimation errors. Define
\begin{align}
    \lambda_{A\setminus B}^{(e)} &= \frac{\lambda_{\min}\left[(\khhe)^\top \khhe\right]}{\|\Sigma_U^{(e)}\|_2}, \qquad  \qquad \lambda_{B\setminus A}^{(e)}= \frac{\lambda_{\min}\left[B^\top(I - \mathsf{P}_{A^{(e)}})B\right]}{\|\Sigma_U^{(e)}\|_2} \label{eq:lda-ab} 
\end{align}

\begin{proposition}
\label{prop:loading-to-factor-i}
Under the assumptions in \cref{prop:wi-wh}, \cref{lemma:concentrate-3} and \cref{lemma:hatk}, and the events $\mathcal{A}_3(t)$ and $\mathcal{A}_{4}(t)$ defined by \cref{prop:wi-wh} and \cref{lemma:concentrate-3}. There exists some constant $\tilde{C}_1 = \poly(c_1)$ such that for any environment $e\in \mathcal{E}$, if we have $$\lambda^{(e)} \ge 100 \qquad \text{and} \qquad \|\hat{\Delta}_{1,3}^{(e)} \|_2 \lor \|\hat{\Delta}_{1,4}^{(e)} \|_2 
\lor [(\snrba^{(e)})^{-1/2} \delta_\Sigma]\le (5c_1 \sqrt{\kappa^{(e)}})^{-1},$$ then the following holds

\noindent (1) There exists some $Q_\tth^{(e)} \in \mathbb{R}^{r_\tth^{(e)} \times r_\tth^{(e)}}$ with $1/(2\sqrt{c_1}) \le \nu_{\min}[Q_\tth^{(e)}] \le \nu_{\max}[Q_\tth^{(e)}] \le 2\sqrt{c_1}$ such that $Q_\tth^{(e)} (\hat{\Phi}_\tth^{(e)})^\top X^{(e)} - F_\tth^{(e)}$ is jointly sub-Gaussian with parameter $$\tilde{C}_1 \left[\sqrt{\kappa^{(e)}} (\delta_{\tth,2}^{(e)} + \delta_{\tth, 1}^{(e)} ) + \sqrt{1 / \lambda_{A\setminus B}^{(e)}}\right].$$

\noindent (2) There exists a shared $Q_{\tti} \in \mathbb{R}^{r_\tti \times r_\tti}$ across $e\in \mathcal{E}$ such that
$Q_{\tti} (\hat{\Phi}_{\tti}^{(e)})^\top X^{(e)} - F_{\tti}^{(e)}$ is jointly sub-Gaussian with parameter
\begin{align}
\begin{split}
g := \tilde{C}_1 & \Bigg[\sqrt{\frac{1}{\snrba^{(e)}}} + \delta_\Sigma^{(e)} + \frac{1}{\lambda^{(e)}}  + \sum_{i=1}^4 \|\hat{\Delta}_{1,i}^{(e)}\|_2 + \sqrt{\kappa^{(e)}} \|\hat{\Delta}_{1,3}^{(e)}\|_2 \\
&\qquad \qquad \qquad \qquad + \delta_\Sigma^{(e)} \sqrt{\kappa^{(e)}} \left\{\frac{1}{\sqrt{\snrba^{(e)}}} + \sqrt{\kappa^{(e)}} \|\hat{\Delta}_{1,3}^{(e)}\|_2 \right\} \Bigg],
\end{split}\label{eq:factor-i-error}
\end{align} suppose further that $g \le 1/\tilde{C}_1$, then $1/(2\sqrt{c_1}) \le \nu_{\min}[Q_\tti] \le \nu_{\max}[Q_\tti] \le 2\sqrt{c_1}$.
\end{proposition}
\aosversion{\begin{proof}[Proof of \cref{prop:loading-to-factor-i}]
    See \cref{sec:proof-prop:loading-to-factor-i}.
\end{proof}}{}

We are ready to state our main theorem on factor estimation error. Denote
\begin{align}
\label{eq:lda-1-2-def}
    \lambda_{\star,1}^{(e)} &= \min\left\{\snre, \snrbar |\mathcal{E}| \cdot \epsilon_A^2\right\}, \qquad  \qquad \lambda_{\star, 2}^{(e)} = \min\left\{\snre, \snrbar \epsilon_A\right\}. 
\end{align}

\begin{theorem}
\label{thm:factor-est-full}
Suppose \cref{cond:x-reg} -- \cref{cond:separate} hold. There exists a large constant $\tilde{C}_1, \tilde{C}_2, \tilde{C}_3=\poly(c_1)$ such that for any $t\ge 1$, denote $\tilde{r}^{(e)} = r^{(e)} + \log(|\mathcal{E}|) + t$, if \begin{align}
\label{eq:cond-large-n-append}
        \max\left\{\frac{\sup_{e\in \mathcal{E}}\tilde{r}^{(e)}}{n_x}, \sup_{e\in \mathcal{E}} \delta_W^{(e)}(t+\log(4|\mathcal{E}|)), \delta_{W_\tti}(t) / \epsilon_A \right\}\le 1/\tilde{C}_1,
\end{align} and we pick $ \lambda_{\mathtt{ihd}} \in [\tilde{C}_2 \delta_{W_{\tti}}(t), \epsilon_A - \tilde{C}_2 \delta_{W_{\tti}}(t)]$, then the following event occurs with probability at least $1-4e^{-t}$: there exists $|\mathcal{E}| + 1$ invertible matrices $Q_\tti, \{Q_\tth^{(e)}\}_{e\in \mathcal{E}}$ such that their singular values all lie in $[1/(2\sqrt{c_1}), 2\sqrt{c_1}]$, and 
\begin{align*}
    \frac{Q_\tth^{(e)} (\hat{\Phi}_\tth^{(e)})^\top X^{(e)} - F_\tth^{(e)}}{\delta_{F_\tth}^{(e)}} \qquad\text{and}\qquad 
    \frac{Q_{\tti} (\hat{\Phi}_{\tti}^{(e)})^\top X^{(e)} - F_{\tti}^{(e)}}{\delta_{F_\tti}^{(e)}}
\end{align*} are both jointly sub-Gaussian with parameter $1$, where
\begin{align}
    \delta_{F_\tth}^{(e)} &= \tilde{C}_3 \left[\frac{1}{\sqrt{\snrab^{(e)}}} + \sqrt{\kappa^{(e)}} \delta_{F, 2}^{(e)} \right],\\
    \delta_{F_\tti}^{(e)} &= \tilde{C}_3 \Bigg[\frac{1}{\sqrt{\snrba^{(e)}}} + \sqrt{\frac{\tilde{r}^{(e)}}{n_x}} + \sqrt{\kappa^{(e)}} \delta_{F, 2}^{(e)} \left(1 + \sqrt{\kappa^{(e)}} \cdot \sqrt{\frac{\tilde{r}^{(e)}}{n_x}}\right)\Bigg],
\end{align} and where $\delta_{F, 2}^{(e)} = \sqrt{\frac{\tilde{r}^{(e)}}{n_x \cdot \lambda_{\star, 1}^{(e)}}} + \frac{1}{\lambda_{\star, 2}^{(e)}} + \frac{d + \log(|\mathcal{E}|) + t}{n_x \cdot (\lambda_{\star, 1}^{(e)} \land \lambda_{\star, 2}^{(e)})}$ can be interpreted as the second order error in estimating the factor loadings. 
\end{theorem}

\begin{proof}[Proof of \cref{thm:factor-est-full}]
    Denote $t_\star = t + \log(|\mathcal{E}|) + 1$. Recall the definition of $\lambda_{\star, 1}^{(e)}$ and $\lambda_{\star, 2}^{(e)}$ in \eqref{eq:lda-1-2-def}. It is easy to verify that 
    \begin{align*}
        \max\left\{(\lambda^{(e)})^{-1},\delta_{\tti, 2}, \delta_{\tth, 2}^{(e)}, \delta_{\tti, 1}^{(e)}, \delta_{\tth, 1}^{(e)}\right\} \le \sqrt{\frac{r^{(e)} + t_\star}{n_x \cdot \lambda_{\star, 1}^{(e)}}} + \frac{1}{\lambda_{\star, 2}^{(e)}} + \frac{d + t}{n_x \cdot (\lambda_{\star, 1}^{(e)} \land \lambda_{\star, 2}^{(e)})} = \delta_{F, 2}^{(e)}
    \end{align*}

    We first specify $\tilde{C}_1,\tilde{C}_2$ and declare the high probability event under which the result would hold. Let $C_1$ be the constant $\tilde{C}_1$ in \cref{prop:wi-wh}, and $C_2$ be the constant $\tilde{C}_1$ in \cref{lemma:concentrate-3}, we set $\tilde{C}_1\gets C_1 \lor C_2$ and $\tilde{C}_2 \gets C_1$. Under \cref{cond:x-reg} and \cref{cond:separate}, and the condition in \eqref{eq:cond-large-n-append} holds with $\tilde{C}_1 = C_1 \lor C_2$, we can apply \cref{prop:wi-wh}, \cref{lemma:concentrate-3}. The rest of the proof proceeds conditioned on the event $\mathcal{A}_{3}(t) \cap \mathcal{A}_{4}(t)$, which occurs with probability at least $1-4e^{-t}$ by the union bound. 
    
    It follows from \cref{prop:wi-wh} that we have the upper bound for the errors in \eqref{eq:wihat-wi} -- \eqref{eq:wh-wh} that characterizes either $\|\sin \Theta(\cdot, \cdot)\|_2$ or $\|\cos \Theta(\cdot, \cdot)\|_2$ between $\hat{W}_c$ and $W_{c'}$ with $c,c'\in \{\tti, \tth\}$. We assume that $8C_1 \cdot \delta_{F,2}^{(e)} \cdot \sqrt{\kappa^{(e)}} \le (40c_1)^{-1}$ such that we can apply \cref{lemma:hatk} and further \cref{prop:loading-to-factor-i}, otherwise the upper bound is of constant order and thus is trivial.

    Let $C_3$ be the constant $\tilde{C}_1$ in \cref{prop:loading-to-factor-i}. The estimation error for the heterogeneous factor and the well-conditionedness of the ambiguity matrix $Q^{(e)}_\tth$ follows directly from the first claim (1) in \cref{prop:loading-to-factor-i}.
    
    For the invariant factor estimation error, we have the simple bound $\|\hat{\Delta}_{1,j}\|_2 \le 8C_1 \cdot (2\delta_{F,2}^{(e)})$ for $j\in \{1,3,4\}$ and $\|\hat{\Delta}_{1,2}\|_2 \le 8C_1 \cdot (2\sqrt{\kappa^{(e)}} \delta_{F,2}^{(e)})$ by the error bounds \eqref{eq:delta-i-i}--\eqref{eq:delta-h-h} and \eqref{eq:wihat-wh}--\eqref{eq:wh-wh}. At the same time, $\delta_\Sigma^{(e)} \le C_2 \sqrt{\tilde{r}^{(e)}/n_x}$ by \cref{lemma:concentrate-3}. Now we substitute these error bounds into the quantity $g$ in \eqref{eq:factor-i-error}, 
    \begin{align*}
        g&\le C_4 \Bigg\{\sqrt{\frac{1}{\snrba^{(e)}}} + \sqrt{\frac{r^{(e)} + t_\star}{n_x}} + \sqrt{\kappa^{(e)}} \left(\delta_{F,2}^{(e)}\right) \\
        &\qquad \qquad + \sqrt{\kappa^{(e)}} \sqrt{\frac{r^{(e)} + t_\star}{n_x}} \left[\sqrt{\frac{1}{\snrba^{(e)}}} + \sqrt{\kappa^{(e)}} \left(\delta_{F,2}^{(e)}\right)\right] \Bigg\}\\
        &\le 2C_4 \left[\sqrt{\frac{1}{\snrba^{(e)}}} + \sqrt{\frac{r^{(e)} + t_\star}{n_x}} + \sqrt{\kappa^{(e)}} \delta_{F,2}^{(e)} \left(1 + \sqrt{\kappa^{(e)}} \sqrt{\frac{r^{(e)} + t_\star}{n_x}} \right) \right]
    \end{align*} for some $C_4$ depending polynomially on $C_1,C_2,C_3$. Again, we can assume $g \le 1/C_3$ because otherwise we can pick the constant $\tilde{C}_3$ large enough such that the error bound trivially holds. This concludes the proof.
\end{proof}

\subsection{Prediction-invariant factors, invariant signals, and predictions}
\label{appendix:factor-est-star}

We first introduce some additional notations, recall the setting in \cref{thm:factor-est-full}, and further define the factor disentanglement errors
\begin{align*}
    \delta_{F}^{(e)} = \sqrt{2}(\delta_{F_\tti}^{(e)} + \delta_{F_\tth}^{(e)}) \qquad \text{and} \qquad \bar{\delta}_{F} = \sqrt{\frac{1}{|\mathcal{E}_z|}\sum_{e\in \mathcal{E}_z} (\delta_{F}^{(e)})^2} 
\end{align*}

Define $\vartheta = 1 + \sup_{k\in [q]} \|\xi^\star_k\|_2$ and signal strength $\mathsf{s}_\Xi = \nu_{\min}(\Xi_\tti^\star) \land \nu_{\min}(\Xi_\tth^\star)$. Recall that $\bar{r} = \frac{1}{|\mathcal{E}_z|} \sum_{e\in \mathcal{E}_z} r^{(e)}$. 
Our proof is based on the first-order approximation of the initial GLM solutions $\hat\Psi_{\tti}$ and $\hat\Psi_\tth^{(e)}$. For any $t\ge 1$, we denote the second-order errors as 
\begin{align}
\label{eq:delta-z}
\begin{split}
\bar{\delta}_{Z,2} &= \vartheta^2 \left[\frac{\bar{r} \log(n_z) + \log(q|\mathcal{E}_z|) + t}{n_z} + [\bar{\delta}_F]^2 \right] \\
\delta_{Z,2}^{(e)} &= \bar{\delta}_{Z,2} + \vartheta^2 \left[\frac{r^{(e)} \log(n_z)}{ n_z} + [{\delta}_F^{(e)}]^2\right].
\end{split}
\end{align} 

We briefly comment on the conditions in \cref{cond:z-reg}, especially about how to relax them:

\begin{itemize}
    \item The i.i.d. assumption on \cref{cond:z-reg} (a) is relatively standard. Here we assume $\mathbb{E}[\varepsilon_k |F, U] = 0$ to simplify the proof for the first-order approximation of GLM solution. One can replace it with $\mathbb{E}[\varepsilon_k^{(e)} F^{(e)}] = 0$ using a similar proof technique, but with a more complicated derivation.
    \item The conditions (b) and (c) are also standard. It is worth noticing that we do not require the Hessian to be well-conditioned from below and above when the covariate $F$ is bounded, (c) is actually minimal.
    \item We assume that the distribution of $(\varepsilon_1, \ldots, \varepsilon_q)$ are jointly sub-Gaussian and thus is dimension-free. We can relax it by dimension-dependent assumption $\Sigma_{\varepsilon}^{-1/2} \varepsilon$ is jointly sub-Gaussian for some positive definite covariance matrix $\Sigma_\varepsilon$ whose maximum eigenvalue may also grow. In this case, one should replace $q$ in the error rate by $q/\|\Sigma_\varepsilon\|_2$.
\end{itemize}

It is worth noticing that we require $n_z \gtrsim \tilde{r}^{3/2}$ in \eqref{eq:cond-z-text}, this is for deriving the first-order approximation for GLM. One can improve it to be $n_z \gtrsim \tilde{r}$ using leave-one-out analysis.

We are now ready to present a general version of \cref{thm:factor-s-est}. The condition \eqref{eq:cond-z} reduces to \eqref{eq:cond-z-text} when $\vartheta \asymp 1$ and $\mathsf{s}_\Xi \asymp \sqrt{q}$. It implicitly implies that $\mathsf{s}_{\Xi} > 0$ and thus \cref{cond:ident-z} holds. \cref{thm:factor-s-est} directly follows from \cref{thm:z-step1} below. 

\begin{theorem}
\label{thm:z-step1}
Under the setting of \cref{thm:factor-est-full} and the event defined therein, suppose further that  \cref{cond:z-reg} holds. There exist constants $\tilde{C}_1, \tilde{C}_2, \tilde{C}_3 = \poly(c_1, c_2)$ such that for any $t\ge 1$, if
\begin{align}
\label{eq:cond-z}
\begin{split}
    (\tilde{C}_1)^{-1} \ge \max\Bigg\{&|\mathcal{E}_z| \cdot \sup_{e\in \mathcal{E}_z} \delta_{Z,2}^{(e)}, \frac{\sup_{e\in \mathcal{E}_z}(r^{(e)} + \log(q |\mathcal{E}_z|) + t)^{3/2}}{n_z}, \\
    &~~~ \frac{\sqrt{q}}{\mathsf{s}_\Xi} \sup_{e\in \mathcal{E}_z} (\vartheta \delta_F^{(e)} + {\delta}_{Z,2}^{(e)}), \frac{\sqrt{\sup_{e\in \mathcal{E}_z} r_\tth^{(e)} + q + t + \log(|\mathcal{E}_z|) + (r_\tti/|\mathcal{E}_z|)}}{\sqrt{n_z} \cdot \mathsf{s}_\Xi} \Bigg\}
\end{split}
\end{align} by picking $\lambda_{\mathsf{sel}} \in [\tilde{C}_2 \lambda_{\mathsf{sel}, \star}, \mathsf{s}_{\Xi} / (2\sqrt{c_1}) - \tilde{C}_2 \lambda_{\mathsf{sel}, \star}]$ with $\lambda_{\mathsf{sel},\star} = \sqrt{q} \sup_{e\in \mathcal{E}_z}(\vartheta \delta_F^{(e)} + \delta_{Z,2}^{(e)}) + (\sup_{e\in \mathcal{E}_z} r_\tth^{(e)} + q + t + \log(|\mathcal{E}_z|) + r_\tti / |\mathcal{E}_z|)^{1/2} / \sqrt{n_z}$, conditioned on the event in \cref{thm:factor-est-full}, the following event occurs with probability at least $1-7e^{-t}$: there exist invertible matrices $Q_{S_\tti}$ and $Q_{S_\tth}$ such that
\begin{align*}
    \frac{Q_{S_\tti} (\hat{\Phi}_{S_\tti})^\top (\hat{\Phi}^{(e)}_{\tti})^\top X^{(e)} - F_{S_\tti^\star}^{(e)}}{\delta_{{S_\tti}}^{(e)}} \qquad \text{and} \qquad \frac{Q_{S_\tth} (\hat{\Phi}_{S_\tth}^{(e)})^\top (\hat{\Phi}^{(e)}_{\tth})^\top X^{(e)} - F_{S_\tth^\star}^{(e)}}{\delta_{{S_\tth}}^{(e)}} 
\end{align*} are jointly sub-Gaussian with parameter $1$ for any $e\in \mathcal{E}$ and any $e\in \mathcal{E}_z$, respectively, where
\begin{align}
\label{eq:si-sg}
    \delta_{S_\tti}^{(e)} &= \tilde{C}_3 \delta_F^{(e)} + \frac{\tilde{C}_3}{\mathsf{s}_\tti} \left[\sqrt{q} (\vartheta \bar{\delta}_F + \bar{\delta}_{Z,2}) + \sqrt{\frac{r_\tti + t}{n_z \cdot |\mathcal{E}_z|}} + \frac{q}{n_z \cdot |\mathcal{E}_z| \cdot \mathsf{s}_\tti}\right] \\
\label{eq:sh-sg}
\begin{split}
    \delta_{S_\tth}^{(e)} &= \tilde{C}_3 \delta_F^{(e)} + \frac{\tilde{C}_3}{\mathsf{s}_\tth} \Bigg[\sqrt{q} \left(\vartheta (\bar{\delta}_F + \delta_F^{(e)}) + {\delta}_{Z,2}^{(e)} \right) + \sqrt{\frac{r_\tth^{(e)} + \log(|\mathcal{E}_z|) + t}{n_z}} \\
    &\qquad \qquad \qquad \qquad + \frac{q^{1.5}}{n_z^{1.5} \mathsf{s}_\tth^2} + \frac{q }{\sqrt{|\mathcal{E}_z|} \cdot n_z \mathsf{s}_\tth}\Bigg]
\end{split}
\end{align} with $\mathsf{s}_\tti = \nu_{\min}(\Xi_{\tti}^\star)$ and $\mathsf{s}_\tth = \nu_{\min}(\Xi_{\tth}^\star)$. Moreover, $\hat{\Xi}_\tth$ in \eqref{eq:est-xi-h} satisfies, for $V_\tth^\star \in \mathcal{O}_{q \times |S_\tth^\star|}$ with $\col(V_\tth^\star) = \col(\Xi_\tth^\star)$
\begin{align}
\label{eq:sh-xi}
    \left\|\sin\Theta(\hat{\Xi}_\tth, V_\tth^\star) \right\|_2 \le \frac{\tilde{C}_3}{\nu_{\min}(\Xi_\tth^\star)} \left[\sqrt{q} (\vartheta \bar{\delta}_F + \bar{\delta}_{Z,2}) + \sqrt{\frac{\bar{r} - r_\tti + q + t}{n_z \cdot |\mathcal{E}_z|}} + \frac{q}{n_z \cdot \nu_{\min}(\Xi_{\tth}^\star)}\right].
\end{align} 
\end{theorem}
\aosversion{\begin{proof}[Proof of \cref{thm:z-step1}] See \cref{sec:proof-thm-z-step1}.
\end{proof}}{}

We also need a regularity condition for the GLM on the response variable. 

\begin{condition}
\label{cond:y-reg}
    There exists a constant $c_3$ such that the following condition holds for $e\in \mathcal{E}_y$:
    \begin{itemize}
        \item[(a)] We observe i.i.d. $\mathcal{D}_y\supe=\{(X_i\supe, Y_i\supe)\}_{i=1}^{n_y}$ from the model \eqref{eq:model-invariant-y}, $\mathcal{D}_y\supe$ are independent of $\mathcal{D}_x\supe, \mathcal{D}_z\supe$, data across different environments are also independent, $\|\beta^\star\|_2 \le c_3$.
        \item[(b)] $\|\sigma_y'\|_\infty \lor \|\sigma_y''\|_\infty \le c_3$, and Hessian matrix is well-conditioned at $\beta^\star$, $$H_y\supe:=\mathbb{E}[\sigma_y'((\beta^\star)^\top F_{S^\star}\supe) F_{S^\star}\supe (F\supe_{S^\star})^\top]$$ has eigenvalues from $c_3^{-1}$ to $c_3$. $\sigma_y'(\cdot) \ge 0$. 
        \item[(c)] $\varepsilon_y\supe$ is sub-Gaussian with parameter $\sqrt{c_3}$ almost surely. 
    \end{itemize}
\end{condition}

Let $\delta_{S^\star} = \sup_{e\in \mathcal{E}_y} (\delta_{S_\tti}^{(e)} + \delta_{S_\tth}^{(e)})$ be the worst-case factor estimation error in environments $e\in \mathcal{E}_y$. The following theorem is a general version of \cref{thm:beta-est}. 

\begin{theorem}
\label{thm:beta-est}
    Under the setting of \cref{thm:z-step1} and the event defined therein, suppose further that \cref{cond:y-reg} holds. There exists constants $\tilde{C}_1, \tilde{C}_2, \tilde{C}_3=\poly(c_1,c_2, c_3)$ such that if
    \begin{align*}
        |\mathcal{E}_y| \cdot n_y \ge \tilde{C}_1 \cdot (t + |S_\star|) \qquad \text{and} \qquad \delta_{S^\star}^{-1} \ge \tilde{C}_1, 
    \end{align*} then conditioned on the event in \cref{thm:z-step1}, the solution $\hat{\beta}$ in \eqref{eq:est-y} satisfies
    \begin{align*}
        \left\|Q^{-\top} \hat{\beta} - \beta^\star \right\|_2 \le \tilde{C}_2 \underbrace{\left[\sqrt{\frac{|S_\star| + t}{n_y \cdot |\mathcal{E}_y|}} + \delta_{S_\star} \right]}_{\delta_y}  \qquad \text{with} \qquad Q = \begin{bmatrix}
        Q_{S_\tti} & 0\\
        0 & Q_{S_\tth} 
        \end{bmatrix}
    \end{align*} with probability at least $1-e^{-t}$. Moreover, we further have
    \begin{align*}
        \forall e\in \mathcal{E}_z \setminus \mathcal{E}_y,  ~~~&\sqrt{\mathbb{E}\left[\left\| \hat{\beta}^\top \tilde{F}\supe_S(X\supe) -  (\beta^\star)^\top F_{S^\star}\supe \right\|_2^2\right]} \le \tilde{C}_3 \left[\delta_y + (\delta^{(e)}_{S_\tti} + \delta^{(e)}_{S_\tth}) \right],\\
        \forall e\in \mathcal{E} \setminus \mathcal{E}_z,  ~~~& \sqrt{\mathbb{E}\left[\left\| \hat{\beta}_\tti^\top \tilde{F}\supe_{S_\tti}(X\supe) -  (\beta_\tti^\star)^\top F_{S^\star_\tti}\supe \right\|_2^2\right]} \le \tilde{C}_3 \left[\delta_y + \delta^{(e)}_{S_\tti}\right].
    \end{align*}
\end{theorem}

\section{Auxiliary results}
\label{append:first-order-approx}
\subsection{Auxiliary results on matrix denoising}

This section offers first-order approximations of the SVD/PCA solutions when the signal-to-noise ratio is larger than a constant. We establish deterministic results on $\|\cdot\|_2$ error bounds for the residuals. This section is self-contained and may use notations $U, V, E, \Sigma, n, r$ that differ from the semantic meaning defined in the other sections. 

\myparagraph{First-order approximation for SVD.} Let $M_\star$ be an $n_1\times n_2$ low-rank matrix with $\rank(M_\star) = r$, and $M = M_\star + E$ where $E$ is a small perturbation sharing the same shape. Denote $M_\star = U_\star \Sigma_\star (V_\star)^\top$, where $\Sigma_\star \in \mathbb{R}^{r\times r}$ is a diagonal matrix with $\{[\Sigma_{\star}]_{i,i} = \sigma_i\}_{i=1}^r$ representing the singular values of $M_\star$ in descending order $\sigma_1\ge \ldots \ge \sigma_r$, $U_\star \in \mathbb{R}^{n_1\times r}$ represents the left singular vectors of $M_\star$ and $V_\star \in \mathbb{R}^{n_2 \times r}$ represents the right singular vectors of $M_\star$. The following deterministic result offers the first-order approximation of the projection matrix $VV^\top$, where $V \in \mathcal{O}_{n_2\times r}$ is any orthonormal basis for the top-$r$ right singular subspaces of $M$.

\begin{theorem}
\label{thm:first-order-svd}
    Under the above setting, suppose the signal-to-noise ratio $\lambda =  \sigma_r / \|E\|_2 \allowbreak > 10$, we have that
    \begin{align}
    \label{eq:first-order-svd-x}
    V = (V_\star + V_\perp X) (I + X^\top X)^{-1/2} O \qquad \text{with} \qquad X = V_{\perp}^\top E^\top U_\star \Sigma_\star^{-1} + R_X
    \end{align} for $O\in \mathcal{O}_r$ and $\|R_X\|_2 \le 5\lambda^{-2}$, and 
    \begin{align}
    \label{eq:first-order-svd}
        \left\| VV^\top - V_\star (V_\star)^\top - G - G^\top \right\|_2 \le 16 \lambda^{-2} ~~~\text{with}~~~  G=V_{\perp} (V_{\perp})^\top E^\top U_\star \Sigma_\star^{-1} (V_\star)^\top, 
    \end{align} where $V_{\perp} = (V_{\star})_{\perp}$ such that $[V_\star, V_\perp] \in \mathcal{O}_{n_2}$.
\end{theorem}

The main technical difficulty here is to get rid of the dependency on the condition number $\sigma_1/\sigma_r$ on the approximation error $1/\lambda^2$. This is the pre-factor that cannot be removed by directly establishing the first-order approximation for standard PCA. 

\myparagraph{First-order approximation for PCA.} We next consider the case where $M, E$ and $M^\star$ are symmetric matrices such that $M=M_\star + E$, while $M_\star$ is only approximately low-rank (with $\lambda_r(M_\star) - \lambda_{r+1}(M_\star) \gg \|E\|_2$). In particular, we write $M_\star$ as 
\begin{align*}
M_\star = \begin{bmatrix} V_\star & V_\perp \end{bmatrix} \begin{bmatrix} \Lambda_\star & 0 \\ 0 & \Lambda_{\perp} \end{bmatrix} \begin{bmatrix} (V_\star)^\top \\ V_\perp^\top \end{bmatrix}.
\end{align*} Here $\Lambda = \begin{bmatrix} \Lambda_\star & 0 \\ 0 & \Lambda_{\perp} \end{bmatrix} = \mathrm{diag}\{\lambda_1, \ldots, \lambda_d\}$ represents the eigenvalues of $M_\star$ in a descending order (i.e., $\lambda_1\ge \lambda_2 \ge \cdots \lambda_d$), $V_\star$ is a $d$ by $r$ matrix whose $j$-th column is the eigenvectors corresponding to $\lambda_j$, and $V_\perp$ is a $d$ by $(d-r)$ matrix whose $j$-th column is the eigenvectors corresponding to $\lambda_{j + r}$. 

This is also the regime considered by \cite{fan2019distributed} Lemma 2. Here we provide the result on $\|\cdot\|_2$ norm rather than $\|\cdot\|_F$ norm in \cite{fan2019distributed}. The following deterministic result offers the first-order approximation of $V$, the top-$r$ eigenvectors of $M$, in \eqref{eq:pca-first-order-x}. The first-order approximation of the projection matrix $VV^\top$, i.e., \eqref{eq:first-order-pca}, is identical to Lemma 2 of \cite{fan2019distributed}. However, only \eqref{eq:first-order-pca}, which is the direct corollary of \eqref{eq:pca-first-order-x}, cannot yield a faster rate for $W_\tti^\top \hat{W}_\tth\supe$ -- we need a slightly stronger result \eqref{eq:pca-first-order-x} here. 

The following theorem offers a first-order approximation for $V \in \mathcal{O}_{d\times r}$, which is any orthonormal basis for the top-$r$ eigenspaces of $M$.

\begin{theorem}
\label{thm:first-order-pca}
    Under the above setting, suppose the signal-to-noise ratio $\lambda = (\lambda_r - \lambda_{r+1}) / \|E\|_2 > 10$, then
    \begin{align}
    \label{eq:pca-first-order-x}
        V = (V^\star + V_\perp X) (I_r + X^\top X)^{-1/2} O \qquad \text{with} \qquad X = G_X + R_X \in \mathbb{R}^{(d - r) \times r}
    \end{align} where
    \begin{align}
        O \in \mathcal{O}_{r}, \qquad \|R_X\|_2 \le 3\lambda^{-2} \qquad \text{and} \qquad G_X = \mathsf{Mut}_{\Lambda,r} \left[(V_{\perp})^\top E V_\star \right]
    \end{align} and where $\mathsf{Mut}_{\Lambda,r}: \mathbb{R}^{(d-r)\times r} \to \mathbb{R}^{(d-r)\times r}$ is a map such that $[\mathsf{Mut}_{\Lambda, r}(A)]_{i,j} = \frac{A_{i,j}}{\lambda_j - \lambda_{r + i}}$. Furthermore,
    \begin{align}
    \label{eq:first-order-pca}
        \left\| VV^\top - V_\star (V_\star)^\top - G - G^\top \right\|_2 \le 10 \lambda^{-2} ~~~\text{with}~~~  G=V_{\perp} G_X (V_\star)^\top.
    \end{align}
\end{theorem}

\subsection{Auxiliary results on generalized linear model}

This section offers $L_2$ parameter estimation error on GLM solutions. We first state the setup. We will use notations $(X,Y) \in \mathbb{R}^d\times \mathbb{R}$ whose semantic meaning is different from the main text. 

\paragraph{Setup and estimation} We assume $X \in \mathbb{R}^d$ and $Y$ follows from GLM with known mean function $\sigma(\cdot)$ and ground-truth parameter $\beta_\star \in \mathbb{R}^d$, we observe a transform of the covariate $X$ with certain measurement error, $\tilde{X}$ such that $\Delta_X := Q \tilde{X} - X$ satisfies $\sup_{\|u\|_2 = 1} \mathbb{E}[|u^\top \Delta_X|_2^2] = o(1)$, where $Q$ is fixed but unknown ambiguity matrix admits the form of 
\begin{align*}
Q = \begin{bmatrix}
    Q_s & 0 \\
    0 & Q_h
\end{bmatrix} \qquad \text{with} \qquad Q_s \in \mathbb{R}^{d_s\times d_s}, Q_h \in \mathbb{R}^{(d-d_s) \times(d - d_s)} ~\text{both invertible}
\end{align*} Here $d_s \in \{0, \ldots, d\}$ is pre-known integer. For any $v\in \mathbb{R}^d$, we also write
\begin{align*}
    v = (v_s^\top, v_h^\top) \in \mathbb{R}^d \qquad \text{with} \qquad v_s \in \mathbb{R}^{d_s} \text{ and } v_h \in \mathbb{R}^{d-d_s}.
\end{align*} When $d_s = 0$, we have $v = v_h$. 

In this case, the $n$ observations can be written as
\begin{align}
\begin{split}
    \tilde{X}_i &= Q^{-1} (X_i + \Delta_{X,i}) \\
    Y_i &= \sigma(\beta_\star^\top X_i) + \varepsilon_i \qquad \text{with} \qquad \mathbb{E}[\varepsilon_i X_i] = 0
\end{split} \label{eq:glm-model-supp}
\end{align} for $i\in [n]$. Note that here we observe independent (not necessarily identically distributed) data $\{(\tilde{X}_i, Y_i)\}_{i=1}^n$ from the above model (the distribution of $(X_i,\tilde{X}_i, Y_i)$ can be different, but they share the same $\beta_\star$, $Q$ with the exogeneity on $\varepsilon$) and run MLE estimation with fixed $\bar{\beta}_s$, that is,
\begin{align}
\label{eq:glm-est-supp}
    \hat{\beta} = \argmin_{\beta = [\bar{\beta}_s, \beta_h]} \frac{1}{n} \sum_{i=1}^n [b(\beta^\top \tilde{X}_i) - \beta^\top \tilde{X}_i \cdot Y_i]
\end{align} where $b(v)$ is the function satisfying $b'(v) = \sigma(v)$. We first impose some regularity condition on the above data-generating process.

\begin{condition}
\label{cond:glm-supp}
    Conditioned on fixed $\bar{\beta}_s$ and $Q$, there exists a universal constant $c_4$ such that the following condition holds

    \noindent (a) $\{(\tilde{X}_i, Y_i)\}_{i=1}^n$ are independent observations from the model in \eqref{eq:glm-model-supp}.

    \noindent (b) For each $i\in [n]$, $X_i$ is jointly sub-Gaussian with parameter $c_4$, $\mathbb{E}[X_i X_i^\top]$ has eigenvalues from $c_4^{-1}$ to $c_4$, $\varepsilon_i$ is sub-Gaussian with parameter $c_4$. 
    
    \noindent (c) For each $i\in [n]$, $\Delta_{X,i}$ is jointly sub-Gaussian with parameter $\delta_{X} = o(1)$.

    \noindent (d) The Hessian matrix is well-conditioned at $\beta_\star$, $H_i = \mathbb{E}[\sigma'(\beta_\star^\top X_i) X_iX_i^\top]$ has eigenvalues from $1/c_4$ to $c_4$ for any $i\in [n]$. $\|\sigma'\|_\infty \lor \|\sigma''\|_\infty \le c_4$ and $\sigma'(\cdot) \ge 0$ such that the loss is convex. 
\end{condition}

The next theorem provides the $L_2$ error bound of $\hat{\beta}$ on $\beta_\star$. 

\begin{theorem}
\label{thm:glm-supp}
    Under \cref{cond:glm-supp}, there exists some constant $\tilde{C} = \poly(c_4)$ such that conditioned on given $\bar{\beta}_s$ and $Q$, the following holds for any $t\ge 1$. If $n \ge \tilde{C} [(d-d_s) + t]$, $\|Q_s^{-\top} \bar{\beta}_s - \beta_{\star, s}\|_2 \lor \delta_X \cdot (\|\beta_\star\|_2+1) \le \tilde{C}^{-1}$, then with probability at least $1-e^{-t}$,
    \begin{align*}
        \|Q_h^{-\top} \hat{\beta}_h - \beta_{\star,h}\|_2 \le \tilde{C} \left[\sqrt{\frac{(d - d_s) + t}{n}} + \delta_X \cdot (\|\beta_\star\|_2+1) + \|Q_s^{-\top} \bar{\beta}_s - \beta_{\star, s}\|_2\right]. 
    \end{align*} 
\end{theorem}

\section{Implementation details for experiments}
\label{append:numerical}

\subsection{Model and implementation details for simulations}
\label{subsec:simulation}

\myparagraph{Data generating process. }
Recall the linear factor model $X\supe = B F_\tti\supe + A\supe F_\tth\supe + U\supe$ with shared loading $B$, environment-specific loading $A\supe$, invariant factors $F_\tti\supe$ and heterogeneous factors $F_\tth\supe$. We will adopt $[F_\tti\supe, F_\tth\supe] \sim \mathcal{N}(0, I_{r\supe})$, and in each trial, the loadings are assigned as
\begin{align*}
    B = \sqrt{d} \cdot W_\tti \qquad \text{and} \qquad A\supe = \sqrt{d} \left[ W_{\mathtt{I}} R_1 + W_{\tth}\supe R_2 \right]
\end{align*} where $W_\tti, W_\tth\supe$ are randomly generated matrices with orthonormal columns, $W_\tti^\top W_\tti = I_{r_\tti}$, $(W_\tth\supe)^\top W_\tth\supe = I_{r_\tth\supe}$, such that $W_\tti^\top W_\tth\supe = 0$ for any $e\in \mathcal{E}$, $R_1$ and $R_2$ are random generated transform matrix with singular values ranging from 0.5 to 2. 

For the auxiliary response $Z\supe \in \mathbb{R}^q$, we consider the binary outcome with logistic $\sigma_z(t) = 1/(1+e^{-t})$ such that
\begin{align*}
    \mathbb{E}[Z\supe | F_\tti\supe, F_\tth\supe] = \overline{\sigma}_z\left(\Xi_\tti^\star F_{S_\tti}\supe + \Xi_\tth^\star F_{S_\tth}\supe\right)
\end{align*} where $\overline{\sigma}_z: \mathbb{R}^q\to \mathbb{R}^q$ applies the logistic function $\sigma_z(\cdot)$ to each coordinate, and the ground truth coefficient parameters $\Xi^\star_\tti \in \mathbb{R}^{q\times |S_\tti|}$, $\Xi^\star_\tth \in \mathbb{R}^{q\times |S_\tth|}$ are generated via
\begin{align*}
    \Xi_\tti^\star = \sqrt{q}\cdot V_\tti \qquad \text{and} \qquad \Xi_\tth^\star = \sqrt{q} \cdot V_\tth
\end{align*} for randomly generated matrices with orthonormal columns, $V_\tti^\top V_\tti = I_{|S_\tti|}$ and $V_\tth^\top V_\tth = I_{|S_\tth|}$. 

We consider the continuous $Y\supe$ with spurious signals from $F_{[r\supe] \setminus [r_\tti] \setminus S_\tth}\supe$ as 
\begin{align}
    Y\supe = \beta_\tti^\top F_{S_\tti}\supe + \beta_\tth^\top F_{S_\tth}\supe + (\beta_{\mathrm{spur}}\supe)^\top F_{[r\supe] \setminus [r_\tti] \setminus S_\tth}\supe + \varepsilon\supe 
\end{align} where the coefficients are also generated randomly in each trial with $\|\beta_\tti\|_2^2=0.5$, $\|\beta_\tth\|_2^2=0.3$, $\|\beta_{\mathrm{spur}}\supe\|_2^2=0.15$, and $\varepsilon\supe \sim \mathcal{N}(0, 0.05)$ is the noise that is independent of $U\supe$ and $F\supe$. In this case, the model on top of $F_\tti$ would have $R^2$ approximately $0.5$, the model on top of $F_S$ will have $R^2$ approximately $0.8$. 

\myparagraph{Implementation for ``Alignment of invariant factors''. }  For the oracle map based on $B^{(e)}$, we use $$\tilde{F}_{\tti}^{(e),\mathtt{oracle}}(x) = \left\{B - (I-\mathsf{P}_B) A\supe [A\supe (I - \mathsf{P}_B) A\supe]^{-1/2} (A\supe)^\top B\right\}^\top x$$ with the project matrix $\mathsf{P}_B = B(B^\top B)^{-1} B^\top$.

\myparagraph{Implementation for ``Alignment of prediction-invariant factors''.} For a method $\mathtt{m}$ that estimate ${F}_{\mathtt{c}}\supe$ from $X\supe$ by the map $\tilde{F}_{S_{\mathtt{c}}}^{(e), \mathtt{m}}: \mathbb{R}^d \to \mathbb{R}^{|S_{\mathtt{c}}|}$ with $\mathtt{c} \in \{\tti, \tth\}$, we use the metric 
\begin{align} \label{eq:simu-fs}
    \mathrm{MSE}_{F_{S_{\mathtt{c}}}}(\mathtt{m}) = \inf_{Q\in \mathbb{R}^{|S_{\mathtt{c}}| \times |S_{\mathtt{c}}|}} \frac{1}{|\mathcal{E}|} \sum_{e\in \mathcal{E}} \frac{1}{n_{test}} \sum_{i=1}^{n_{test}} \left\|Q \tilde{F}_{S_{\mathtt{c}}}^{(e), \mathtt{m}}(X_i\supe) - F_{S_\tti,i}\supe \right\|_2^2
\end{align} with another $n_{test}=30000$ data to evaluate the accuracy in aligning $F\supe_{S_{\mathtt{c}}}$ across $e\in \mathcal{E}$. We consider the three methods:
\begin{itemize}
    \item The oracle estimator: the estimate builds a map on top of $M_\tti F_\tti\supe$ and $M_\tth\supe F_\tth\supe$ with oracle $F_\tti\supe$ and $F_\tth\supe$ instead of the estimated surrogates $\tilde{F}_\tti\supe, \tilde{F}_\tth\supe$.
    \item The spectral method (SP): our estimator described in \cref{sec:method2} with $\tilde{F}_{S_\tti}(x) = \hat{\Phi}_{S_\tti}^\top (\hat{\Phi}\supe_\tti)^\top x$ and $\tilde{F}_{S_\tti}(x) = (\hat{\Phi}_{S_\tth}\supe)^\top (\hat{\Phi}\supe_\tti)^\top x$ with $\hat{\Phi}_{S_\tti}, \hat{\Phi}_{S_\tth}\supe$ being both estimated via spectral methods as described in \cref{sec:method2}.
    \item The MLE method: our estimator described in \cref{sec:method2} with $\tilde{F}_{S_\tti}(x) = (\hat{\Phi}_{S_\tti}^{\mathrm{MLE}})^\top (\hat{\Phi}\supe_\tti)^\top x$ and $\tilde{F}_{S_\tti}(x) = (\hat{\Phi}_{S_\tth}^{(e), \mathrm{MLE}})^\top (\hat{\Phi}\supe_\tti)^\top x$ with $\hat{\Phi}_{S_\tti}^{\mathrm{MLE}}, \hat{\Phi}_{S_\tth}^{(e), \mathrm{MLE}}$ being both estimated via nonconvex MLE objective in \eqref{eq:method-mle}.
\end{itemize} 

\myparagraph{Implementation for ``Generalization in new environments''. } For out-of-sample environments, we generate three additional environments $\mathcal{E}_o$ via the same procedure of generating environments in $\mathcal{E}$. The number of i.i.d. $X$ observations and $(X,Z)$ observations in new environments $\mathcal{E}_o$ that are used to construct corresponding diversified projection matrices are also $n_x$ and $n_z$, respectively. We evaluate the prediction performance of the prediction $\hat{g}^{(e),\mathtt{m}}(x)$ via the worst-case out-of-sample $R^2$ in $\mathcal{E}_o$, defined as 
\begin{align}
    R^2_{\mathtt{oos}, \mathtt{m}} = \min_{e\in \mathcal{E}_o} \left[1-\frac{\frac{1}{n_{test}} \sum_{i=1}^{n_{test}} \left(\hat{g}^{(e),\mathtt{m}}(X_i\supe) - Y_i\supe \right)^2}{\frac{1}{n} \sum_{i=1}^{n_{test}} |Y_i\supe - \hat{\mu}\supe|^2}\right],
\end{align} where $\hat{\mu}\supe = \frac{1}{n_{test}} \sum_{i=1}^{n_{test}} Y_i\supe$ and $n_{test}=30000$. We consider the following predictors,
\begin{itemize}
    \item \myred{Oracle}: It directly regress $Y\supe$ on $F_S\supe$ using all the data across $e\in \mathcal{E}$, and use the estimated $\hat{\beta}_{oracle}$ to make predictions in new environments when the corresponding $F_S\supe$ is also available. It gives the best worst-case OOS performance by \cref{prop:robust1}.
    \item \myorange{Oracle $F_{S_\tti}$ }: It shares a similar spirit to {\it Oracle} when only $F_{S_\tti}\supe$ is available. It gives the best worst-case OOS performance when only $F_{S_\tti}$ is available according to \cref{prop:robust2}.
    \item \myblue{ATLAS}: It regresses $Y\supe$ on $[\tilde{F}_{S_\tti}\supe(X\supe), \tilde{F}_{S_\tth}\supe(X\supe)]$ using all the data in $e\in \mathcal{E}$ and get estimated $\hat{\beta}_\tti$, $\hat{\beta}_\tth$, then make predictions in environments $t\in \mathcal{E}_o$ by $\hat{g}^{(t)}(x) = \hat{\beta}_\tti^\top \tilde{F}_{S_\tti}^{(t)}(x) + \hat{\beta}_\tth^\top \tilde{F}_{S_\tth}^{(t)}(x)$. We also consider another version \mylightblue{ ATLAS(wo/Z)} without the knowledge of $Z$: in this case, the prediction are made by $\hat{g}^{(t)}(x) = \hat{\beta}^\top \tilde{F}_{S_\tti}^{(t)}(x) = \hat{\beta}^\top \hat{\Phi}_{S_\tti}^\top (\hat{\Phi}_{\tti}\supe)^\top x$.
    \item \mylightblue{IHD}: It regress $Y\supe$ on $\tilde{F}_\tti\supe$ using all the data in $e\in \mathcal{E}$ and get $\hat{\beta}$, then make predictions in new environment $t\in \mathcal{E}_o$ by $\hat{g}^{(t)}(x) = \hat{\beta}^\top (\hat{\Phi}_\tti^{(t)})^\top x$.
    \item \mygreen{PoolPCA}: It regress $Y\supe$ on the $\hat{W}^\top X\supe$ using all the data in $e\in \mathcal{E}$ and predict in the environment by $\hat{g}(x) = \hat{\beta}^\top \hat{W}^\top x$. 
    \item \myyellow{AJIVE}: It regress $Y\supe$ on the $\hat{W}_\tti X\supe$ using all the data in $e\in \mathcal{E}$ and predict in the environment by $\hat{g}(x) = \hat{\beta}^\top \hat{W}_\tti^\top x$. 
\end{itemize}

\subsection{Implementation details real data application} 
\label{subsec:real-data}

\begin{figure}[!t]
\begin{center}
\begin{tikzpicture}
    \node[anchor=west] at (10.3, 0) {time};
    \draw[line width=1.34pt, ->] (0, 0) -- (10, 0);
    \draw[line width=1.34pt, -] (9, -0.134) -- (9, 0.134);
    \node[anchor=south] at (9, 0.5) {$Y$/$Z$ Label};
    \node[anchor=south] at (9, 0.1) {$t_0$};
    \node[anchor=north west] at (8.5, -0.4) {$Y$: DAS category $\in\{$Moderate, High$\}$  };
    \node[anchor=north west] at (8.5, -0.9) {$Z_{1:3}$: Extracted $\sqrt{\text{TJC28}}$,  $\sqrt{\text{SJC28}}$, $\log($CRP$)$};
    \node[anchor=north west] at (8.5, -1.4) {$Z_{4:6}$: Inferred DAS category};
    \node[anchor=north west] at (8.5, -1.9) {$Z_{7:19}$: Indicator basis for $Z_{1:3}$};
    \draw[myblue!50, fill=myblue!50] (5, -0.3) rectangle (3, 0.1);
    \draw[myred!50, fill=myred!50] (1, -0.3) rectangle (3, 0.1);
    \draw[line width=1.34pt, -] (1, -0.134) -- (1, 0.134);
    \draw[line width=1.34pt, -] (3, -0.134) -- (3, 0.134);
    \draw[line width=1.34pt, -] (5, -0.134) -- (5, 0.134);
    \node[myblue] at (4, -0.1) {Near};
    \node[myred] at (2, -0.1) {Past};
    \node[myred] at (2, 0.34) {180d};
    \node[myblue] at (4, 0.34) {180d};
    \draw[line width=1.34pt,color=gray,-] (5, 0) -- (9, 0);
    \node[gray] at (7, 0.34) {360d};
    \node[myred] at (2, -0.9) {$X_{1:p}$};
    \node at (2, -1.3) {tfidf(\myred{$c_{past}$})};
    \node[myblue] at (4, -0.9) {$X_{p+1:2p}$};
    \node at (4, -1.3) {tfidf(\myblue{$c_{near}$})};
    \node[myorange] at (3, -2) {$X_{2p+1:2p+3}$};
    \node at (3, -2.4) {sex, race, age\_first\_RA};
\end{tikzpicture}
\end{center}
\caption{Illustration of covariate/label construction.}
\label{fig:das-data}
\end{figure}
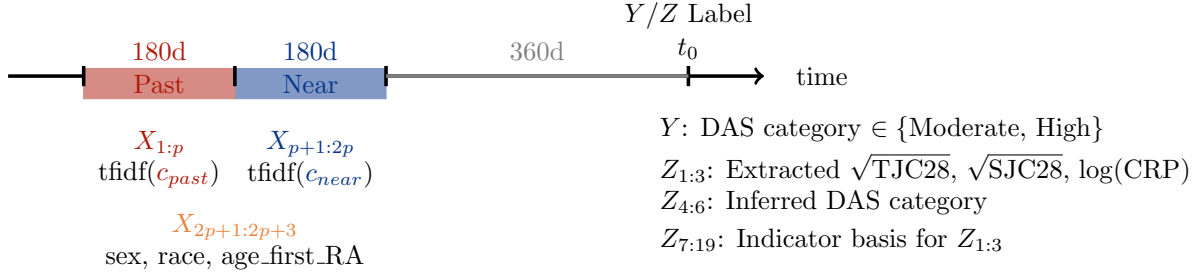
\myparagraph{Covariate construction.} Each observation corresponds to one labeled patient visit, the covariates are constructed from the 360-day interval $[t_0-720,t_0-360)$. We divide this covariate window into two 180-day sub-windows $[t_0-720,t_0-540)$ and $[t_0-540,t_0-360)$, which we refer to as the past and near windows, respectively. The construction is illustrated in \cref{fig:das-data}. 

We will use the counts of $m_0=278$ medical codes and $m_1=484$ clinical concepts (CUI) in the two time windows to construct the covariate vector, we use the indices $\{1, \ldots, m_0\}$ to represent the codes and indices $\{m_0+1, \ldots, m_0+m_1\}$ to represent the CUIs. For any time interval $[t_l,t_r)$, let $c^{[t_l,t_r)} \in \mathbb{N}^{m_0+m_1}$ with $c_j^{[t_l,t_r)}$ being code/CUI counts for the patient in that interval. Let $w_1,\ldots,w_{m_0}\in\mathbb{R}^p$ be the code embeddings and $w_{m_0+1},\ldots,w_{m_0+m_1}\in\mathbb{R}^p$ be the CUI embeddings with $p=1768$. The embedding $w_1$ corresponds to the RA code ``PheCode:714.1'', and $w_{m_0+1}$ corresponds to the RA CUI ``C0003873''. Let $n_j$ be the number of patients in the full cohort who have code/CUI $j$ at least once. We transform the count vector by the embedding-weighted log-count map
\begin{align*}
    \mathrm{tfidf}(c^{[t_l, t_r)}) = \sum_{j=1}^{m_0} \frac{\log(c_j^{[t_l, t_r)} + 1)}{\log(n_j + 1)} (w_j^\top w_1) \cdot w_j + \sum_{j=m_0+1}^{m_0+m_1} \frac{\log(c_j^{[t_l, t_r)} + 1)}{\log(n_j + 1)} (w_j^\top w_{m_0+1}) \cdot w_j \in \mathbb{R}^p
\end{align*}
The final covariate vector is
\begin{align*}
    X
    =
    \Big[
    \mathrm{tfidf}\!\left(c^{[t_0-720,t_0-540)}\right),
    \mathrm{tfidf}\!\left(c^{[t_0-540,t_0-360)}\right),
    X_{\mathrm{sex}},
    X_{\mathrm{race}},
    X_{\mathrm{age}}
    \Big]
    \in \mathbb{R}^{2p+3},
\end{align*}
where $X_{\mathrm{sex}}\in\{0,1\}$ is a binary indicator for sex, $X_{\mathrm{race}}\in\{0,1\}$ indicates whether the race of the patient is white or non-white, and $X_{\mathrm{age}}$ is the age at which the patient first got a RA code.

\myparagraph{Response and auxiliary labels.} For the patient visit at time $t_0$, the response label $Y$ is
\[
    Y_l = 1\{\mathrm{DAS~category~at~} t_0~\mathrm{in}~\{ \text{Moderate, High}\}\}.
\]
The $(X, Y)$ pairs in the training environments are from the chart review of the clinical notes, where the DAS category is labeled by the human expert following a certain procedure. The $(X, Y)$ pairs in the testing environments are from the cohort of Brigham and Women's Hospital Rheumatoid Arthritis Sequential Study (BRASS), whereas the DAS category is determined by the cut-off of the three-component DAS28-CRP score,
\begin{align*}
    \mathrm{DAS}
    =
    [0.56 \sqrt{\mathrm{TJC28}}
    + 0.28 \sqrt{\mathrm{SJC28}}
    + 0.36 \log(\mathrm{CRP})] \cdot 1.1
    + 1.15.
\end{align*}
Here, $\mathrm{TJC28}$ and $\mathrm{SJC28}$ are the tender and swollen joint counts over 28 joints, $\mathrm{CRP}$ is a blood marker of systemic inflammation. The cutoffs for category ``High'' and ``Moderate'' are $\mathrm{DAS}>5.1$ and ($3.2<\mathrm{DAS}\le 5.1$, respectively.

The auxiliary label $Z_l \in\mathbb{R}^{19}$ is constructed from the structured output extracted from the clinical notes at time $t_0$ by the LLM agent. Its first three coordinates are continuous DAS-related quantities, $Z_{l,1:3}=\left(\sqrt{\mathrm{TJC28}_l}, \sqrt{\mathrm{SJC28}_l}, \ln(\mathrm{CRP}_l)\right)$. The next three coordinates encode the ordinal DAS category inferred by the agent, i.e., $Z_4, Z_5, Z_6$ are binary indicators of whether the inferred DAS category lies in $\{\text{High}\}$, $\{\text{High, Moderate}\}$, $\{\text{High, Moderate, Low}\}$. The remaining coordinates are indicator bases for the three continuous auxiliary variables. For each $k\in\{1,2,3\}$, we compute empirical quantile $\{\tau_{k, j}\}_{j=1}^5$ cutoffs of $Z_k$ and include indicators of the form $1\{Z_{l,k}\ge \tau_{k,j}\}$. Specifically, $Z_1$ and $Z_2$ each use four distinct cutoffs, giving $Z_{7:10}$ and $Z_{11:14}$, while $Z_3$ uses five cutoffs, giving $Z_{15:19}$. The smaller number of cutoffs for $Z_1$ and $Z_2$ is due to repeated lower empirical quantile values at zero.

\begin{table}[!t]
\begin{center}
\begin{tabular}{c|l|l|r|r|l}
\hline
\hline
Env ID & Time begin & Time end & $\#(X,Z)$ & $\#(X,Y)$ & Type\\
\hline
1 & 2009-01-01 & 2012-12-31 & 37849 & 118 & Train\\
2 & 2013-01-01 & 2015-12-31 & 36534 & 153 & Train\\
3 & 2016-01-01 & 2018-12-31 & 52157 & 208 & Train\\
4 & 2019-01-01 & 2019-12-31 & 19056 & 251 & Test\\
5 & 2020-01-01 & 2020-12-31 & 13014 & 89 & Test\\
6 & 2021-01-01 & 2021-12-31 & 15922 & 85 & Test\\
\hline
\hline
\end{tabular}
\end{center}
\caption{Sample sizes by temporal environment. Each observation is a labeled patient visit.}
\label{table:real_data}
\end{table}

\myparagraph{Implement details for different methods.} The implementations for dimension reduction methods and downstream prediction of $Y$ are similar to those in the simulation studies. PoolPCA pools the training environments and extracts principal components. AJIVE estimates a shared low-dimensional structure across environments. IHD is our invariance-heterogeneity decomposition and uses only the invariant representation for prediction. ATLAS w/$r_\tti=0$ is an ablation that uses auxiliary labels to align the latent factors without first separating invariant and heterogeneous components. ATLAS is the full proposed method. After each method produces its low-dimensional representation, we fit the same downstream prediction model for $Y$ using the training environments and use the fitted score to evaluate test-year AUC.

The dimension of the low-dimensional projection for PoolPCA and the $L_1$ hyper-parameter for Lasso are all selected by a validation set. For AJIVE and our methods, we set the environment-specific rank to $r^{(e)}\equiv 64$. For AJIVE, the dimension of the shared subspace is selected by the same validation criterion. For IHD and ATLAS, we set $\lambda_{\mathtt{IHD}}=0.01$ in the invariance-heterogeneity decomposition in \cref{algo:dp}. 

\bibliographystyle{apalike2}    
\bibliography{main.bbl}




\end{document}